\documentclass[11pt,reqno]{amsproc}

\usepackage{amsmath,amsfonts,amssymb,amsthm}
\usepackage[abbrev,lite,nobysame]{amsrefs}
\usepackage{graphics,graphicx}
\usepackage[usenames,dvipsnames]{color}
\usepackage{times}
\usepackage{bbm}
\usepackage[margin=1in]{geometry}
\usepackage[colorlinks=true, pdfstartview=FitV, linkcolor=blue, citecolor=blue, urlcolor=blue]{hyperref}

\usepackage{mathtools}

\mathtoolsset{showonlyrefs=true}

\makeatletter

\renewcommand\subsection{\@startsection{subsection}{2}%
\normalparindent{.5\linespacing\@plus.7\linespacing}{-.5em}
{\normalfont\bfseries}}

\renewcommand\subsubsection{\@startsection{subsubsection}{3}%
\normalparindent{.5\linespacing\@plus.7\linespacing}{-.5em}
{\normalfont\bfseries}}

\newcommand{\bullpar}[1]{\vspace{.5em}\noindent $\bullet$ \normalfont {\bfseries #1.}}

\newcommand{\diampar}[1]{\vspace{.5em}\noindent $\diamond$ \normalfont {\itshape #1.}}

\def\@tocline#1#2#3#4#5#6#7{\relax
  \ifnum #1>\c@tocdepth 
  \else
    \par \addpenalty\@secpenalty\addvspace{#2}%
    \begingroup \hyphenpenalty\@M
    \@ifempty{#4}{%
      \@tempdima\csname r@tocindent\number#1\endcsname\relax
    }{%
      \@tempdima#4\relax
    }%
    \parindent\z@ \leftskip#3\relax \advance\leftskip\@tempdima\relax
    \rightskip\@pnumwidth plus4em \parfillskip-\@pnumwidth
    #5\leavevmode\hskip-\@tempdima
      \ifcase #1
       \or\or \hskip 1em \or \hskip 2em \else \hskip 3em \fi%
      #6\nobreak\relax
    \dotfill\hbox to\@pnumwidth{\@tocpagenum{#7}}\par
    \nobreak
    \endgroup
  \fi}
\makeatother

\newtheorem{theorem}{Theorem}
\newtheorem{proposition}{Proposition}[section]
\newtheorem{lemma}[proposition]{Lemma}

\theoremstyle{definition}

\newtheorem{remark}[proposition]{Remark}

\numberwithin{equation}{section}

\newcommand\eps{\varepsilon}
\newcommand\e{{\rm e}}
\newcommand\dd{{\rm d}}
\newcommand\ddt{{\frac{\dd}{\dd t}}}
\def\Re{{\rm Re}}

\def\l {\langle}
\def\r {\rangle}
\newcommand\de{{\partial}}

\newcommand{\norm}[1]{\left\lVert #1 \right\rVert}
\newcommand{\jap}[1]{\left\langle #1 \right\rangle}
\newcommand{\abs}[1]{\left\lvert #1 \right\rvert}

\newcommand{\NN}{\mathbb{N}}
\newcommand{\ZZ}{\mathbb{Z}}
\newcommand\TT {{\mathbb T}}
\newcommand\RR {{\mathbb R}}

\newcommand\bu{{\boldsymbol u}}

\newcommand\bI{{\boldsymbol I}}
\newcommand\bU{{\boldsymbol U}}
\newcommand\bX{{\boldsymbol X}}

\newcommand\tA{{\widetilde A}}

\newcommand\cE{{\mathcal E}}
\newcommand\cF{{\mathcal F}}
\newcommand\cG{{\mathcal G}}
\newcommand\cH{{\mathcal H}}
\newcommand\cR{{\mathcal R}}

\newcommand\cV{{\mathcal V}}

\newcommand\hOmega{{\widehat \Omega}}

\newcommand\hTheta{{\widehat \Theta}}
\newcommand\hh{{\widehat h}}

\newcommand\sDelta{{\Lambda}}
\newcommand\sOmega{{\Omega^\star}}
\newcommand\sQ{{ Q^\star}}
\newcommand\sPsi{{ \Psi^\star}}
\newcommand\sTheta{{ \Theta^\star}}
\newcommand\sZ{{ Z^\star}}
\newcommand\hsOmega{{\widehat \Omega^\star}}

\newcommand\hsPsi{{\widehat \Psi^\star}}
\newcommand\hsTheta{{\widehat \Theta^\star}}

\renewcommand{\div}{\mathrm{div}}
\def\R {\mathbb{R}}

\def\de{{\partial}}

\def\N{{\mathbb{N}}}

\def\G{\mathcal{G}}
\newcommand{\skli}{\sum_{k,\ell}\int_{\RR^2}}
\newcommand{\TS}{2\max\{\sqrt{|\eta|},\sqrt{|\xi|}\}}
\newcommand{\TL}{2\min\{|\eta|,|\xi|\}}

\makeatletter
\def\@xfootnote[#1]{%
	\protected@xdef\@thefnmark{#1}%
	\@footnotemark\@footnotetext}
\makeatother

\title[Inviscid damping and instability in Boussinesq]{Nonlinear inviscid damping and shear-buoyancy instability\\ in the  two-dimensional Boussinesq equations}
\author[J. Bedrossian]{Jacob Bedrossian}

\address{Department of Mathematics, University of Maryland, College Park, MD 20742, USA}
\email{jacob@math.umd.edu}
\author[R. Bianchini]{Roberta Bianchini}
\address{IAC, Consiglio Nazionale delle Ricerche, 00185 Rome, Italy}
\email{r.bianchini@iac.cnr.it}

\author[M. Coti Zelati]{Michele Coti Zelati}
\author[M. Dolce]{Michele Dolce}
\address{Department of Mathematics, Imperial College London, London, SW7 2AZ, UK}
\email{m.coti-zelati@imperial.ac.uk}
\email{m.dolce@imperial.ac.uk}

\keywords{Inviscid damping, shear-buoyancy instability, stratified fluids, Boussinesq approximation, mixing}

\subjclass[2000]{35Q35, 76F10}

\begin{document}

\begin{abstract}
We investigate the long-time properties of the two-dimensional inviscid Boussinesq equations near a stably stratified Couette flow,
for  an initial Gevrey perturbation of size $\varepsilon$. 
Under the classical Miles-Howard stability condition on the Richardson number, we prove that the system experiences a shear-buoyancy instability: the density variation and velocity undergo an $O(t^{-1/2})$ inviscid damping while the vorticity and density gradient grow as $O(t^{1/2})$.
The result holds at least until the natural, nonlinear timescale $t \approx \varepsilon^{-2}$.
Notice that the density behaves very differently from a passive scalar, as can be seen from the inviscid damping and slower gradient growth. 
The proof relies on several ingredients: (A) a suitable symmetrization that makes the linear terms amenable to energy methods and takes into account the classical Miles-Howard spectral stability condition; (B) a variation of the Fourier time-dependent energy method introduced for the inviscid, homogeneous Couette flow problem developed on a toy model adapted to the Boussinesq equations, i.e. tracking the potential nonlinear echo chains in the symmetrized variables despite the vorticity growth. 
\end{abstract}

\maketitle

\setcounter{tocdepth}{1}
\tableofcontents
\section{Introduction}
This article is concerned with the long-time dynamics of a 2D incompressible and non-homogeneous fluid under the \emph{Boussinesq approximation} near a stably stratified Couette flow in the infinite periodic strip $\mathbb{T}\times \mathbbm{R}$. The background density profile is taken to be affine, thus we study the simple equilibrium $$\bu_E=(y,0), \qquad \rho_E=\bar{\rho}- by,$$
where $\bar \rho > 0 $ is the averaged constant density and  $b>0$ is a fixed constant. 
Given a density perturbation $\rho$, we define the modified density perturbation $\theta=\rho/b$. The 2D Euler-Boussinesq system for perturbations around the steady state $\bu_E,\rho_E$ reads 
\begin{align}\label{eq:Boussinesq_vel}
	\begin{cases}
		\de_t\bu +y\de_x \bu +(u^y,0) +\nabla p =- \theta (0,\beta^2)-(\bu\cdot \nabla) \bu, \\
		\de_t\theta+y\de_x \theta = u^y-\bu \cdot \nabla \theta,
	\end{cases}\qquad  (x,y)\in \mathbb{T}\times\mathbb{R}, \ t\geq 0,
\end{align}
where $\bu=(u^x, u^y)$ is the perturbation velocity field, $p$ is the pressure and $\beta=\sqrt{-\rho_E'\mathfrak{g}/\bar{\rho}}$, with $\mathfrak{g}$ being the gravitational constant. The parameter $\beta$ is the \textit{Brunt-V\"ais\"al\"a} frequency, which is the characteristic frequency of the oscillations of vertically displaced fluid parcels, and hence provides a measure of the strength of the buoyancy force. 
We write the system \eqref{eq:Boussinesq_vel} in vorticity-stream formulation as  
\begin{align}\label{eq:Boussinesq}
	\begin{cases}
		\de_t\omega +y\de_x \omega =-\beta^2\de_x \theta-\bu\cdot \nabla \omega,\\
		\de_t\theta+y\de_x \theta = \de_x\psi-\bu \cdot \nabla \theta, \\
		\bu=\nabla^\perp \psi, \qquad \Delta \psi=\omega,
	\end{cases} \qquad  (x,y)\in \mathbb{T}\times\mathbb{R}, \ t\geq 0,
\end{align}
with $\omega=\nabla^\perp \cdot \bu $, where we denote $\nabla^\perp=(-\de_y,\de_x)$.

Density stratification is a common feature of geophysical flows;  under appropriate averaging, most of the Earth’s ocean is well-approximated as an incompressible, stably stratified fluid so that its dynamics are well described by fluctuations around a mean background density profile which increases with depth (\emph{stable} stratification profile,  \cites{dauxois2021,majda2003intro, cushman2011}). 
The system  \eqref{eq:Boussinesq} under investigation models a stably stratified fluid with the additional \emph{Boussinesq assumption}, according to which density is assumed constant except when it directly causes buoyancy forces \cites{long1965bouss, lannes2020}.
The Boussinesq system gained the interest of the mathematical community thanks to its wide range of applications, especially in oceanography \cites{dauxois2021, majda2003intro}, and many mathematical works have been dedicated to it \cites{masmoudi20bouss, doering2018long, elgindi2015bouss, Charve04,CN97,Chae06,DP09,DP11,AVSWXY16,TWZZ20}. It also holds mathematical interest through a connection with the 3D axisymmetric Euler equations for homogeneous fluids \cite{majdabertozzi2002}, where the term multiplied by $\beta^2$ in \eqref{eq:Boussinesq} plays the role of the vortex stretching.

Perturbations of the equilibrium state in a stably stratified fluid induce two related mechanisms as consequences of gravity's restoring effect  (Archimedes' principle) and the shearing transport of the equilibrium. The first one is a buoyancy force generated by the pressure gradient of the stable stratification as a response to gravity, which pushes the higher density fluid downwards. The second one is vorticity production due to the horizontal density gradient, which acts as a source term in the vorticity equation of \eqref{eq:Boussinesq}.
These two mechanisms are coupled even at the \emph{linear} level, in such a way that their interplay may lead to an overall instability of the system \cite{majda2003intro}. Note that gravity's restoring effect also manifests itself as radiation of internal gravity waves, whose propagation is supported by stably stratified fluids as a remarkable feature: understanding the dynamics of internal waves is in fact of crucial importance to many geophysical applications \cites{Akylas2003, dauxois2021, majda2003intro}. The non-trivial underlying dynamics have been observed 
in laboratory experiments \cites{browand1973laboratory,koop_browand_1979} and investigated in the physics literature \cites{Hartman75, case1960stabilityat, farrell1993transient}.

In the case of the Couette flow, linear stability is ensured by the so-called Miles-Howard criterion \cites{miles1961stability,howard1961note}, which requires the Richardson number $\mathrm{Ri}=\beta^2$  to be greater than $1/4$. Under this condition,  precise quantitative estimates can be extrapolated from the linear dynamics. In 1975, Hartman \cite{Hartman75} observed an enstrophy Lyapunov instability with a growth of $O(t^{1/2})$,  despite the fact that the velocity field undergoes an $O(t^{-1/2})$ time decay. This phenomenon persists for more general, stably stratified fluids without the Boussinesq approximation, as showed by Case \cite{case1960stabilityat}. A decay of $O(t^{-1/2})$ for both the velocity and the density has been proved  rigorously for the Couette flow in \cite{YL18} and extended to shears near Couette in \cite{BCZD20}. In addition, a vorticity and density gradient growth with rate $O(t^{1/2})$, which confirms the observation of  \cite{Hartman75}, has been rigorously proved in \cite{BCZD20}. Due to the nature and the origin of such growth, we will refer to it as a \emph{shear-buoyancy instability}. It is worth pointing out that the enstrophy growth is in striking contrast with the 2D homogeneous and inviscid Couette flow, which is Lyapunov stable in the enstrophy norm for both the linear and nonlinear problem (in fact, the enstrophy of the perturbation is conserved in both).    

The decay of the velocity field of the perturbation, called \textit{inviscid damping}, is due to the mixing of vorticity and is a key dynamical property of shear flows and vortices. 
This was first noticed by Orr \cite{Orr07} and later studied by Case and Diki\u{\i} \cites{Diki61,case1960stabilityat} for a 2D homogeneous fluid, where the velocity field decays as  $O(t^{-1})$. In particular, inviscid damping occurs when the shear transfers enstrophy to high frequencies. This is a fundamental mechanism of inviscid fluids, intimately connected with the stability of coherent structures \cites{EulerVortex,SM95} and the theory of 2D turbulence \cite{BraccoEtAl2000}. 
Its first mathematically rigorous study in the \emph{full} 2D homogeneous Euler equations was carried out in \cite{BM15} for the Couette flow.
It bears remarking that due to transient unmixing effects, the Couette flow is in fact Lyapunov \emph{unstable} in the kinetic energy norm (a consequence of \cite{BM15}).

\subsection{The main result}
The purpose of this article is to provide the first rigorous study on the long-time dynamics of the Couette flow for the 2D inviscid Boussinesq 
system  \eqref{eq:Boussinesq_vel}.  We prove that the nonlinear system undergoes a 
shear-buoyancy instability and nonetheless the velocity field experiences nonlinear inviscid damping, confirming that the linear dynamic extends to the nonlinear setting at least on a natural timescale $O(\eps^{-2})$.  
Fix $s>1/2$ and define the Gevrey norm of class $1/s$ as
\begin{equation}\label{eq:gevnorm}
	\norm{f}^2_{\G^{\lambda}}=\sum_{k\in \mathbb{Z}}\int_\R \e^{2\lambda(|k|+|\eta|)^s}|\widehat{f}_k(\eta)|^2\dd \eta.
\end{equation}
Moreover, set
 \begin{equation}\label{eq:zeromodenotation}
f_0(y)=\frac{1}{2\pi}\int_{\mathbb{T}}f(x,y)\dd x, \qquad f_{\neq}=f-f_0.
\end{equation}
The main result of this article is stated in the next theorem.

\begin{theorem}\label{thm:mainT}
	Let $\beta>1/2$. For all $1/2< s\leq 1$, $\lambda_0>\lambda'>0$ there exist  $\delta=\delta (\beta,\lambda_0,s)\in(0, 1)$ and  $\eps_0=\eps_0 (\beta,\lambda_0,s)\in(0,\delta)$  such that the following holds true: let $\eps\leq \eps_0$  and $\omega^{in}, \theta^{in}$ be mean-free initial data satisfying  
	\begin{equation}\label{data:reg}
   	\norm{u^{in}}_{L^2} + \norm{\omega^{in}}_{\G^{\lambda_0}}+\norm{\theta^{in}}_{\G^{\lambda_0}}\leq \eps.
   \end{equation}
   Then, if we define the shift $\Phi(t,y) = \int_0^t u_0^x(\tau,y) \dd\tau$, for  all $0\leq t \leq \delta^2\eps^{-2}$ we have 
   \begin{align}
   	\label{bd:longtimestab}
   	\norm{u_0^x(t)}_{\G^{\lambda'}}+\norm{\theta_0(t)}_{\G^{\lambda'}} & \lesssim \eps, \\ 
   	\norm{\omega(t,x + ty +  \Phi(t,y),y)}_{\G^{\lambda'}}+ \jap{t} \norm{\theta_{\neq}(t,x + ty +  \Phi(t,y),y)}_{\G^{\lambda'}} & \lesssim \eps \jap{t}^{1/2}.\label{bd:longtimestab2}
   \end{align}
As a consequence, the velocity field and the modified density satisfy 
\begin{align}
	\label{bd:invxmain}
	\norm{u^x_{\neq}(t)}_{L^2}+\norm{\theta_{\neq}(t)}_{L^2}&\lesssim \frac{\eps }{\jap{t}^\frac12},\\
	\label{bd:invymain}\|u^y_{\neq}(t)\|_{L^2}&\lesssim \frac{\eps }{\jap{t}^\frac32}. 
\end{align}
Moreover, there exists $ K = K(\beta,\lambda_0,\lambda',s) > 0$ such that if 
\begin{equation}
	\norm{\omega^{in}_{\neq}}_{H^{-1}}+\norm{\theta^{in}_{\neq}}_{L^2}\geq K \eps \delta , 
\end{equation}
then 
\begin{equation}
	\label{bd:instamain}
	\norm{\omega_{\neq}(t)}_{L^2}+\norm{\nabla \theta_{\neq}(t)}_{L^2} \approx \eps\jap{t}^\frac12.
\end{equation}
for all $0 < t < \delta^2\eps^{-2}$.
	\end{theorem}

The above result describes the long-time dynamics of the Boussinesq system  \eqref{eq:Boussinesq} in the perturbative regime near the linearly stratified Couette flow,
and it is the first of its kind describing such behavior in a fully inviscid coupled system  which has both wave propagation and phase mixing. 
The works \cites{masmoudi20bouss,ZillingerBouss,DWXZ20} study nonlinear systems with both phase mixing and wave propagation, but these problems all contain dissipative effects, whereas the works \cites{BCZD20,YL18} are all linear. 
The inviscid damping due to vorticity mixing is encoded in \eqref{bd:invxmain}-\eqref{bd:invymain}.

One of the main novelties here is the quantification of the shear-buoyancy instability  given by \eqref{bd:instamain}. 
The linearized dynamics of \eqref{eq:Boussinesq} predict exactly the decay rates  \eqref{bd:invxmain}-\eqref{bd:invymain} and the instability  \eqref{bd:instamain} for all times \cites{BCZD20,Hartman75,YL18},
see also Theorem \ref{thm:Couette-lin} below.
Therefore, in a nonlinear perturbative regime as the one studied here, the time-scale $O(\eps^{-2})$ appears naturally. As another manifestation of the instability,
the rates in \eqref{bd:longtimestab}-\eqref{bd:invymain} are $\jap{t}^{1/2}$ slower compared to the constant density case studied in \cite{BM15}. This is due to creation 
of vorticity in the perturbation by interaction with the density stratification. 

The proof of Theorem \ref{thm:mainT}, described in detail in the next Section \ref{02Ideas}, truly uses the specific linear coupling of $\omega$ and $\theta$
 via a suitable  symmetrization of the unknowns. Specifically, the scaled density $\theta$ is not simply transported by the Couette flow, as this would imply 
 a growth rate of order $\jap{t}$ for $\nabla \theta_{\neq}$, rather than the $\jap{t}^{1/2}$ appearing in \eqref{bd:instamain}.

The need of an infinite regularity (Gevrey) space is by-now classical in phase mixing problems, both for Landau damping in plasma physics 
\cites{MV11,faou2016landau,DFG18,bedrossian2016landau,GNR20} and for inviscid damping in fluid mechanics \cites{BM15,ionescu20,ionescu2020nonlinear,masmoudi2020nonlinear,ionescu2019axi}. This is strictly connected with loss of derivatives as a price to pay for the control of  \emph{transient growths} or \emph{echoes}: further discussions on this aspect can be found in the course of the paper. 
The  regularity requirement on the initial data \eqref{data:reg} is the same as in the constant density case \cites{BM15,ionescu2020nonlinear,masmoudi2020nonlinear}, and it is likely
to be sharp \cite{DM2018}.  
This can be heuristically understood by a toy model that estimates the worst possible growths due to the nonlinear interactions. Despite predicting the same total loss of regularity as in the constant density case, the model is tailored specifically to the Boussinesq system and displays crucial differences in terms of the regularity imbalance between resonant and non-resonant modes  (see Section \ref{sub:nonlingrowth}). The picture may change with the addition of thermal diffusivity and/or viscosity. 
When viscosity is added in the vorticity equation, the Gevrey index can be relaxed to $s=1/3$ as in \cite{masmoudi20bouss}, while 
 when also diffusivity is present in the density equation one can work in Sobolev regularity \cites{ZillingerBouss,DWXZ20}.

The restriction of the parameter $\beta$ in Theorem \ref{thm:mainT}  is sharply consistent with   the classical Miles-Howard criterion for linear spectral stability \cites{miles1961stability,howard1961note} mentioned above. 
The role of this restriction is very explicit 
in the coercivity of the main energy functional used to prove Theorem \ref{thm:mainT}, but also implicitly appears in many of the constants hidden by the symbol $\lesssim$, which blow up as 
$\beta  \to 1/2$. The linear dynamics when $\beta\leq 1/2$ was studied in \cite{YL18}. In this case  the vorticity grows with faster rates 
(and the density decays with slower rates). Reproducing the results of \cite{YL18}  by means of an energy method like the one used in \cite{BCZD20} could lead to further insight at the nonlinear level as well.

Finally, we do not expect that the linear dynamics persist to leading order after times $ O(\eps^{-2})$, but rather that a secondary instability engages to carry the solution a fully nonlinear regime. 
Specifically, after this time, we expect that mixing creates large adverse vertical density gradients, resulting in an overturning instability.
There are some analogies between Theorem \ref{thm:mainT} and the work on subcritical transition in 3D Couette \cite{BGM15II}: 
both study a spectrally stable problem with an algebraic instability and show that the only way to trigger a  secondary instability is through the underlying  destabilizing mechanism  (at least in Gevrey class). 
The possible secondary instability, the 3D case, and the case of stably stratified fluids without the Boussinesq approximation will be studied in future work.

\subsection{Organization of the article}
Section \ref{02Ideas} describes the main ideas needed for the proof of Theorem \ref{thm:mainT}, including the symmetrized variables, the weighted energy functionals and 
the fundamental bootstrap Proposition \ref{prop:bootimpr}. In Section \ref{03Proof} we prove Theorem \ref{thm:mainT} assuming Proposition \ref{prop:bootimpr}.
The rest of the article is dedicated to the proof Proposition \ref{prop:bootimpr}. The construction of the time-dependent Gevrey weights is carried out in Section \ref{sec:weightW}, 
while Section \ref{sec:ellest} is dedicated to the proof of the elliptic estimates crucial to control the nonlinear terms. 
The heart of the article is contained in Section \ref{sec:mainEn}, where we prove the energy estimate on the symmetric variables. These require direct bounds on the vorticity
and the gradient of the density, which are carried out in Section \ref{sec:naturalenergy}. Finally, Section \ref{sec:zero} contains the control of the nonlinear change of coordinates.

\subsection{Notations and conventions}
We use the notation $f\lesssim g$ when there exists a constant $C>0$, independent of the
parameters of interest, such that $f\leq C g$. Similarly, $f\approx g$ means that here exists $C>0$ such that $C^{-1} g\leq f\leq C g$. We will denote by $c$ a generic positive
constant smaller than 1. 

Given a vector $(k,\eta)$, we indicate by $|k,\eta|=|k|+|\eta|$   its norm. We will use the symbol  $\jap{a}=\sqrt{1+|a|^2}$ for either scalars or vectors. Given a normed space $X$, its norm is denoted by
$\|\cdot\|_{X}$, omitting the subscript when $X=L^2$. We recall also that \eqref{eq:gevnorm} and \eqref{eq:zeromodenotation} are used throughout the article.

For a Schwartz function $f=f(z,v):\TT\times\RR\to\RR$, we define the Fourier transform as
\begin{align}
\widehat{f}_k(\eta)=\frac{1}{2\pi} \int_{\TT\times\RR} \e^{-ikz-i\eta}f(z,v)\dd z\dd v, \qquad (k,\eta)\in\ZZ\times\RR.
\end{align}
The Littlewood-Paley dyadic decomposition is defined as follows: we take $\phi\in C^\infty_0(\RR)$ be such that $\phi(\eta)=1$ for $|\eta| \le \frac 12$, $\phi=0$ for $|\eta| \ge \frac 34$ 
and set $\tilde \phi(\eta)=\phi(\eta)-\phi(\eta/2)$. Then
\begin{align}
1=\phi(\eta) + \sum_{M \in 2^{\NN}} \tilde \phi (M^{-1} \eta).
\end{align}
In this way, $\tilde \phi_M(\eta):=\tilde \phi(M^{-1}\eta)$ is supported in $\frac M 2 \le |\eta| \le \frac{3M}{2}$. For a function $g=g(v)\in L^2(\RR)$, we define
\begin{align}
g=g_\frac 12 + \sum_{M \in 2^\NN} g_M := \phi(|\de_v|)g+ \sum_{M \in 2^\NN} \tilde \phi_M(|\de_v|) g.
\end{align}
We also use the notation
\begin{align}
g_{<M} = g_\frac 12 + \sum_{K \in 2^{\NN}: K<M} g_K,\qquad g_{\sim M} =  \sum_{M \in 2^{\NN}; \frac{M}{C} <K<CM} g_K,
\end{align}
for some constant $C$ independent of $M$. The paraproduct decomposition will then be denoted as follows 
\begin{align}
fg &= \sum_{M \in 2^{\NN}: M \ge 8} f_{<M/8} g_M + \sum_{M \in 2^{\NN}: M \ge 8} g_{<M/8} f_M + \sum_{M,M' \in 2^{\NN}: M/8 \le M' \le 8M} g_{M'} f_M. \label{eq:paraprod} 
\end{align}

\section{Outline of the proof}\label{02Ideas}

In this section, we outline the proof of Theorem \ref{thm:mainT}. There are a number of different ideas that go into it, some arising from the inviscid damping result for the homogeneous problem \cite{BM15}, others arising from the study of the linearized problem \cite{BCZD20}, and others which are new and specific to this nonlinear problem. 

\subsection{Change of coordinates}
Given the incompressibility of the flow, we know $\boldsymbol{u}_0=(u_0^x,0)$, which implies
\begin{equation*}
	\boldsymbol{u}\cdot \nabla=u_0^x\de_x +\boldsymbol{u}_{\neq}\cdot \nabla.
\end{equation*}
Due to the inviscid damping, we expect the non-zero $x$-frequencies to decay and hence it is natural to treat the last term as a perturbation. 
However, there is no decay mechanism for $u_0^x$ and so this term could be treated perturbatively on an $O(\eps^{-1})$ time-scale at most. 
To deal with this difficulty, \cite{BM15} introduced a change of coordinates that depends on $u_0^x(t)$, and for the same reason, we use the same coordinate change.   
We briefly recall it here; see \cite{BM15} for more details. 
Define 
\begin{align}\label{def:coord}
v=y+\frac{1}{t}\int_0^t u_0^x(s,y)\dd s,\qquad z=x-vt. 
\end{align}
Provided that $u_0^x$ is sufficiently small, this coordinate change can be inverted; we assume this is the case for now.  
The corresponding unknowns written in the new variables (writing $x = x(t,z,v)$, $y = y(t,v)$) are given by 
\begin{align}\label{def:omega-psi-moving}
\Omega(t,z,v)=\omega(t,x,y),\qquad \Theta(t,z,v)=\theta(t,x,y),\qquad \Psi(t,z,v)=\psi(t,x,y).
\end{align}
In this way we obtain (we write the change of variables only for $\Omega$ but similar relations hold for the other functions)
\begin{align}\label{eq:deriv}
\de_t\omega=\de_t\Omega +\dot{z} \de_z\Omega +\dot{v}\de_v\Omega, \qquad \de_x\omega=\de_z\Omega,\qquad \de_y\omega=v' (\de_v-t\de_z)\Omega
\end{align}
where
\begin{alignat}{2}
\dot{z}&:=\de_tz=-y - u^x_0,  &\qquad\qquad   \dot{v}&:=\de_t v=\frac{1}{t}\left[u^x_0-\frac{1}{t}\int_0^t u_0^x(s,y)\dd s\right], \label{eq:prop1}\\	
v'&:=\de_y v =1+\frac{1}{t}\int_0^t \omega_0(s,y)\dd s, &\qquad\qquad   v''&:=\de_{yy} v=\frac{1}{t}\int_0^t\de_y \omega_0(s,y)\dd s\label{eq:prop2}.	
\end{alignat}
The Biot-Savart law also gets transformed as
\begin{align}\label{def:deltat}
\Delta_t\Psi=\Omega, \qquad \Delta_t:=\de_{zz}+(v')^2(\de_v-t\de_z)^2+v'' (\de_v-t\de_z).
\end{align}
In the new coordinates, the original system \eqref{eq:Boussinesq} is now expressed as
\begin{align}\label{eq:BoussinesqMove}
	\begin{cases}
		\de_t\Omega =-\beta^2\de_z \Theta-\bU\cdot \nabla \Omega,\\
		\de_t\Theta= \de_z\Psi-\bU \cdot \nabla \Theta, \\
		\bU=(0,\dot{v})+v'\nabla^\perp \Psi_{\neq}, \qquad \Delta_t \Psi=\Omega,
	\end{cases}
\end{align}
where $\nabla=\nabla_{z,v}$. Notice that the zero mode in $z$ and in $x$ are the same, and therefore we use the same symbol 
as in \eqref{eq:zeromodenotation} to denote the projection of $\Psi$ off the zero mode in $z$.

To control the coordinate system itself, as in \cite{BM15}, we need to introduce more variables, specifically 
\begin{equation}\label{def:h-H}
	h(t,v)=v'(t,y)-1, \qquad \cH(t,v)=\de_y\dot{v}(t,y)=\frac{1}{t}\left(\omega_0(t,y)-\frac{1}{t}\int_0^t \omega_0(s,y)\dd s\right).
\end{equation}
Notice that from \eqref{eq:prop2} we have
\begin{align}
	v'(t,y)-1=\frac{1}{t}\int_0^t \omega_0(s,y)\dd s\qquad \Rightarrow \qquad \cH(t,v)=\frac{1}{t}\left(\Omega_0(t,v)-h(t,v)\right).
\end{align}
Since from \eqref{eq:prop2} one has that $\de_t (t(v'-1))=\omega_0$, we have from \eqref{eq:deriv} that
\begin{align}
	\label{eq:h} (\de_t+ \dot{v}\de_v)(th)=\Omega_0 \qquad \Rightarrow \qquad \de_th + \dot{v}\de_v h =\cH.
\end{align}
Taking the $z$ average of the first equation in \eqref{eq:BoussinesqMove}, we similarly derive 
\begin{align}
		\label{eq:H} \de_t\cH&=-\frac2t\cH-\dot{v}\de_v\cH-\frac{v'}{t}\left(\nabla^\perp \Psi_{\neq} \cdot \nabla \Omega_{\neq}\right)_0.
\end{align}
Finally, we also record the equation satisfied by $\dot{v}$ \emph{in the $(t,v)$ coordinates}, namely
\begin{align}
	\label{eq:vdot} \de_t\dot{v}&=-\frac2t\dot{v}-\dot{v}\de_v\dot{v}-\frac{v'}{t}\left(\nabla^\perp \Psi_{\neq} \cdot \nabla U^x_{\neq}\right)_0,
	\qquad  U^x(t,z,v):=u^x(t,x,y).
\end{align}

\subsection{The linearized dynamics: symmetric variables}\label{sec:lindyn}
Unlike \cite{BM15}, the linear dynamics are non-trivial. 
The linearized dynamics associated with \eqref{eq:BoussinesqMove} are best understood by passing to Fourier variables $(z,v)\mapsto(k,\eta)$. 
Since at the linear level we have $v=y$, the differential operators in these coordinates read 
\begin{align}\label{eq:DeltaL}
	\nabla_L:=(\de_z,\de_v-t\de_z), \qquad \Delta_L:=\de_{zz}+(\de_v-t\de_z)^2.
\end{align}
We denote the symbols associated to $-\Delta_L$ as
\begin{align}\label{def:p}
	p_k(t,\eta)=k^2+(\eta-k t)^2, \qquad \de_tp_k(t,\eta)=-2k(\eta-k t).
\end{align}
The explicit dependence on $t$ of the above quantities will often be omitted. The linearized equations are obtained from  \eqref{eq:BoussinesqMove}  by neglecting all nonlinear terms 
including the one arising from the nonlinear change of coordinate (hence $\Delta_t$ is formally replaced by $\Delta_L$). 
On the Fourier side, they take the form
\begin{align}\label{eq:BoussinesqMoveFlin}
		\de_t\hOmega =-i\beta^2k \widehat{\Theta},\qquad \de_t\hTheta= -\dfrac{ik}{p}\hOmega,
\end{align}
and in particular they decouple in $k$ and $\eta$. 
While the zero-mode is clearly conserved, the nonzero modes exhibit an interesting behaviour which has been studied in the applied mathematics literature since the 1950s; we refer to \cite{Hartman75} for a detailed literature review. In \cite{Hartman75}, the system \eqref{eq:BoussinesqMoveFlin} is investigated by a method involving hypergeometric functions, made mathematically rigorous and precise in \cite{YL18}. For our purposes, it is more convenient to recall the energy method used in \cite{BCZD20}, originally introduced to deal with the linear stability of the Couette flow in a compressible fluid \cite{antonelli2021linear}. The idea is to symmetrize the system \eqref{eq:BoussinesqMoveFlin} via time-dependent Fourier multipliers and use an energy functional for the new auxiliary variables. 
Compared to  \cite{BCZD20}, we slightly change the symmetrized variables by modifying powers of $k$, defining them here as
\begin{equation}\label{eq:Z1Z2couette}
Z_k(t,\eta):= \left(\left(p/k^2\right)^{-\frac{1}{4}} \hOmega\right)_k(t,\eta), \qquad Q_k(t,\eta):=\left(\left(p/k^2\right)^\frac{1}{4}i  k\beta\hTheta\right)_{k}(t,\eta).
\end{equation}
for which  \eqref{eq:BoussinesqMoveFlin} takes the particularly amenable form
\begin{align}
		\de_tZ =-\dfrac{1}{4}\dfrac{\de_tp}{p}Z-|k|\beta p^{-\frac{1}{2}}Q,\qquad \de_tQ=\dfrac{1}{4}\dfrac{\de_tp}{p}Q+|k|\beta p^{-\frac{1}{2}}Z.\label{eq:Qlin}
\end{align}
The presence of the $k^2$ factors in \eqref{eq:Z1Z2couette} only modifies the linearized equations by changing $k$ to $\abs{k}$, however, the adjustment to the definition of $Z,Q$ will be important to treat the nonlinear problem later. 
Define the following energy functional point-wise in frequency 
\begin{align}\label{def:pointwise-functional-Couette}
E(t)&=\frac12\left[|Z(t)|^2+|Q(t)|^2+\frac{1}{2\beta} \Re  \left(\frac{\de_tp} {|k|p^{\frac12}} Z(t) \overline{Q(t)}\right)\right].
\end{align}
Since, $|\de_t p/(kp^{1/2})|\leq 2$, the energy functional is coercive for $\beta>1/2$ with
\begin{align}\label{eq:coercive-pointwise}
\frac12\left(1-\frac{1}{2\beta}\right)\left[|Z|^2+|Q|^2\right](t)\leq E(t)\leq\frac12\left(1+\frac{1}{2\beta}\right)\left[|Z|^2+|Q|^2\right](t),
\end{align}
and can be shown to satisfy
\begin{align}\label{eq:energqueq}
\ddt E=\frac{1}{4\beta}\de_t\left(\frac{\de_tp} {|k|p^{\frac12}}\right)\Re\left(Z\overline{Q}\right).
\end{align}
Since
\begin{equation}
	\label{eq:dtdtp}
	\de_t\left(\frac{\de_tp} {|k|p^{\frac12}}\right)=-2\de_t \left(\frac{\frac{\eta}{k}-t}{(1+(\frac{\eta}{k}-t)^2)^{\frac12}}\right)=\frac{2}{(1+(\frac{\eta}{k}-t)^2)^{\frac32}}
\end{equation}
is positive, from the coercivity of $E$ we arrive at
\begin{equation}
	-\frac{1}{2(1+2\beta)}\de_t\left(\frac{\de_tp} {|k|p^{\frac12}}\right) E\leq \ddt E\leq \frac{1}{2(1-2\beta)}\de_t\left(\frac{\de_tp} {|k|p^{\frac12}}\right) E, 
\end{equation}
and hence 
\begin{align}
E(t)\approx E(0), \qquad \forall t\geq 0.
\end{align}
The precise linear dynamics can therefore be described by the following theorem.
\begin{theorem}[\cite{BCZD20}*{Theorem 1.1}]\label{thm:Couette-lin}
For any $k\neq0$ and $\eta\in \RR$, the solution to the linearized equations \eqref{eq:BoussinesqMoveFlin} satisfies the uniform bounds
\begin{align}\label{eq:estimatecouette}
 |p^{-\frac14}\widehat\Omega(t)|^2+|p^{\frac14}\widehat\Theta(t)|^2
 \approx  |(k^2+\eta^2)^{-\frac{1}{4}} \widehat\Omega(0)|^2 + |(k^2+\eta^2)^\frac{1}{4} \widehat\Theta(0)|^2,
\end{align}
for every $t\geq 0$, where the constant hidden in \eqref{eq:estimatecouette} blows up as $\beta\to 1/2$. 
\end{theorem}

The above Theorem \ref{thm:Couette-lin} implies a (linear) inviscid damping for the velocity \emph{and} the density fluctuation of \eqref{eq:Boussinesq} 
	\begin{align}
	\norm{\theta_{\neq}(t)}_{L^2}+\norm{u_{\neq}^x(t)}_{L^2}\lesssim&\frac{1}{\l t\r^{\frac12}}\left[\norm{\omega_{\neq}^{in}}_{H^{-1}_xL^2_y}+\norm{\theta_{\neq}^{in}}_{L^2_xH^1_y}\right], \qquad \forall t\geq0,\label{eq:main-lin1}\\
	\norm{u^y(t)}_{L^2}\lesssim& \frac{1}{\langle t \rangle^{\frac32}}\left[\norm{\omega_{\neq}^{in}}_{H^{-1}_xH^1_y}+\norm{\theta_{\neq}^{in}}_{L^2_xH^2_y}\right],  \qquad \forall t\geq0.\label{eq:main-lin2}
	\end{align}
In stark contrast to the homogeneous Couette flow \cite{BM15}, in light of the lower bound in \eqref{eq:estimatecouette}, the system undergoes a Lyapunov \emph{instability}
\begin{equation}
\label{bd:Lyomega}
\norm{\omega_{\neq}(t)}_{L^2}+\norm{\nabla\theta_{\neq}(t)}_{L^2}\gtrsim\langle t \rangle^{\frac12}\left[\norm{\omega_{\neq}^{in}}_{L^2_x H^{-1}_y}+\norm{\theta_{\neq}^{in}}_{H^1_x L^2_y}\right], \qquad \forall t\geq0.
\end{equation}
This can be considered the reason why the decay rates in \eqref{eq:main-lin1}-\eqref{eq:main-lin2} are slower, by a factor of $t^{1/2}$, compared to those that
can be obtained in the constant density case \cites{Orr1907,WZZ18,Zillinger17,WZZ19}.

\begin{remark}
Note that the requirement for coercivity, $\beta > 1/2$, is the same as the Miles-Howard condition for the spectral stability of stratified shear flows put forward in  \cites{miles1961stability,howard1961note}. This is not coincidental, as the original spectral stability proof is also an energy-based argument, albeit of a different type.
\end{remark}

\subsection{The nonlinear growth mechanism}\label{sub:nonlingrowth}
The full nonlinear system corresponding to \eqref{eq:BoussinesqMove} in the $(Z,Q)$ variables \eqref{eq:Z1Z2couette}
reads
\begin{align}
		\de_tZ &=-\dfrac{1}{4}\dfrac{\de_tp}{p}Z-|k|\beta p^{-\frac{1}{2}}Q-\left(\frac{p}{k^2}\right)^{-\frac{1}{4}} \cF(\bU\cdot \nabla \Omega),\label{eq:Z}\\
		\de_tQ&=\dfrac{1}{4}\dfrac{\de_tp}{p}Q+|k|\beta p^{-\frac{1}{2}}Z-\beta |k|^{\frac32} p^{-\frac{3}{4}}\cF\left((\Delta_t-\Delta_L)\Psi\right) -\left(\frac{p}{k^2}\right)^\frac{1}{4}k\cF(\bU \cdot \nabla (i \beta\Theta)),\label{eq:Q}
\end{align}
where in \eqref{eq:Q} we have isolated the linear part identical to that in \eqref{eq:Qlin} and we have used the identity
 $\Delta_L\Psi=\Omega-(\Delta_t-\Delta_L)\Psi$.

In inviscid damping around Couette flow, the unmixing of enstrophy causes transient growth of the velocity, called the \emph{Orr mechanism}, with analogous transient growth effects in other phase mixing problems \cites{MV11,Becho20,BM15,Orr1907}. 
As discussed in \cites{BM15,MV11,DM2018,Becho20,VMW98}, when studying nonlinear phase mixing problems, a key effect to look for are ``echoes'', wherein well-mixed enstrophy, through nonlinear interactions, transfers back to frequencies which will be un-mixed at a future time and hence cause growth in the velocity field by the Orr mechanism, possibly repeating the process into a chain of nonlinear oscillations. 
Echo chains were captured in experiments for plasmas modeled by the Vlasov equations \cite{Landau68} and in plasmas modeled by the 2D Euler equations near a vortex in \cite{EulerVortex}.  
This can be considered a kind of ``resonance'' associated with the linear transient growth mechanism that appears at the second iterate of linearization (i.e. if one linearizes around the linear dynamics) \cites{BM15,MV11,DM2018,Becho20,VMW98,deng2019echo}. It is the primary reason that proving nonlinear inviscid damping (or Landau damping) type results is challenging and why such results have generally required very high regularity; see e.g.  \cites{BM15,BGM15II,MV11,ionescu20,masmoudi20bouss,masmoudi2020nonlinear,ionescu2019axi,ionescu2020nonlinear}. In order to account for the echo resonances, \cite{BM15} introduced a time-dependent Fourier multiplier method which builds a norm carefully designed to match exactly the worst-case estimates of these resonances. 
In fluid mechanics, there does not yet exist an alternative to this method for studying nonlinear inviscid damping problems. 
In order to adapt these ideas to the system \eqref{eq:Z}-\eqref{eq:Q}, we need to derive a ``toy model'' that captures the worst possible growth 
caused by nonlinear interactions. 
As we will see, though we proceed in the spirit of \cite{BM15} and \cite{BGM15II}, the toy model has significant differences with previous works. 

As $\Omega$ is the unstable quantity, the worst possible nonlinear term appears in the equation for $Z$. 
As in \cite{BM15}, we derive a formal toy model by a paraproduct decomposition of the nonlinearity, which can be thought of as a secondary linearization of the evolution of the high frequencies about the low frequency linear dynamics. 
Ignoring the linear terms (as the corresponding linear semigroup is bounded) and assuming that $\Delta_t^{-1}$ can be well-approximated by $\Delta_L^{-1}$,
we want to write a good model for the nonlinear interactions of the scalar equation
\begin{align}
	\de_tZ\approx (p/k^2)^{-\frac14}\mathcal{F}\left(\nabla^\perp \Delta_L^{-1}\Omega\cdot \nabla \Omega\right).
\end{align}
At this point we want to extract the contribution which is most ``dangerous''.
Through a standard paraproduct decomposition of the nonlinearity, namely 
\begin{equation}
	\label{eq:NLhom}
	\nabla^\perp \Delta_L^{-1}\Omega \cdot \nabla\Omega=(\nabla^\perp \Delta_L^{-1}\Omega)_{Hi} \cdot (\nabla\Omega)_{lo}+(\nabla^\perp \Delta_L^{-1}\Omega)_{lo} \cdot (\nabla\Omega)_{Hi}+(\nabla^\perp \Delta_L^{-1}\Omega)_{Hi} \cdot (\nabla\Omega)_{Hi},
\end{equation}
we (formally) neglect the \emph{Hi-Hi} term (with enough regularity, it is easily controlled), the
 \emph{lo-Hi} term (by trading regularity for time-decay of the velocity), and the  \emph{Hi-lo} term involving $\de_z\Delta_{L}^{-1}\Omega$ (which is 
 uniformly bounded in $k$ and $\eta$). 
Writing down explicitly the only remaining, and most difficult-looking, term involving $(\de_v \Delta_L^{-1}\Omega)_{Hi}  (\de_z\Omega)_{lo}$ and using the definition of $Z$ in \eqref{eq:Z1Z2couette}, we end up with
\begin{align}
	\label{eq:firstapp}
	\de_tZ_k(\eta)\approx \frac{|k|^{\frac12}}{p_k(\eta)^{\frac14}}\sum_{\ell\neq0} \int_\R \frac{\xi(k-\ell)}{|\ell|^{\frac12}p_\ell(\xi)^{\frac34}}Z_\ell(\xi)_{Hi}\frac{p_{k-\ell}(\eta-\xi)^{\frac14}}{|k-\ell|^{\frac12}}Z_{k-\ell}(\eta-\xi)_{lo}\dd\xi.
\end{align}
Since the variable $\eta-\xi$ is at low frequency, we further approximate the equation above by considering $\eta=\xi$, which give us the infinite system of ODEs 
\begin{align}
	\label{eq:firstapp2}
	\de_tZ_k(\eta)\approx \frac{|k|^{\frac12}}{p_k(\eta)^{\frac14}}\sum_{\ell\neq0} \frac{\eta(k-\ell)}{|\ell|^{\frac12}|k-\ell|^{\frac12}p_\ell(\eta)^{\frac34}}Z_\ell(\eta)_{Hi}p_{k-\ell}(0)^{\frac14}Z_{k-\ell}(0)_{lo},
\end{align}
where $\eta$ is to be considered as a fixed parameter.  
As observed in \cite{BM15}, the dangerous scenario is when $\eta k^{-2}>1$ and there is a 
\textit{high-to-low} cascade in which the $k$ mode has a strong effect at time $\eta/k$ that excites the $k-1$ mode, which itself 
has a strong effect at time $\eta/(k-1)$ that excites the $k-2$ mode and so on. This physically corresponds to an echo chain \cites{BM15,deng2019echo,DM2018,Becho20,EulerVortex}.  
Therefore, we focus near one critical time  $\eta/k$ on a time interval of length roughly $\eta/k^2$, so that $\eta/(k-1)$ is not critical, 
and consider the interaction between the mode $k$ and a nearby mode $\ell$  with $\ell\neq k$. Calling
$f_R=Z_k(t,\eta)$ and $f_{NR}=Z_{k-1}(t,\eta)$ the resonant and non-resonant dominant modes, respectively,  keeping only the leading order terms and taking absolute values, we obtain the coupled toy system 
\begin{align}
	\label{eq:toyfR0}\de_t f_R& \approx \eps  \frac{\eta}{p_{k}(\eta)^{\frac14}p_{k-1}(\eta)^\frac34}p_1(0)^\frac14 f_{NR},\\
	\de_t \label{eq:toyfNR0}f_{NR}& \approx \eps  \frac{\eta}{p_{k-1}(\eta)^{\frac14}p_{k}(\eta)^\frac34}p_{-1}(0)^\frac14 f_R,
\end{align}
where we also included that $Z_1(0)\approx \eps$, as we are assuming linear dynamics to leading order. Since 
$\eta/(k-1)$ is not critical and $k^2\leq \eta$, we have that $p_{k-1}(\eta)\approx (\eta/k)^2$. Therefore, using \eqref{def:p} and that
$\eps p_1(0)^\frac14 \approx \eps t^\frac{1}{2} \lesssim 1$ for our purposes, the toy model that we finally consider is
\begin{align}
	\label{eq:toyfR1}
	\de_t f_R&=  \left(\frac{k^2}{\eta}\right)^\frac12\frac{1}{(1+|t-\eta/k|^2)^\frac14}f_{NR},\\
	\label{eq:toyfNR1}\de_t f_{NR}&= \bigg(\frac{\eta}{k^2}\bigg)^\frac12\frac{1}{(1+|t-\eta/k|^2)^\frac34}f_{R}.
\end{align}
In Section \ref{03Proof} we construct a weight based on this model, that takes
into account a regularity imbalance between the resonant and non-resonant modes. 
Some remarks are in order.

\begin{remark}[On the Gevrey-2 regularity]
	As in previous works \cite{BM15}, one can deduce that the maximal possible growth for $f_R$ and $f_{NR}$ is of order $(\eta/k^2)^{C}$ for some constant $1<C<16$.  
	If this growth accumulates for all the frequencies $k=1,\dots,\lfloor \sqrt{\eta}\rfloor$,  Stirling's formula implies a growth of order $\exp(\sqrt{\eta})$. This is consistent with the loss of Gevrey-2 regularity in the inviscid, homogeneous case \cites{BM15,DM2018}.	

\end{remark}

\begin{remark}[Comparison with previous works]
	The toy model introduced in \cite{BM15} for the vorticity of the homogeneous Euler equations reads 
	\begin{align}
		\label{eq:toyfR1BM}
		\de_t \widetilde{f}_R&= \frac{k^2}{\eta}\widetilde{f}_{NR}, \quad \de_t \widetilde{f}_{NR}= \frac{\eta}{k^2}\frac{1}{1+|t-\eta/k|^2}\widetilde{f}_{R}.
	\end{align}
	The key practical difference among the two toy models is that for $(\widetilde{f}_R,\widetilde{f}_{NR})$ the power $1/2$ in \eqref{eq:toyfR1}-\eqref{eq:toyfNR1} is replaced by the power $1$. This implies that while both models predict the same regularity loss (Gevrey-2) we are going to impose a smaller regularity imbalance between resonant and non-resonant modes, and hence are going to lose less derivatives when measuring the effect of non-resonant modes on resonant modes and gain less when measuring the effect of resonant modes on non-resonant modes.
	 	
	The toy model used to build the norm in \cite{masmoudi20bouss} for the Boussinesq equations with viscosity (but not thermal diffusivity) near Couette flow is more significantly different. However, the derivation and use of the model depend crucially on the presence of viscosity. 
\end{remark}
\begin{remark}[On the time-scale $O(\eps^{-2})$]
	In the construction of the toy model \eqref{eq:toyfR1}-\eqref{eq:toyfNR1} we have used in a crucial way that $\eps t^{\frac12} \lesssim 1$ (analogous to the way the lift-up effect time-scale of $O(\eps^{-1})$ dictated the toy model in \cite{BGM15II}). 
	For times $t \geq \eps^{-2}$, the $\eps t^{1/2}$ in \eqref{eq:toyfR1}-\eqref{eq:toyfNR1}, due to the structure of the system, would lead to an exponential growth for $f_R$ and $f_{NR}$ which could not be controlled in any regularity class. 
	Instead, for the toy model \eqref{eq:toyfR1}-\eqref{eq:toyfNR1} we show that the growth is at most polynomial, see Proposition \ref{prop:key} and accumulates only to Gevrey-2 losses. 
\end{remark}

\subsection{Weights and energy functionals}\label{sub:weenbo}

Ultimately, the main step in the proof of Theorem \ref{thm:mainT} is to obtain the following uniform-in-$t$ estimate for $t < \delta^2 \eps^{-2}$
\begin{align}
	\sup_{0 < t< \delta^2 \eps^{-2}} \norm{Z(t)}_{\cG^{\lambda}} + \norm{Q(t)}_{\cG^{\lambda}} \lesssim \eps, 
\end{align}
for some $\lambda > 0$ and $s > 1/2$. Using this estimate (and suitable estimates on the change of coordinates), it is not too difficult to complete the proof of Theorem \ref{thm:mainT}; see Section \ref{03Proof}. 
However, we cannot obtain such an estimate directly, instead, there are several additional ingredients that are required. 
In order to obtain uniform bounds in the presence of the linear term, we need to estimate $Z,Q$ with an energy based on the linear analysis of \cite{BCZD20}; we will call this energy functional $E_L$. 
However, the $Z,Q$ variables break the natural energy structure of the quadratic transport nonlinearities in the vorticity and density equations. 
Hence, we need to couple the $Z,Q$ estimate with an energy which estimates $\Omega$ and $\nabla_L \Theta$ directly; this estimate is at the highest level of regularity, so controls the highest frequencies, but due to the linear instability, necessarily grows in time. We denote this energy functional $E_{n}$.  
Both of these estimates are in turn coupled to an energy that controls the coordinate system which can be considered to be an estimate on the evolving shear $u_0^x$; this energy is denoted $E_v$. 
The control of these three energies (and associated time-integrated quantities), form the main bootstrap argument, detailed below in Section \ref{sec:bootstrap}. 
This general scheme: with one energy that grows in time adapted to control high frequencies ($E_{n}$), one energy at lower regularity adapted to obtain the optimal decay estimates ($E_{L}$) is common in studying perturbative quasilinear problems such as scattering in dispersive PDEs (see e.g. \cites{GMS12,IP15} and related references) and also when studying Landau damping in kinetic theory (e.g. \cites{faou2016landau,bedrossian2016landau}). 
Here there is the additional complication of requiring estimates on a coordinate system that is coupled to the other unknowns; this same additional complication arises in certain dispersive PDEs (see e.g. \cites{IP15,lindblad2010global,oh2016global}). 

The key idea to the Fourier multiplier method of \cite{BM15} is to introduce time dependent Fourier multipliers that allow us to capture the possible growth mechanisms by suitably weakening the norms in a time and frequency dependent way. All three energies, $E_L,E_{n},E_v$ are based on such Fourier multiplier norms. 
As in other methods based on time-dependent norms, weakening the norm generates \textit{artificial damping} terms in the equations that can be used to absorb terms in the energy method. 
We remark that this method is reminiscent also of Alinhac's ghost weight method \cite{alinhac2001null}, however (aside from being on the Fourier side), this method necessitates the norm losing a significant amount of regularity in an anisotropic way, as time proceeds, which significantly increases the complexity. 
The main weight is defined as a time-dependent Fourier multiplier $A=A_k(t,\eta)$ of the form 
\begin{equation}
	\label{def:A}
	A_k(t,\eta)=\jap{k,\eta}^\sigma \e^{\lambda(t)|k,\eta|^s}\left(m^{-1}J\right)_{k}(t,\eta),
\end{equation}
where $\sigma >16$ is a fixed constant, $\lambda(t)$ is the bulk Gevrey regularity index and $m$, $J$ are suitable Fourier multipliers to be defined in the sequel. The function $\lambda(t)$ is assumed to satisfy 
\begin{align}
	\label{def:dotlambda}
	\dot{\lambda}(t)&=-\frac{\delta_\lambda}{\jap{t}^{2q}}(1+\lambda(t)), \quad t>1,\\
	\lambda(t)&=\frac34 \lambda_0+\frac14\lambda', \quad t\leq 1,
\end{align}
where $\lambda_0,\lambda'$ are those of Theorem \ref{thm:mainT}, $\delta_\lambda \approx \lambda_0-\lambda'$ is a small parameter to ensure 
that $2\lambda(t)>\lambda_0+\lambda'$, and $1/2<q\leq 1/4+s/2$ is a parameter chosen by the proof. The function $\lambda(t)$ restricts the radius of regularity, by a finite amount, in a continuous way.  
As discussed in \cite{BM15}, it suffices to consider the case $s$ close to $1/2$ as higher regularities can be treated by adding an additional factor $\exp\left(\gamma(t) | k,\eta |^{p}\right)$ for any $s < p \leq 1$ which would play little role in the energy estimates that follow. 

\subsubsection{The linear weight $m$}

As we have seen in Section \ref{sec:lindyn}, the error term appearing in \eqref{eq:energqueq} can be integrated in 
time at any fixed frequency $(k,\eta)$. However, the nonlinear case cannot simply be treated point-wise in $(k,\eta)$,
and we are forced to introduce the bounded Fourier multiplier 
\begin{equation}
	\label{def:m}
	m_k(t,\eta)=\begin{cases} \exp\left(C_\beta \arctan(t-\eta/k)\right), \quad &\text{for }k\neq0 , \\
		1\quad &\text{for }k=0,
	\end{cases}
	\qquad
C_{\beta}=\frac{1}{2\beta-1} .	
\end{equation}
Notice that
\begin{equation}\label{def:detm}
	\de_t m_k= \dfrac{C_\beta}{1+\left(t-\eta/k\right)^2} m_k. 
\end{equation}
Such multiplier creates the artificial damping term (see \eqref{def:dtEmintro} below) that controls the analogous of the linear error term in \eqref{eq:energqueq}. This multiplier (or similar ones) have been used previously in e.g. \cites{Zillinger17,liss2020sobolev,bedrossian2018sobolev}.

\subsubsection{The nonlinear weight $J$}
The remaining multiplier to complete the definition of $A$ in \eqref{def:A} is given by
 \begin{equation}\label{eq:Jweight}
 	J_k(t,\eta)=\frac{\e^{\mu|\eta|^\frac12}}{w_k(t,\eta)}+\e^{\mu|k|^\frac12},
 \end{equation}
where $1<\mu<23$. The weight $w_k$ is extremely important and is constructed using the toy model 
\eqref{eq:toyfR1}-\eqref{eq:toyfNR1} in Section \ref{sec:weightW}. In particular, it is used to distinguish between the resonant and 
non-resonant behavior of the system (see Section \ref{sub:weigthprop} for all the properties of $w_k$). 
For the moment we can think of it as a correction to the main exponential factors of $J$
and $A$ that mimics the behavior of the toy model \eqref{eq:toyfR1}-\eqref{eq:toyfNR1} near the critical times $t=\eta/\ell$. 
Most importantly, it assigns more regularity to the ``resonant'' frequencies $(\ell,\eta)$ than to the ``non-resonant'' frequencies $(\ell',\eta)$.  
It is analogous to the corresponding weight in \cite{BM15}, however, the $w_k$ weight here is different from the one in \cite{BM15} due to the different toy model. 
Finally, for technical reasons it is convenient to define 
\begin{align}
	\label{def:tildeJ}
	\widetilde{J}_k(t,\eta)&=\e^{\mu |\eta|^\frac12}w_k^{-1}(t,\eta),
	\end{align}
and the corresponding weight
\begin{align}
	\label{def:Atilde}	\widetilde{A}_k(t,\eta)&=\jap{k,\eta}^\sigma \e^{\lambda(t)|k,\eta|^s}(m^{-1}\widetilde{J})_k(t,\eta).
\end{align}
Notice that
\begin{equation}\label{eq:Adt}
\de_tA=\dot{\lambda}(t)|k,\eta|^sA-\frac{\de_tw}{w}\tA-\frac{\de_tm}{m}A.
\end{equation}

\subsubsection{The coordinate system weight $A^v$.}
It turns out that the energy functional that controls coordinate system needs to use a stronger (compare to $A$ above) weight
of a similar form
\begin{align}
	\label{def:ARvintro} A^{v}(t,\eta)&=\jap{\eta}^\sigma \e^{\lambda(t)|\eta|^s}J^{v}(t,\eta),\qquad 
	 	J^v(t,\eta)=\frac{\e^{\mu|\eta|^\frac12}}{w^v(t,\eta)}.
\end{align}
Here, $J^{v}$ plays a similar role as $J$ in \eqref{eq:Jweight}, and is  defined in terms of a weight $w^v$ below in \eqref{def:wv}. 
However, $J^v$ is constructed from the toy model for the homogeneous 2D Euler equations \eqref{eq:toyfR1BM} used in \cite{BM15}, making it essentially the same as the weight used in \cite{BM15}. 
Due to the different toy model being used here, this implies we will be propagating a relatively large amount of additional regularity on the coordinate system (relative to the $z$-dependent unknowns). 
This additional regularity is crucial to closing the estimates below. 

\subsubsection{The linear-type energy functional $E_L$}
The energy functional that will permit uniform bounds is designed by taking inspiration from the linearized energy \eqref{def:pointwise-functional-Couette}. It is defined as
\begin{align}\label{eq:Energy}
E_L(t)=\frac12\left[\|AZ\|^2+\|AQ\|^2+\frac{1}{2\beta}   \left\l \frac{\de_tp}{|k| p^{\frac12} }AZ ,AQ\right\r\right].
\end{align}
Notice that, for $\beta>1/2$, we have the same coercivity bounds as in \eqref{eq:coercive-pointwise}. 
Through a careful computation of its time derivative (carried out in Section \ref{app:energyIDZQ}) and using the
definition of $A$ (see \eqref{def:A}, \eqref{def:detm} and \eqref{eq:Jweight}) we arrive at an inequality of the type
\begin{align}
	\label{def:dtEmintro}
	\ddt E_L + \left(1-\frac{1}{2\beta}\right)\sum_{j\in\{\lambda,w,m\}}(G_j[Z]+G_j[Q]) \leq L^{Z,Q}+NL^{Z,Q}+\cE^{\div}+\cE^{\Delta_t}.
\end{align}
Here we denote
\begin{equation}
	\label{def:Gw}
	G_w[f]=\norm{\sqrt{\frac{\de_t w}{w}}\sqrt{A\tA} f}^2,\qquad G_m[f]=\norm{\sqrt{\frac{\de_t m}{m}}Af}^2,\qquad G_\lambda[f]=-\dot{\lambda}(t)\norm{|\nabla|^{\frac{s}{2}}Af}^2.
\end{equation}
 The terms above are also called \textit{Cauchy-Kovalevskaya} terms since they come from the weakening of the norms caused by the Fourier multipliers. Those are \textit{good} terms since they have a definite sign and can be used to control the other error terms in the identity \eqref{def:dtEmintro}, which we divide as follows: $L^{Z,Q}$ is a linear error term, analogous to that in  \eqref{eq:energqueq} and defined precisely in \eqref{eq:linear-error}, $NL^{Z,Q}$ contains the main nonlinear
 errors that come from the transport structure of the equations (see \eqref{def:NL}), while $\mathcal{E}^{\div}$ and $\mathcal{E}^{\Delta_t}$ are simpler error terms to treat that
 arise as a consequence of the nonlinear change of coordinates \eqref{def:coord}, and are defined in \eqref{eq:div-error} and \eqref{eq:lasterror}, respectively.

\subsubsection{The coordinate change energy functional $E_v$}
The control on the change of coordinates, described by equations \eqref{eq:h}-\eqref{eq:H}
is achieved via the energy functionals
\begin{equation}
	\label{def:Ev}
	 E_v(t)= \frac12\left(\jap{t}^{2+2s}\norm{\frac{A}{\jap{\de_v}^{s}}\mathcal{H}(t)}^2+ \frac{1}{C_1}\big(\norm{A^{v} h(t)}^2 +\jap{t}^{-2s}\norm{A^{v}|\de_v|^s h(t)}^2\big) \right)
\end{equation}
and
\begin{equation}
	\label{def:normvdot}
	\jap{t}^{4}\norm{\dot{v}(t)}_{\G^{\lambda(t),\sigma-6}}^2,
\end{equation}
where the weight $A^{v}$ is defined in \eqref{def:ARvintro} and $C_1=C_1(\beta,\lambda_0,s)>1$ is a constant chosen in the proof. In Section \ref{sec:zero} we derive the energy inequalities for the two functionals above, where it will be more convenient to treat each term in $E_v$ separately. Due to the presence of multiplier $A^v$ in the definition of $E_v$, when computing the time derivative of $A^v$ we get as good terms 
\begin{align}
	\label{def:GwR}
	G_w^{v}[f]&=\norm{\sqrt{\frac{\de_t w^v}{w^v}}A^{v}f}^2, \qquad G_\lambda^{v}[f]=-\dot{\lambda}(t)\norm{|\de_v|^{\frac{s}{2}}A^{v}f}^2.
\end{align}
\begin{remark}
	The structure of the energy functionals for the change of coordinates is heavily inspired by \cite{BM15}, indeed, the coordinate change and the associated equations \eqref{eq:h}-\eqref{eq:vdot} are exactly the same. 
	However, the control we have on the quantities under study is significantly different. 
	First, we need an additional, even higher regularity control on $h$, namely the last term in $E_v$.
	Moreover, while $A^v$ is essentially the weight $A^R$ used in \cite{BM15}, the norm $A$ is not the same, and so we have a significantly different regularity gap between the estimate on $\mathcal{H}$ and those on $h$.  
	For this reason, we cannot rely completely on the proofs given in \cite{BM15}. 
		\end{remark}

\subsubsection{The nonlinear-type energy functional $E_{n}$}
To control high frequencies  we also need a direct control on the vorticity $\Omega$ and the gradient of the density $\nabla_L\Theta$ that is consistent with the linear prediction.  
We therefore derive the differential equation for the \textit{natural} energy 
\begin{equation}\label{def:Omega}
	E_n(t)=\frac{1}{2}\left[\norm{A\Omega}^2+\beta^2\norm{A\nabla_L\Theta}^2\right],
\end{equation}
which satisfies the following inequality (see Section \ref{app:energyIDVD})
\begin{equation}\label{def:dtEnintro}
	\ddt E_n+\sum_{j\in\{\lambda,w,m\}}\left(G_j[\Omega]+\beta^2 G_j[\nabla_L\Theta]\right) \le \frac12 \left\l \frac{\de_tp}{|k|p^\frac12}AQ,AQ\right\r+NL^{\Omega,\Theta}+\widetilde{\mathcal{E}}^{\div}+\widetilde{\mathcal{E}}^{\Delta_t}.
\end{equation}
Note that the structure is very similar to \eqref{def:dtEmintro}, with the good Cauchy-Kovalevskaya terms, a linear error, a nonlinear error and the errors due to the change of coordinates.
Crucially, the linear term involves precisely $Q$ and a bounded multiplier, thanks to the definition \eqref{eq:Z1Z2couette}, and it will be treated
thanks to the energy $E_L$ above for $Z$ and $Q$. All the error terms involved in this energy balance are analyzed in Section \ref{sec:naturalenergy}.

\begin{remark}[On the necessity of the symmetric variables]
	The results stated in Theorem \ref{thm:mainT} can also be obtained by $E_n(t)\lesssim \eps^2 t$ and coordinate system estimates. 
	The reason why it is necessary to control the symmetric variables $Z,Q$  is the  term  
	 \begin{equation}
	 	\label{def:LQEn}
	 	\frac12 \left\l \frac{\de_tp}{|k|p^\frac12}AQ,AQ\right\r=\frac{1}{2}\jap{\frac{\de_t p}{p}Ap^\frac12\widehat{\Theta},Ap^\frac12\widehat{\Theta}},
	 \end{equation}
since $\de_tp$ does not have a definite sign and is positive  for $t>\eta/k$. The weight $w$ cannot be used to control this term for all the frequencies, and any other weight on $\Omega$ and $\nabla_L\Theta$ would have to be of order $p^{-1/4}$ (at best), hence leading back to the symmetric variables. Instead, $Z,Q$ have a nice structure at the linear level and, once we have a control on them, the bound on \eqref{def:LQEn} is immediate. 
Error terms containing $\de_tp/p$ have been previously handled in the literature for  linear inviscid \cite{BCZD20} or  viscous problems (e.g. \cites{BGM15I,BGM15II,BGM15III,liss2020sobolev}). 
	\end{remark}

\subsection{The bootstrap proposition}\label{sec:bootstrap}
To control the energy functionals $E_L,E_v,E_n$ and \eqref{def:normvdot}, we rely on a continuity argument. Hence, we first state the local well-posedness result. 
We omit the proof since it follows by standard reasoning for 2D Euler in Gevrey spaces (see \cite{BM15} for discussion). 
\begin{proposition} \label{lem:loc}
For all $s > 1/2$, $\lambda_0 > 0$, $ \eps > 0$, there exists an $\eps'_0 > 0$ such that  for every $ \eps' < \eps_0'$, 
\begin{align}	
\norm{\omega^{in}}_{\G^{\lambda_0}}+\norm{\theta^{in}}_{\G^{\lambda_0}}\leq \eps'
\end{align}
implies that 
\begin{align}
	& \sup_{0 \le t \le 2}E_L(t) + \frac12\left(1-\frac{1}{2\beta}\right)\sum_{j\in\{\lambda,w\}}\int_0^2 G_{j}[Z]+G_j[Q] \dd\tau\leq 2 \eps^2,\\
	 & \sup_{0 \le t \le 2} E_n(t)+\sum_{j\in\{\lambda,w\}} \int_0^2  G_{j}[\Omega]+\beta^2G_{j}[\nabla_L\Theta] \dd\tau\leq  \frac{16\beta}{2\beta-1}\eps^2, \\
	 &\sup_{0 \le t \le 2} E_v(t)+\sum_{j\in\{\lambda,w\}}  \int_0^2\jap{\tau}^{2+2s}G_j[\jap{\de_v}^{-s}\mathcal{H}]+\frac{1}{C_1}\left(G^{v}_j[h]+\jap{\tau}^{-2s}G^{v}_j[|\de_v|^s h]\right) \dd\tau\leq 32\eps^2,\\
	 &\sup_{0 \le t \le 2} \norm{\dot{v}}_{\G^{\lambda(t),\sigma-6}}\leq 2\eps.
\end{align}
\end{proposition}

By a standard approximation argument, we may assume the quantities on the right-hand side take values continuously in time (see \cite{BM15}). 
We now introduce the bootstrap hypotheses, which hold for some $T_\star > 1$ by continuity and Proposition \ref{lem:loc}.

\medskip 
 \textsc{Bootstrap hypotheses.}\textit{
	For $1 \leq t \leq T_\star$ the following holds 
		\begin{align}
			\label{boot:E} \tag{\bf{H1}}& E_L(t) + \frac12\left(1-\frac{1}{2\beta}\right)\sum_{j\in\{\lambda,w\}}\int_1^t G_{j}[Z]+G_j[Q] \dd\tau\leq 8 \eps^2,\\
			\label{boot:En} \tag{\bf{H2}} &E_n(t)+\sum_{j\in\{\lambda,w\}} \int_1^t G_{j}[\Omega]+\beta^2G_{j}[\nabla_L\Theta] \dd\tau\leq  \frac{128\beta}{2\beta-1}\eps^2\jap{t}\\
		\label{boot:Ev}\tag{\bf{H3}} &E_v(t)+\sum_{j\in\{\lambda,w\}}  \int_1^t\jap{\tau}^{2+2s}G_j[\jap{\de_v}^{-s}\mathcal{H}]+\frac{1}{C_1}\left(G^{v}_j[h]+\jap{\tau}^{-2s}G^{v}_j[|\de_v|^s h]\right) \dd\tau\leq 128\eps^2 \jap{t},\\
					\label{boot:vdot} \tag{\bf{H4}} &\jap{t}^2\norm{\dot{v}}_{\G^{\lambda(t),\sigma-6}}\leq 16\eps,
			\end{align}
where $G_j[\cdot]$ are defined in \eqref{def:Gw} and $G^{v}_j[\cdot]$ in \eqref{def:GwR}.}

\medskip

By the local well-posedness Proposition \ref{lem:loc} and choosing $\eps$ sufficiently small, we know that for $t=1$ the bounds \eqref{boot:E}-\eqref{boot:vdot} hold with all the constants on the right-hand side divided by 4.  
Then, our goal is to prove the following proposition.
\begin{proposition}
	\label{prop:bootimpr}
	There exists $\eps_0\in (0,1/2)$, $C_0>2$ and $\delta\in (C_0^{-1},1/2)$ depending only on $\beta,s,\lambda,\lambda',\sigma$ such that $\eps_0$ is much smaller than $\delta$ with the following property. If $\eps<\eps_0$ and the bootstrap hypotheses \eqref{boot:E}-\eqref{boot:vdot} hold on $[1,\delta^2\eps^{-2}]$,  then for any $T^\star\in(1,\delta^2\eps^{-2}]$  and any $t\in[1,T^\star]$ we have
			\begin{align}
		\label{boot:impE}  \tag{\bf{B1}}& E_L(t) + \frac12\left(1-\frac{1}{2\beta}\right)\sum_{j\in\{\lambda,w\}}\int_1^t G_{j}[Z]+G_j[Q] \dd\tau\leq 4 \eps^2,\\
		\label{boot:impEn}  \tag{\bf{B2}} &E_n(t)+\sum_{j\in\{\lambda,w\}} \int_1^t G_{j}[\Omega]+\beta^2G_{j}[\nabla_L\Theta] \dd\tau\leq  \frac{64\beta}{2\beta-1}\eps^2\jap{t},\\
		\label{boot:impEv} \tag{\bf{B3}}&E_v(t)+\sum_{j\in\{\lambda,w\}}  \int_1^t\jap{\tau}^{2+2s}G_j[\jap{\de_v}^{-s}\mathcal{H}]+\frac{1}{C_1}\left(G^{v}_j[h]+\jap{\tau}^{-2s}G^{v}_j[|\de_v|^s h]\right) \dd\tau\leq 64\eps^2 \jap{t},\\
				\label{boot:impvdot}  \tag{\bf{B4}} &\jap{t}^{2}\norm{\dot{v}}_{\G^{\lambda(t),\sigma-6}}\leq 8\eps.
					\end{align}
In particular, this implies $T^\star =\delta^2\eps^{-2}$.
\end{proposition}
	\begin{remark}[On the smallness of the parameters]
	The factor $1-1/(2\beta)$ appearing in \eqref{boot:E} is related to the control of linear error terms. This immediately gives the following restriction 
	\begin{equation*}
		\delta = O(\beta-1/2).
	\end{equation*}
\end{remark}

\subsubsection{Immediate consequences of the bootstrap hypotheses}
The bounds in the bootstrap hypotheses imply a control on several other quantities that we need to prove Proposition \ref{prop:bootimpr} and Theorem \ref{thm:mainT}. We first show the bounds we have for the \textit{unweighted} variables in lower regularity spaces. In fact, we only need the following bounds to prove the main Theorem \ref{thm:mainT}.
\begin{lemma}
	\label{lemma:consboot}
	Under the bootstrap hypothesis, the following inequalities holds 
	\begin{align}
		\label{bd:OmTh}&\norm{\Omega}_{\G^{\lambda,\sigma}}+\norm{\nabla_L\Theta}_{\G^{\lambda,\sigma}}+\norm{h}_{\G^{\lambda,\sigma}}+\norm{1-(v')^2}_{\G^{\lambda,\sigma-4}}+\norm{v''}_{\G^{\lambda,\sigma-4}}\lesssim \eps \jap{t}^{\frac12},\\
	    \label{bd:Thneqlow}&\norm{\Theta_{\neq}}_{\G^{\lambda,\sigma-1}}\lesssim \frac{\eps}{\jap{t}^{\frac12}},\\
	    \label{bd:Psilow}&\norm{\Psi_{\neq}}_{\G^{\lambda,\sigma-3}}\lesssim \frac{\eps}{\jap{t}^{\frac32}}.
\end{align}
	\end{lemma}
The proof of the Lemma above is straightforward from the definition of $A$ and Proposition \ref{prop:lossyelliptic}, which shows that by paying Sobolev regularity, decay on $\Psi$ follows as in the case $\Delta_t = \Delta_L$; see \cite{BM15} for more detail. 

From the definition of $\Delta_t$ and $\bU$, we need to have a control also on $1-(v')^2$, $v''$ and $\dot{v}$ in the proper regularity classes. 
These coefficients are controlled by $h$ and bounds on $\dot{v}$ are recovered from $\mathcal{H}$. We collect the estimates in the following.
\begin{lemma}\label{lemma:bdcoeff}
	Under the bootstrap hypothesis, the following inequalities hold 
		\begin{align}
		\notag
&\norm{A^v(1-(v')^2)}^2+\norm{A^v\jap{\de_v}^{-1}v''}^2 \\ 		\label{bd:Avcoeff} & \quad +\sum_{j\in \{\lambda, w\}}\int_1^t G^v_j[(1-(v')^2)]+G^v_j[\jap{\de_v}^{-1}v'']\dd\tau\lesssim \eps^2 t,\\
				\notag &\jap{t}^{-2s}\left(\norm{A^v|\de_v|^s(1-(v')^2)}^2+\norm{A^v|\de_v|^s\jap{\de_v}^{-1}v''}^2\right)\\
				\label{bd:Avscoeff}& \quad +\sum_{j\in \{\lambda, w\}}\int_1^t \jap{\tau}^{-2s}(G^v_j[|\de_v|^s(1-(v')^2)]+G^v_j[|\de_v|^s\jap{\de_v}^{-1}v''])\dd\tau\lesssim \eps^2 t,\\
		   \label{bd:dvdotv}&\jap{t}^{2+2s}\norm{\frac{A}{\jap{\de_v}^s}\de_v \dot{v}}^2+\sum_{j\in\{\lambda,w\}}  \int_1^t\jap{\tau}^{2+2s}G_j[\jap{\de_v}^{-s}\de_v \dot{v}]\dd \tau \lesssim \eps^2 t,		\\
			  \label{bd:Glambdadotv}& \norm{\frac{A}{\jap{\de_v}^s} |\de_v|^{\frac{s}{2}}\dot{v}}\lesssim 
		\frac{\eps}{t^{2}}+\norm{\frac{A}{\jap{\de_v}^s}|\de_v|^{\frac{s}{2}} \mathcal{H}}, \\
	\label{bd:Gwdotv}	&\ G_w[\jap{\de_v}^{-s}\dot{v}]\lesssim  G_w[\jap{\de_v}^{-s}\mathcal{H}]. 
		\end{align}
	\end{lemma}
The proofs of the bounds \eqref{bd:Avcoeff}-\eqref{bd:Avscoeff} resemble the ones providing the analogous estimates of \cite{BM15}; they are sketched briefly in Section \ref{sec:zero}.

\section{Proof of the main theorem}\label{03Proof}
In this section we prove Theorem \ref{thm:mainT}, under the assumption that Proposition \ref{prop:bootimpr} holds. We need the bounds of Lemma \ref{lemma:consboot} (which follow directly from the bootstrap hypothesis \eqref{boot:E}-\eqref{boot:vdot}).
We remark again that the main part of this paper is the proof of Proposition \ref{prop:bootimpr}.

The first step of the proof of Theorem \ref{thm:mainT} is undoing the nonlinear coordinate transform. 
Instead of the change of coordinates \eqref{def:coord}-\eqref{def:omega-psi-moving}, we 
want to use $y$ as the second spatial coordinate and define
\begin{align}\label{def:omega-psi-tilde}
	\sOmega(t,z,y)=\omega(t,x,y),\qquad \sTheta(t,z,y)=\theta(t,x,y),\qquad \sPsi(t,z,y)=\psi(t,x,y),
\end{align}
where $z$ is still given by \eqref{def:coord}.
Define $\lambda_\infty = \lambda(\delta^2 \eps^{-2})$. 
By following the arguments in \cite{BM15}*{Section 2.3} using Lemma \ref{lemma:consboot}, we see that we may solve for $y$ in terms of $v$ and vice-versa and by a Gevrey composition lemma (see e.g. \cite{BM15}*{Lemma A.4}), we have from Lemma \ref{lemma:consboot} for all $1 < t < \delta^2 \eps^{-2}$
\begin{align}
t^{-1/2}\left(\norm{\sOmega(t)}_{\G^{\lambda_\infty'}}  + \norm{\de_y\Theta^\star_{0}(t)}_{\G^{\lambda_\infty'}} \right) + t^{1/2} \norm{\Theta^\star_{\neq}(t)}_{\G^{\lambda_\infty'}} + t^{3/2} \norm{\Psi^\star_{\neq}(t)}_{\G^{\lambda_\infty'}} \lesssim \eps, \label{ineq:sctrls}
\end{align}
for some $0 < \lambda_\infty' < \lambda_\infty$. 
The estimates on $\omega$, $\theta$, and $u_{\neq}$ stated in Theorem \ref{thm:mainT} now follow immediately. 

Taking the $x$-average of the momentum equations \eqref{eq:Boussinesq_vel} in the $y$ coordinates we have  (using that $u_0^y = 0$ by incompressibility)
\begin{align}
	\partial_t u_0^x(t,y) & = -\partial_y \int_{\mathbb T } u_{\neq}^{y}(t,x,y) u^{x}_{\neq}(t,x,y)  \dd x =- \de_y\int_{\mathbb T } U_{\neq}^{\star,y}(t,z,y) U^{\star,x}_{\neq}(t,z,y)  \dd z, 
\end{align}
where we denote $U^\ast(t,z,y) = u(t,x,y)$. 
Using \eqref{ineq:sctrls}, it then follows that 
\begin{align}
	\norm{\jap{\partial_y}^{-1} u_0^x(t) }_{\G^{\lambda_\infty'}} \lesssim \eps. 
\end{align}
This takes care of the uniform estimates on $u_0^x$ stated in Theorem \ref{thm:mainT}. The bound on $\theta_0$ follows similarly.

Next, we are interested in proving the instability result, which requires a more detailed analysis of the dynamics. 
First, we observe that \eqref{eq:BoussinesqMove} in the new Fourier variables becomes
\begin{align}\label{eq:BoussinesqMoveFtilde}
	\begin{cases}
		\de_t\hsOmega =-i\beta^2k \hsTheta-\cF(\nabla^\perp \Psi^\star_{\neq} \cdot \nabla \sOmega),\\
		\de_t\hsTheta= ik\hsPsi-\cF(\nabla^\perp  \Psi^\star_{\neq} \cdot \nabla \sTheta), \\
		\widehat{\sDelta_t \sPsi}=\hsOmega,
	\end{cases}
\end{align}
where
\begin{align}
\sDelta_t =\de_{zz} +(\de_y -v't\de_z)^2.
\end{align}
As before, it is convenient also to define
\begin{align}
\sDelta_L =\de_{zz} +(\de_y -t\de_z)^2,\qquad p_k(t,\zeta)= k^2+(\zeta-kt)^2,
\end{align}
namely, the analogues of  \eqref{eq:DeltaL}-\eqref{def:p} in the $(z,y)$ coordinates (with $\zeta$ the $y$-Fourier variable), and
the new auxiliary variables (as in \eqref{eq:Z1Z2couette})
\begin{equation}\label{eq:Z1Z2couettestar}
Z^\star_k(t,\eta):= \left(\left(p/k^2\right)^{-\frac{1}{4}} \hOmega\right)_k(t,\eta), \qquad Q^\star_k(t,\eta):=\left(\left(p/k^2\right)^\frac{1}{4}i  k\beta\hTheta\right)_{k}(t,\eta).
\end{equation}
Similarly to \eqref{eq:Z}-\eqref{eq:Q}, we have
\begin{align}
		\de_t\sZ &=-\dfrac{1}{4}\dfrac{\de_tp}{p}\sZ-|k|\beta p^{-\frac{1}{2}}\sQ-\left(\frac{p}{k^2}\right)^{-\frac{1}{4}} \cF(\nabla^\perp \Psi^\star_{\neq} \cdot \nabla \sOmega),\label{eq:Zstar}\\
		\de_t\sQ&=\dfrac{1}{4}\dfrac{\de_tp}{p}\sQ+|k|\beta p^{-\frac{1}{2}}\sZ-\beta |k|^{\frac32} p^{-\frac{3}{4}}\cF\left((\sDelta_t-\Lambda_L)\sPsi\right) \\
		&\quad -\beta\left(\frac{p}{k^2}\right)^\frac{1}{4}ik\cF(\de_z \Psi^\star_{\neq}  \de_y\Theta_0^\star+\nabla^{\perp}\Psi^\star_{\neq}\cdot \nabla \Theta^{\star}_{\neq}).\label{eq:Qstar}
\end{align}
Let us view the system as the vector ODE pointwise-in-frequency
\begin{align}\label{eq:modelODE}
\de_t \bX= L(t) \bX +F(t,\bX),
\end{align}
where $\bX=(\sZ,\sQ)$
and the linear part $L(t)$ is given by the time-dependent matrix
\begin{align}
L(t)= \begin{pmatrix}
\displaystyle-\frac14\frac{\de_tp}{p} & -|k|\beta p^{-\frac12}\\
|k|\beta p^{-\frac12} & \displaystyle\frac14\frac{\de_tp}{p}
\end{pmatrix}.
\end{align}
Calling $\Phi_L(t,\tau)$ the associated solution operator, we may re-write \eqref{eq:modelODE} as
\begin{align}\label{bd:insta0}
\bX(t)=\Phi_L(t,0)\bX(0) +\int_0^t \Phi_L(t,\tau) F(\tau,\bX(\tau)) \dd\tau.
\end{align}
A direct consequence of the linear estimate \eqref{eq:estimatecouette} is that, point-wise in $(t,k,\zeta)$, we have
\begin{align}\label{bd:insta1}
|\Phi_L(t,0)\bX_{\neq}(0)| \gtrsim  |\bX_{\neq}(0)|, \qquad |\Phi_L(t,\tau) F(\tau,\bX(\tau))_{\neq}|\lesssim |F(\tau,\bX(\tau))_{\neq}|,
\end{align}
for every $t\geq \tau\geq 0$.
Using the elementary inequality $\jap{a-b}\jap{a}\gtrsim \jap{b}$, since $p^{1/2}_k(t,\zeta)= |k|\jap{\zeta/k-t}$ we have $p^{1/2}_k(t,\zeta)\gtrsim |k|\jap{t}/\jap{\zeta/k}$. Thus
\begin{align}
\norm{\omega_{\neq}(t)}^2_{L^2_{x,y}}+\norm{\nabla\theta_{\neq}(t)}^2_{L^2_{x,y}}
&\approx\norm{\Omega^\star_{\neq}(t)}^2_{L^2_{z,y}}+\norm{(-\Lambda_L)^{\frac12}\Theta^\star_{\neq}(t)}^2_{L^2_{z,y}}\approx \sum_{k\neq0} \int_{\RR} \frac{p_k^\frac{1}{2} (t,\zeta)}{|k|} |\bX_k(t,\zeta)|^2 \dd\zeta\notag\\
&\gtrsim\l t \r  \sum_{k\neq0} \int_{\RR} \frac{1}{\l\zeta\r} |\bX_k(t,\zeta)|^2 \dd\zeta.
\end{align}
In light of \eqref{bd:insta0}-\eqref{bd:insta1}, we therefore have that there exists $ K' > 0$ such that
\begin{align}\label{eq:instalowerbdd1}
\norm{\omega_{\neq}(t)}^2_{L^2_{x,y}}+\norm{\nabla\theta_{\neq}(t)}^2_{L^2_{x,y}}\gtrsim
\l t \r \left[\norm{\bX^{in}_{\neq}}_{L^2_{z}H^{-\frac12}_{y}}^2 - K'\left(\int_0^t\norm{F(\tau, \bX(\tau))_{\neq}}_{L^2_{z,y}} \dd\tau\right)^2\right].
\end{align}
The rest of this section is devoted to providing a suitable upper bound for the nonlinear term above.
Precisely, we prove the following bound: 
\begin{lemma}\label{lem:instagoal}
Assuming Proposition \ref{prop:bootimpr}, there holds 
\begin{equation}\label{eq:instagoal}
\norm{F(t, \bX(t))_{\neq}}_{L^2_{z,y}} \lesssim \frac{\eps^2}{\l t\r^\frac12}, \qquad \forall t\leq \frac{\delta^2}{\eps^2},
\end{equation}
where $\eps,\delta$ are as in Theorem \ref{thm:mainT}. 
In fact, there holds for all $\lambda''_\infty < \lambda_\infty'$,
\begin{equation}\label{eq:instagoalGev}
	\norm{F(t, \bX(t))_{\neq}}_{\G^{\lambda_\infty''}} \lesssim \frac{\eps^2}{\l t\r^\frac12}, \qquad \forall t\leq \frac{\delta^2}{\eps^2},
\end{equation}
\end{lemma}

Assuming now Lemma \ref{lem:instagoal}, there exists some $K > 1$ such that \eqref{eq:instalowerbdd1} becomes
\begin{align}\label{eq:instalowerbdd2}
	\norm{\omega_{\neq}(t)}^2_{L^2_{x,y}}+\norm{\nabla\theta_{\neq}(t)}^2_{L^2_{x,y}}
	\gtrsim \l t \r \left[\norm{\bX^{in}_{\neq}}_{L^2_{z}H^{-\frac12}_{y}}^2 -K\eps^4\l t\r\right]
	\gtrsim \l t \r \left[\norm{\bX^{in}_{\neq}}_{L^2_{z}H^{-\frac12}_{y}}^2 -K\delta^2 \eps^2\right],
\end{align}
for every $t\leq \delta^2 \eps^{-2}$, which completes the proof of Theorem \ref{thm:mainT}. 
It now suffices to prove Lemma \ref{lem:instagoal}
\begin{proof}[Proof of Lemma \ref{lem:instagoal}]
We will simply prove \eqref{eq:instagoal} as \eqref{eq:instagoalGev} is a straightforward extension and is not required for the statement of Theorem \ref{thm:mainT}. 
From \eqref{eq:Zstar}-\eqref{eq:Qstar} and the fact that $H^3_{z,y}$ is an algebra, we find that
\begin{align}
\norm{F(t, \bX(t))_{\neq}}_{L^2_{z,y}} &\lesssim \norm{\left(p/k^2\right)^{-\frac{1}{4}}\cF(\nabla^\perp  \Psi^\star_{\neq} \cdot \nabla \sOmega)}_{L^2_{z,y}}
+\norm{ \left(p/k^2\right)^{-\frac{3}{4}}\cF((\sDelta_t-\Lambda_L)\Psi^\star_{\neq})}_{L^2_{z,y}}\notag\\
&\quad+\norm{\left(p/k^2\right)^\frac{1}{4}k \cF(\de_z\Psi^\star_{\neq}\de_y\Theta_0^{\star}+\nabla^\perp \Psi^\star_{\neq}\cdot \nabla  \Theta_{\neq}^{\star})}_{L^2_{z,y}}\notag\\
&\lesssim \frac{1}{\l t\r^\frac12} \norm{ \nabla^\perp \Psi^\star_{\neq}\cdot \nabla \sOmega}_{H^3_{z,y}}
+ \frac{1}{\l t\r^\frac32} \norm{(\sDelta_t-\Lambda_L)\Psi^\star_{\neq}}_{H^3_{z,y}}\\
&\quad+\l t\r^{\frac12}\norm{\de_z \Psi^\star_{\neq}   \de_y\Theta^{\star}_0}_{H^3_{z,y}}+\l t\r^{\frac12}\norm{\nabla^\perp \Psi^\star_{\neq} \cdot \nabla  \Theta^{\star}_{\neq}}_{H^3_{z,y}}\notag\\
\notag &\lesssim  \frac{1}{\l t\r^\frac12} \norm{ \Psi^\star_{\neq}}_{H^3_{z,v}}\left(\norm{ \sOmega}_{H^3_{z,y}}+ \l t\r (\norm{\de_y\Theta^{\star}_{0}}_{H^3_{z,y}}+\norm{\Theta^{\star}_{\neq}}_{H^3_{z,y}})\right)\\
&\quad + \frac{1}{\l t\r^\frac32} \norm{(\sDelta_t-\Lambda_L)\Psi^\star_{\neq}}_{H^3_{z,y}}.
\end{align}
In view of \eqref{ineq:sctrls}, it follows that  
\begin{align}
\norm{(\Lambda_t - \Lambda_L) \Psi_{\neq}^\ast}_{H^3_{z,y}} \lesssim \eps \jap{t}^{1/2}, 
\end{align}
and together with the rest of the estimates of \eqref{ineq:sctrls}, Lemma \ref{lem:instagoal} follows. 
\end{proof} 
This concludes the proof of Theorem \ref{thm:mainT}.

\section{The main weights and their properties}\label{sec:weightW}
This section is dedicated to the construction of the Fourier multipliers which will play the role of weights in our energy functional.
As anticipated in Section \ref{sub:weenbo}, the Fourier modes with horizontal frequency $k \neq 0$ need two different weights. We call the first one the \emph{linear weight}, as it allows to control linear terms, and has already been defined in \eqref{def:m}. We also introduced the  \emph{nonlinear weight} $A_k(t, \eta)$ in \eqref{def:A}, which encodes the dynamics of the nonlinear toy model derived in the previous sections. Here we provide a construction of the multiplier $w_k(t, \eta)$ in \eqref{eq:Jweight}.
Finally, the treatment of the zero mode $k=0$ requires a slightly different nonlinear weight, introduced in \eqref{def:ARvintro}, which we define now.

\subsection{Construction of the weight}
As the nonlinear weight $w_k(t, \eta)$ in \eqref{eq:Jweight} actually encodes the dynamics of the toy model, we start with a more detailed description of its growths.

\begin{proposition}\label{prop:key}
We denote $\tau=t-\eta/k$ and assume $\tau \in[-\eta/k^2,\eta/k^2]$: Let $f_R(-\eta/k^2)=f_{NR}(-\eta/k^2)=1$ be the initial data associated with system 
\eqref{eq:toyfR1}-\eqref{eq:toyfNR1}. Assume  also that $\eta/k^2\geq 1$. Then, there exists $\gamma\in (1,2)$
such that

\begin{align}
	f_R(\tau)&\lesssim \left(\frac{\eta}{k^2}\right)^{\gamma} 
	\begin{cases}
	(1+|\tau|)^{-\gamma}, \quad  &\tau\in [-\eta/k^2,0],\\
	(1+|\tau|)^{\gamma+\frac12}, \quad &\tau\in [0,\eta/k^2],
	\end{cases}\\
	f_{NR}(\tau)&\lesssim \left(\frac{\eta}{k^2}\right)^{\gamma+\frac12} 
	\begin{cases}
	(1+|\tau|)^{-\gamma-\frac12}, \quad  &\tau\in [-\eta/k^2,0],\\
	(1+|\tau|)^{\gamma}, \quad &\tau\in [0,\eta/k^2].
	\end{cases}
\end{align}
\end{proposition}

The proof can be obtained through a simple ODE argument. Based on this,
we are  ready to construct the weight $w_k(t, \eta)$ in \eqref{eq:Jweight}.
We first define the critical intervals. For $a\geq0 $, let $\lfloor a \rfloor \in \mathbb{N}$ be the integer part. Then, for any $\eta\in \R$, $1\leq |k|\leq \lfloor \sqrt{\eta}\rfloor $ and $\eta k\geq 0$, set 
\begin{align}
	t_{|k|,\eta}=\frac{|\eta|}{|k|}-\frac{|\eta|}{2|k|(|k|+1)}=\frac{|\eta|}{|k|+1}+\frac{|\eta|}{2|k|(|k|+1)},\qquad t_{0,\eta}= 2|\eta|.
\end{align}
The critical intervals are then defined as 
\begin{equation}
	\label{def:Iketa}
	{I}_{k,\eta}={I}^L_{k,\eta}\cup {I}^R_{k,\eta}= \begin{cases}
		\left[t_{|k|,\eta},\frac{\eta}{k}\right]\cup \left[\frac{\eta}{k},t_{|k|-1,\eta}\right] \quad& \text{if } \eta k\geq0 \text{ and } 1\leq |k|\leq \lfloor \sqrt{\eta}\rfloor ,\\
		\emptyset \quad & \text{otherwise}. 
	\end{cases}
\end{equation}
For technical reasons, it is convenient to also introduce the \textit{resonant intervals} as 
\begin{equation}\label{def:boldIketa}
\boldsymbol{I}_{k,\eta}=\begin{cases}
		I_{k,\eta} \quad& \text{if } \ 2\sqrt{\eta }\leq t_{k,\eta},\\
		\emptyset \quad & \text{otherwise}.
	\end{cases}
\end{equation}
We now follow the construction of \cite{BM15}, using \eqref{eq:toyfR1}-\eqref{eq:toyfNR1} as the reference toy model. In particular, for $t\in \bI_{k,\eta}$ we choose $(w_{NR},w_R)$ such that
\begin{align}
	\label{def:detwR}\de_t w_R\approx& \left(\frac{k^2}{\eta}\right)^\frac12 \frac{1}{(1+|t-\frac{\eta}{k}|)^\frac12}w_{NR},\\
	\label{def:detwNR}	\de_t w_{NR}\approx &\left(\frac{\eta}{k^2}\right)^\frac12 \frac{1}{(1+|t-\frac{\eta}{k}|)^\frac32}w_{R}.
\end{align}
We assume $w_R(t, \eta)=w_{NR}(t, \eta)=1$ for $t \ge 2 \eta$ and we construct the weight
\textit{backward} in time, by gluing all the growths of  Proposition \ref{prop:key}. For simplicity, we assume $k,\eta\geq 0$, but the construction below easily applies to the case $k,\eta \leq0$ (when they have different signs we take $w_k(t,\eta) \equiv 1$). We start our construction with the non-resonant part of the weight.
 Let $w_{NR}$ be such that $w_{NR}(t,\eta)=1$ for $t\geq 2\eta$ or $|\eta|\leq 2$. Assume that $w_{NR}(t_{|k|-1,\eta},\eta)$ is known. Motivated by Proposition \ref{prop:key}, for any $1\leq k\leq \lfloor \sqrt{\eta}\rfloor $, we define 
 \begin{align}
 	\label{def:wNR}
 		w_{NR}(t,\eta)=&\left(\frac{k^2}{\eta}\left(1+b_{k,\eta}\left|t-\frac{\eta}{k}\right|\right)\right)^{\gamma}w_{NR}(t_{k-1},\eta), \quad \text{for }t \in I^R_{k,\eta}\\
 	w_{NR}(t,\eta)=&\left(1+a_{k,\eta}\left|t-\frac{\eta}{k}\right|\right)^{-\frac12-\gamma}w_{NR}\left(\frac{\eta}{k},\eta\right), \quad \text{for }t \in {I}^L_{k,\eta},
 \end{align}
where $I_{k,\eta}^R, I_{k, \eta}^L$ have been introduced in \eqref{def:Iketa}, while $b_{k,\eta}$ and $a_{k,\eta}$ satisfy
\begin{equation*}
	\frac{k^2}{\eta}\left(1+b_{k,\eta}\left|t_{k-1,\eta}-\frac{\eta}{k}\right|\right)=1, \qquad \frac{k^2}{\eta}\left(1+a_{k,\eta}\left|t_{k,\eta}-\frac{\eta}{k}\right|\right)=1.
\end{equation*}
In particular, we have 
\begin{align}
	b_{k,\eta}=\begin{cases}\displaystyle
		\frac{2(k-1)}{k}\left(1-\frac{k^2}{\eta}\right), \quad & \text{for } k\geq 2,\\
	\displaystyle	1-\frac{1}{\eta}, \quad &\text{for } k=1,
	\end{cases} \qquad
a_{k,\eta}=	\frac{2(k+1)}{k}\left(1-\frac{k^2}{\eta}\right). 
\end{align}
Thanks to this choice, notice that 
\begin{align}
	\label{eq:maxgrowwNR} w_{NR}\left(\frac{\eta}{k},\eta\right)=&\left(\frac{k^2}{\eta}\right)^\gamma w_{NR}(t_{k-1,\eta},\eta), \qquad w_{NR}(t_{k-1,\eta},\eta)=\left(\frac{\eta}{k^2}\right)^{\frac12+2\gamma}w_{NR}(t_{k,\eta},\eta ),
\end{align} 
where $a_{k,\eta}, b_{k,\eta}$ are chosen to ensure the last identity. For $t\in [0,t_{\lfloor \sqrt{\eta}\rfloor ,\eta}]$  we define $w_{NR}(t,\eta)=w(t_{\lfloor\sqrt{\eta}\rfloor,\eta},\eta)$. We then define $w_R$ as suggested by Proposition \ref{prop:key}, namely 
\begin{align}
	\label{def:wR}
	w_{R}(t,\eta)=&\left({\frac{k^2}{\eta}\left(1+b_{k,\eta}\left|t-\frac{\eta}{k}\right|\right)}\right)^\frac12 w_{NR}(t,\eta), \quad \text{for }t \in {I}^R_{k,\eta},\\
	w_{R}(t,\eta)=&\left({\frac{k^2}{\eta}\left(1+a_{k,\eta}\left|t-\frac{\eta}{k}\right|\right)}\right)^\frac12w_{NR}\left(t,\eta\right), \quad \text{for }t \in {I}^L_{k,\eta}.
\end{align}
By the expressions of $a_{k,\eta}, b_{k,\eta}$, we also have that
\begin{equation}
	w_R(t_{k,\eta},\eta)=w_{NR}(t_{k,\eta},\eta), \qquad w_R\left(\frac{\eta}{k},\eta\right)=\sqrt{\frac{k^2}{\eta}}w_{NR}\left(\frac{\eta}{k},\eta\right).
\end{equation}
The main weight $w_k(t, \eta)$ is finally given by 
\begin{equation}
	\label{def:w}
	w_k(t,\eta)=\begin{cases}
		w_k(t_{[\sqrt{\eta},\eta],\eta},\eta), \quad &t<t_{\lfloor \sqrt{\eta}\rfloor,\eta},\\
		w_{NR}(t,\eta), \quad& t\in [t_{\lfloor \sqrt{\eta}\rfloor,\eta},2\eta]\setminus {I}_{k,\eta},\\
		w_{R}(t,\eta), \quad& t\in I_{k,\eta},\\
		1 \quad& t\geq 2\eta.
	\end{cases}
\end{equation}
\begin{remark}
	\label{rem:maxgrowth}
	Since $a_{k,\eta}, b_{k,\eta}\to 0$ as $k\to \lfloor \sqrt{\eta}\rfloor$, we also have that $\de_tw\approx 0$ for $k\approx \lfloor \sqrt{\eta}\rfloor$. This means that estimates \eqref{def:detwR}-\eqref{def:detwNR} are only useful provided that $k\leq \frac12 \sqrt{\eta}$ or $t\geq 2\sqrt{\eta}$ (equivalent). 
\end{remark}

The weights $A$ and $J$ are defined in terms of $w$ in \eqref{def:A} and \eqref{eq:Jweight}, respectively.
Notice that for $k=0$ the weight $w$ is \textit{always} non-resonant and the linear multiplier is $m \equiv 1$. Therefore, $A_0(t,\eta)$ always encodes non-resonant regularity.

The change of coordinates requires a stronger weight. More precisely, we need to propagate the same regularity of the homogeneous case treated in \cite{BM15}, where the weight of the coordinate system assigns always the resonant regularity given by the toy model \eqref{eq:toyfR1BM}. Hence, we define 
\begin{align}
	\label{def:wRv}
 w_{R}^v(t,\eta)=&\frac{k^2}{\eta}\left(1+b_{k,\eta}\left|t-\frac{\eta}{k}\right|\right) w_{NR}(t,\eta), \quad \text{for }t \in {I}^R_{k,\eta},\\
		w_{R}^v(t,\eta)=&\frac{k^2}{\eta}\left(1+a_{k,\eta}\left|t-\frac{\eta}{k}\right|\right)w_{NR}\left(t,\eta\right), \quad \text{for }t \in {I}^L_{k,\eta}.
\end{align}
The weight $w^v(t, \eta)$ is given by
\begin{align}\label{def:wv}
	w^{v}(t,\eta)&=\begin{cases}
		(w^{v}_R)^{-1}(t_{[\sqrt{\eta},\eta],\eta},\eta), \quad &t<t_{\lfloor \sqrt{\eta}\rfloor,\eta},\\
		(w^{v}_R(t,\eta))^{-1}, \quad& t\in [t_{\lfloor \sqrt{\eta}\rfloor,\eta},2\eta],\\
		1 \quad& t\geq 2\eta,
	\end{cases}
\end{align}
and $A^v$ is defined in \eqref{def:ARvintro}.

\begin{remark}
	\label{remarkino}
	It is immediate to check that $w^v(t, \eta)$ encodes more regularity than $w(t, \eta)$. Notice that
\begin{align}
	p^{-\frac14}_{k}(t,\eta)(w_R)^{-1}(t,\eta) \lesssim \left(\frac{|k|}{|\eta|}\right)^{\frac12}(w_{R}^v)^{-1}(t,\eta),
\end{align}
and if $t\geq 2\sqrt{\eta}$ then
\begin{align}
	p^{-\frac14}_{k}(t,\eta)(w_R)^{-1}(t,\eta)\mathbbm{1}_{\{t\geq 2 \sqrt{\eta}\}}\approx \left(\frac{|k|}{|\eta|}\right)^{\frac12}(w_{R}^v)^{-1}(t,\eta)\mathbbm{1}_{\{t\geq 2 \sqrt{\eta}\}}.
\end{align}
\end{remark}
	Finally, we underline the following useful inequalities:
	\begin{align}
		\label{bd:towerA} A_0\leq A^{v}, \qquad A_0\leq \widetilde{A}\leq A, \qquad A_0\leq\mathbbm{1}_{|k|\leq |\eta|}A\lesssim \widetilde{A}\lesssim A^v.
	\end{align}
	
\subsection{Properties of the weights}\label{sub:weigthprop}
Here we present technical results which will be used to deal with the weights, throughout the paper.
First, we recall the trichotomy lemma due to \cite{BM15}*{Lemma 3.2}.
\begin{lemma}[Lemma 3.2, \cite{BM15}]
	\label{lemma:trichotomy}
	Let $\xi, \eta$ be such that there exists some $C\geq 1$ with $C^{-1}|\xi|\leq |\eta|\leq C|\xi|$ and let $k,\ell$ be such that $t\in I_{k,\eta}$ and $t\in I_{\ell,\xi}$, hence $k\approx \ell$. Then at least one of the following holds:
	\begin{itemize}
		\item[(a)] $k=\ell$  \qquad (almost the same interval),
		\item[(b)] $\displaystyle\left|t-\frac{\eta}{k}\right|\geq (10C)^{-1}\frac{|\eta|}{k^2}$ and $\displaystyle\left|t-\frac{\xi}{\ell}\right|\geq (10C)^{-1}\frac{|\xi|}{\ell^2}$ \qquad (far from resonance),
		\item[(c)] $\displaystyle\left|\eta-\xi\right|\gtrsim \max\left\{\frac{|\eta|}{|\ell|},\frac{|\xi|}{|k|}\right\}$  \qquad  (well-separated).
	\end{itemize} 
\end{lemma}
We state here some useful inequalities, whose proofs can be found in \cite{BM15}.
\begin{lemma}\label{lemma:appfreq} Let $0 < s < 1$ and $a,b\geq0$.
\begin{itemize}
\item If $|a-b| \le b/C$ for some $C>2$, then
\begin{align}\label{app:inequality} 
| a^s-b^s| \le \frac{s}{(C-1)^{1-s}} |a-b|^s.
\end{align}
\item If $| a-b| \le C b$ for some $C>0$, then
\begin{align}\label{app:inequality2}
a^s \le \left(\frac{C}{1+C}\right)^{1-s} (|a-b|^s + b^s).
\end{align}

\end {itemize}
\end{lemma}
\subsubsection{Properties of the main weight $w$}
We collect the properties of the main weight $w_k(t, \eta)$, which in most cases will be analogous to \cite{BM15}, while the substantial differences will be carefully highlighted.  
First, we note that the maximal growth of the weight $w$ dictates the Gevrey-2 regularity requirements.
The proof is essentially the same as in \cite{BM15} and is hence omitted here. 
\begin{lemma}
	\label{lemma:maxgrowth}
	Let $\mu=4(1/2+2\gamma)$. For $|\eta|>1$ we have  
	\begin{equation}
		w_k^{-1}(0,\eta)=w_k^{-1}(t_{[\sqrt{|\eta|}],\eta},\eta)\sim \frac{1}{|\eta|^{\frac{\mu}{8}}}\e^{\frac{\mu}{2}\sqrt{|\eta|}}.
	\end{equation}
\end{lemma}
Our weights also have the analogous property of \cite{BM15}*{Lemma 3.3}; the proof is similar and is hence omitted. 

\begin{lemma}
	\label{lemma:detw/wfar}
	 For $t\in I_{k,\eta}$ and $t>2\sqrt{\eta}$, we have 
	\begin{equation}
		\label{bd:detw/w}
		\frac{\de_t w_{NR}(t,\eta)}{w_{NR}(t,\eta)}\approx \frac{1}{1+|\frac{\eta}{k}-t|}\approx\frac{\de_t w_{R}(t,\eta)}{w_{R}(t,\eta)}.
	\end{equation}
\end{lemma} 
We now state two crucial results, which allow us to \emph{exchange frequency} when dealing with $w$. This is completely analogous to \cite{BM15}{Lemma 3.4} and the proof is hence omitted. 
\begin{lemma}
	\label{lemma:detw/w}
	For $t\geq 1$ and $k,\ell, \eta, \xi$ such that $2\max(\sqrt{|\xi|},\sqrt{|\eta|})<t<2\min(|\xi|,|\eta|)$, we have 
	\begin{equation}
		\label{bd:wexfaway}
		\frac{\de_t w_k(t,\eta)}{w_k(t,\eta)}\frac{w_\ell(t,\xi)}{\de_tw_\ell(t,\xi)}\lesssim \jap{\eta-\xi}.
	\end{equation}
	For all $t\geq1$ and  $k,\ell, \eta, \xi$ such that for some $C\geq 1$, $C^{-1}|\xi|\leq |\eta|\leq C|\xi|$ one has 
	\begin{equation}
		\label{bd:wexfgen}
		\sqrt{\frac{\de_t w_\ell(t,\xi)}{w_\ell(t,\xi)}}\lesssim_C\left(\sqrt{\frac{\de_t w_k(t,\eta)}{w_k(t,\eta)}}+\frac{|\eta|^\frac{s}{2}}{t^s}\right)\jap{\eta-\xi}.
	\end{equation}
\end{lemma}
\begin{lemma}
	For all $t,\eta,\xi$ we have 
	\begin{equation}
		\label{bd:wNRexgen}
		\frac{w_{NR}(t,\xi)}{w_{NR}(t,\eta)}\lesssim \e^{\mu|\eta-\xi|^\frac12}.
	\end{equation}
\end{lemma}

\subsubsection{Properties of $J, J^v$.}
We show how to exchange frequencies when dealing with the weights $J, J^v$. 
The proof is again analogous to that of \cite{BM15}*{Lemma 3.6}, with a different regularity imbalance and is omitted.  
\begin{lemma}\label{lem:analog}
	In general we have
	\begin{equation}
		\label{bd:Jexgen}
		\frac{J_k(\eta)}{J_\ell(\xi)}\lesssim \left(\frac{|\eta|}{k^2(1+|t-\frac{\eta}{k}|)}\right)^\frac12 \e^{9\mu|k-\ell,\eta-\xi|^\frac12}.
	\end{equation}
If $t \in \boldsymbol{I}_{k,\eta}\cap \boldsymbol{I}_{k,\xi}$, $k\neq \ell$ then 
\begin{align}
	\label{bd:JexTDs}
		\frac{J_k(\eta)}{J_\ell(\xi)}&\lesssim \frac{|\eta|^\frac12}{|k|}\sqrt{\frac{\de_tw_k(t,\eta)}{w_k(t,\eta)}}\e^{23\mu|k-\ell,\eta-\xi|^\frac12},\\
			\label{bd:JexTD}
		\frac{J_k(\eta)}{J_\ell(\xi)}&\lesssim \frac{|\eta|}{k^2}\sqrt{\frac{\de_tw_k(t,\eta)}{w_k(t,\eta)}}\sqrt{\frac{\de_tw_\ell(t,\xi)}{w_\ell(t,\xi)}}\e^{23\mu|k-\ell,\eta-\xi|^\frac12}.
\end{align}
If any of the following holds: $(t\in \boldsymbol{I}_{k,\eta}^c)$ or $(k=\ell)$ or ($|\eta|\approx |\xi|$ and $t\in \boldsymbol{I}_{k,\eta}\cap \boldsymbol{I}_{k,\xi}^c$) we have 
\begin{equation}
	\label{bd:Jeximp}
	\frac{J_k(\eta)}{J_\ell(\xi)}\lesssim \e^{10\mu|k-\ell,\eta-\xi|^\frac12}, \quad \frac{J^v(\eta)}{J^v(\xi)}\lesssim \e^{10\mu|\eta-\xi|^\frac12}.
\end{equation}
If  $t\in \boldsymbol{I}_{k,\eta}^c\cap \boldsymbol{I}_{\ell,\xi}$ and $|\eta|\approx |\xi|$ then 
\begin{equation}
	\label{bd:Jexgood}
	\frac{J_k(\eta)}{J_\ell(\xi)}\lesssim \left(\frac{\ell^2(1+|t-\frac{\xi}{\ell}|)}{|\xi|}\right)^\frac12 \e^{11\mu|k-\ell,\eta-\xi|^\frac12}.
\end{equation}
\end{lemma}

\begin{remark}
Lemma \ref{lem:analog} is analogous to  \cite{BM15}*{Lemma 3.6}. However, estimate \eqref{bd:JexTDs} is due to the specific structure of our weight.
This is a consequence of the fact that our weight is slightly weaker than the one used in \cite{BM15}. 
\end{remark}

When $t$ is small enough, $J, J^v$ attain their maximal growth and behave like exponential Fourier multipliers, so allowing us to gain half derivative from a commutator term. This is the content of the next result, which is the analogue of \cite{BM15}*{Lemma 3.7}.
\begin{lemma}
	\label{lemma:commJ}
	Let $t\leq\frac12 \min\{\sqrt{|\eta|},\sqrt{|\xi|}\}$. Then 
	\begin{equation}
		\left|\frac{J_k(\eta)}{J_\ell(\xi)}-1\right|\lesssim \frac{\jap{k-\ell,\eta-\xi}}{\sqrt{|\eta|+|\xi|+|k|+|\ell|}}\e^{11\mu|k-\ell,\eta-\xi|^\frac12}, \quad \left|\frac{J^v(\eta)}{J^v(\xi)}-1\right|\lesssim \frac{\jap{\eta-\xi}}{\sqrt{|\eta|+|\xi|}}\e^{11\mu|\eta-\xi|^\frac12}.
	\end{equation} 
\end{lemma}

\subsubsection{Properties of $p$}
We also need to exchange frequencies in the multiplier $p$.

\begin{lemma}Let $k,\ell,\eta,\xi$ be given. We have the following:
\begin{equation}
\label{bd:p/p}
\sqrt{\frac{p_\ell(\xi)}{p_k(\eta)}}\lesssim\l k-\ell,\eta-\xi\r^3\begin{cases}
\displaystyle	\frac{|\eta|}{|k|^2(1+|\frac{\eta}{k}-t|)}, \quad & \text{if }t\in I_{k,\eta}\cap I_{\ell,\xi}^c,\\
\displaystyle\frac{|\ell|^2(1+|\frac{\xi}{\ell}-t|)}{|\xi|}, \quad & \text{if }t\in I_{k,\eta}^c\cap I_{\ell,\xi},\\
1,\quad & \text{in all the other cases}.
\end{cases}
\end{equation}
In general 
\begin{equation}
\label{bd:p/pgen}
\sqrt{\frac{p_\ell(\xi)}{p_k(\eta)}}\lesssim\l k-\ell,\eta-\xi\r^3 \jap{t}.
\end{equation}
If $k=\ell$, we have 
\begin{equation}
\label{bd:p/pkk}
\sqrt{\frac{p_k(\xi)}{p_k(\eta)}}\lesssim\l\eta-\xi\r.
\end{equation}
Let $\ell\neq 0$, $t\geq1$ and $1/2<s<1$. Then 
\begin{align}
\label{bd:ketapc}
|\ell,\xi|^{1-\frac{s}{2}}p^{-1}_{\ell}(\xi)\mathbbm{1}_{t\in I_{\ell,\xi}^c}&\lesssim\jap{\frac{\xi}{\ell t}}^{-1} \frac{|\ell,\xi|^{\frac{s}{2}}}{\jap{t}^{2s}},\\
\label{bd:ketapct}
\frac{1}{\jap{t}}|\ell,\xi|^{1-\frac{s}{2}}p^{-\frac12}_{\ell}(\xi)\mathbbm{1}_{t\in I_{\ell,\xi}^c}&\lesssim\jap{\frac{\xi}{\ell t}}^{-s} \frac{|\ell,\xi|^{\frac{s}{2}}}{\jap{t}^{2s}},\\
\label{bd:ketapct12}
\frac{1}{\jap{t}^\frac12}|\ell,\xi|^{1-\frac{s}{2}}p^{-\frac34}_{\ell}(\xi)\mathbbm{1}_{t\in I_{\ell,\xi}^c}&\lesssim\jap{\frac{\xi}{\ell t}}^{-s} \frac{|\ell,\xi|^{\frac{s}{2}}}{\jap{t}^{2s}}.
 \end{align}

\end{lemma}

\begin{proof}	
First, notice that 
\begin{equation}
\label{bd:trivp/p}
\sqrt{\frac{p_\ell(\xi)}{p_k(\eta)}}\lesssim \left(\frac{|\ell|(1+|\frac{\xi}{\ell}-t|)}{|k|(1+|\frac{\eta}{k}-t|)}\right)\lesssim \l k-\ell\r\left(\frac{1+|\frac{\xi}{\ell}-t|}{1+|\frac{\eta}{k}-t|}\right).
\end{equation}
If $t\in I_{k,\eta}^c\cap I_{\ell,\xi}$, then \eqref{bd:p/p} follows from $|\frac{\eta}{k}-t|\gtrsim \frac{|\eta|}{k^2}$. Now consider $t\in I_{k,\eta}\cap I_{\ell,\xi}^c$. It holds that
\begin{equation}
\frac{1+|\frac{\xi}{\ell}-t|}{1+|\frac{\eta}{k}-t|}
\leq 1+\frac{1}{1+|\frac{\eta}{k}-t|}\left(\left|\frac{\xi}{\ell}-\frac{\eta}{\ell}\right|+\left|\frac{\eta}{\ell}-\frac{\eta}{k}\right|\right)
\leq 1+\frac{1}{1+|\frac{\eta}{k}-t|}\left(\left|\eta-\xi\right|+\frac{|\eta|}{|k\ell| }\left|k-\ell\right|\right).\label{bd:profp/p1}
\end{equation} 
Then, if $|\ell| /2\leq |k|\leq 2|\ell|$ or $|k|\geq 2|\ell|$, we have 
\begin{equation*}
\frac{|\eta|}{|k\ell| }\lesssim \frac{|\eta|}{|\ell|^2}.
\end{equation*}
If $|k|\leq |\ell|/2$, then $|k-\ell|\geq |\ell|/2$, so that
\begin{equation}
\frac{|\eta|}{|k\ell| }\lesssim \frac{|\ell||\eta|}{|k||\ell|^2}\lesssim |k-\ell|\frac{|\eta|}{|\ell|^2}.
\end{equation}
Since $|\eta/k|\approx t \geq 1 $ in this interval, combining the two bounds above with \eqref{bd:profp/p1} implies
\begin{align}
\label{bd:p/pjapgen}
\frac{1+|\frac{\xi}{\ell}-t|}{1+|\frac{\eta}{k}-t|}\mathbbm{1}_{t\in I_{k,\eta}\cap I_{\ell,\xi}^c}
\lesssim \l k-\ell,\eta-\xi\r^2 \frac{|\eta|}{|k|^2(1+|\frac{\eta}{k}-t|)}.
\end{align} 
In the case $t\in I_{k,\eta}^c\cap I^c_{\ell,\xi}$, we use \eqref{bd:p/pjapgen} and $|\frac{\eta}{k}-t|\gtrsim  \frac{|\eta|}{k^2}$.
We are left with $t\in I_{k,\eta}\cap I_{\ell,\xi}$. We need to use the trichotomy Lemma \ref{lemma:trichotomy}. If  $(a)$ holds, then the conclusion follows by \eqref{bd:profp/p1}. If $(b)$ holds, then we apply the reasoning for $t\in I_{k,\eta}^c\cap I^c_{\ell,\xi}$. If $(c)$ holds, by \eqref{bd:profp/p1} we deduce that 
\begin{equation}
\frac{1+|\frac{\xi}{\ell}-t|}{1+|\frac{\eta}{k}-t|}\leq 1+\frac{1}{1+|\frac{\eta}{k}-t|}\left(\left|\xi-\eta\right|+\frac{|\xi-\eta|}{|k| }\left|k-\ell\right|\right)\leq \jap{k-\ell,\eta-\xi}^2,
\end{equation}
which gives \eqref{bd:p/p}-\eqref{bd:p/pkk}. Estimates \eqref{bd:ketapc} is proved in \cite{BM15}*{Lemma 6.1, (6.11a)} and we omit the proof.
To prove \eqref{bd:ketapct},  by Young's inequality, we get 
\begin{equation*}
	a^{2-2s}b^{2s}\leq (1-s)a^2+sb^2,
\end{equation*}
so that for $|\ell t|/2<|\xi|<2|\ell t|$, we have 
\begin{equation*}
\frac{1}{\jap{t}}|\ell,\xi|^{1-s}p^{-\frac12}_\ell(\xi)\mathbbm{1}_{t\in I_{\ell,\xi}^c}\lesssim \frac{1}{\jap{t}}\frac{|\ell|^{1-s}t^{1-s}}{(\ell^2+\jap{\xi/\ell}^2)^{\frac12}}\lesssim \frac{1}{\jap{t}}\frac{|\ell|^{1-s}t^{1-s}}{|\ell|^{1-s}\jap{\xi/\ell}^{s}}\lesssim \jap{\frac{\xi}{\ell t}}^{-s} \frac{1}{\jap{t}^{2s}}.
\end{equation*}
 When $|\xi|\leq |\ell t|/2$, recalling that $1/2<s<1$,
\begin{equation*}
\frac{1}{\jap{t}}|\ell,\xi|^{1-s}p^{-\frac12}_\ell(\xi)\mathbbm{1}_{t\in I_{\ell,\xi}^c}\lesssim \frac{1}{\jap{t}}\frac{|\ell t|^{1-s}}{|\ell t|}\lesssim \frac{1}{\jap{t}^{1+s}}\lesssim \jap{\frac{\xi}{\ell t}}^{-s} \frac{1}{\jap{t}^{2s}}.
\end{equation*}
For $|\xi|>2|\ell t|$, the factor $\jap{\xi/\ell t}^{-s}$ plays a role.We then argue as follows
\begin{equation*}
\frac{1}{\jap{t}}|\ell,\xi|^{1-s}p^{-\frac12}_\ell(\xi)\mathbbm{1}_{t\in I_{\ell,\xi}^c}\lesssim \frac{1}{\jap{t}}\frac{|\xi|^{1-s}}{|\xi|}\left(\frac{|\xi|}{|\ell t|}\frac{|\ell t|}{|\xi|}\right)^{s} \lesssim \frac{1}{\jap{t}^{1+s}}\left(\frac{|\ell t|}{|\xi|}\right)^s\leq\jap{\frac{\xi}{\ell t}}^{-s} \frac{1}{\jap{t}^{2s}},
\end{equation*}
where $s<1$. The remaining  \eqref{bd:ketapct12} can be proved as above using that
\begin{equation*}
a^{\frac43(1-s)}b^{\frac23(1+2s)}\leq \frac{2-2s}{3}a^2+\frac{1+2s}{3}b^2.
\end{equation*}
The proof is over.
\end{proof}

\section{Elliptic estimates}\label{sec:ellest}

This section is devoted to some elliptic estimates that play a crucial role for the nonlinear bounds.
The first inequality is a \emph{lossy elliptic estimate}, as it allows to gain time decay at the price of regularity. 

\begin{proposition} \label{prop:lossyelliptic}
Under the bootstrap hypotheses, for $\delta$ small enough, 
\begin{align}\label{bd:Ulow}
\norm{\Psi_{\neq}}_{\mathcal{G}^{\lambda, \sigma-3}} \lesssim \frac{\norm{\Omega}_{\mathcal{G}^{\lambda, \sigma-1}}}{\jap{t}^2}.
\end{align}
\end{proposition}
Notice that the smallness is encoded by $\delta$, as for the coordinate system we only have \eqref{boot:Ev}. Thanks to this result, we can treat $\Delta_t$ as a perturbation of $\Delta_L$ in a lower regularity class. Its proof is identical to \cite{BM15}*{Lemma 4.1}: we summarize it below for convenience of the reader. 
 \begin{proof}[Proof of Proposition \ref{prop:lossyelliptic}]
 		To write $\Delta_t$ as a perturbation of $\Delta_L$  we introduce the notation
 	\begin{equation}\label{def:g}
 		g=1-(v')^2.
 	\end{equation}
 	By definition of $\Delta_t$ in \eqref{def:deltat}, we get
 	\begin{equation}
 		\label{def:pertdeltat}
 		\Delta_L\Psi=\Omega+g(\de_v-t\de_x)^2\Psi-v''(\de_v-t\de_z)\Psi.
 	\end{equation}
 By the algebra properties of Gevrey spaces and the bound \eqref{bd:Avcoeff}, since $\eps t^{\frac12}\leq \delta$ we deduce 
 \begin{align*}
 	\norm{\Delta_L\Psi}_{\G^{\lambda, \sigma-1}}&\lesssim \norm{\Omega}_{\G^{\lambda,\sigma-1}}+\norm{g}_{G^{\lambda,\sigma-1}}\norm{(\de_v-t\de_z)^2{\Psi}}_{\G^{\lambda,\sigma-1}}+\norm{v''}_{\G^{\lambda,\sigma-1}}\norm{(\de_v-t\de_z)\Psi}_{\G^{\lambda,\sigma-1}}\\
 	&\lesssim \norm{\Omega}_{\G^{\lambda,\sigma-1}}+\delta \norm{\Delta_L\Psi}_{\G^{\lambda,\sigma-1}}.
 \end{align*}
As  $\delta$ is small, we can absorb the last term on the left-hand side. Since $p^{-1}_k(\eta)\lesssim (\jap{\eta}/\jap{kt})^2$, we have
\begin{equation*}
	\norm{\Psi_{\neq}}_{\G^{\lambda,\sigma-3}}=\norm{p^{-1}(p\widehat{\Psi}_{\neq})}_{\G^{\lambda,\sigma-3}}\lesssim \frac{1}{\jap{t}^2}\norm{(-\Delta_L)\Psi_{\neq}}_{\G^{\lambda,\sigma-1}}\lesssim \frac{\norm{\Omega}_{\G^{\lambda,\sigma-1}}}{\jap{t}^2},
\end{equation*}
hence proving the proposition.
 \end{proof}
 
 We now provide a more precise elliptic control, which plays a central role in the rest of the paper. In fact, if one knows a priori that $u_0^x\equiv 0$ then $\Delta_t=\Delta_L$ and the following proposition would not be needed.
\begin{proposition}\label{prop:elliptic}
	Under the bootstrap hypotheses, for $\delta$ small enough, 
	\begin{align}
		\norm{\jap{\frac{\de_v}{t\de_z}}^{-1}\left(\frac{|\nabla|^{\frac{s}{2}}}{\jap{t}^q}A+\sqrt{\frac{\de_tw}{w}}\widetilde{A}\right)((-\Delta_L)\Psi)_{\neq}}^2&\lesssim \sum_{j\in\{\lambda,w\}}G_j[\Omega]+\delta^2 (G_j^v[1-(v')^2]+G_j^v[\jap{\de_v}^{-1}v''])\notag\\
		&=: \ G_{elliptic}^\delta,\label{bd:precellcontr}
\end{align}
where $q$ is given in \eqref{def:dotlambda} and $G_j[\cdot]$ are defined in \eqref{def:Gw}, while $G_j^v[\cdot]$ are in \eqref{def:GwR}.
In addition
\begin{align}
	\notag \norm{\left(\frac{|\nabla|^{\frac{s}{2}}}{\jap{t}^q}A+\sqrt{\frac{\de_tw}{w}}\widetilde{A}\right)(-\Delta_L)^\frac34 |\de_z|^\frac 1 2 \Psi }^2&\lesssim \sum_{j\in \{\lambda, w\}}G_j[Z]+\eps^2 \big(G_j^v[1-(v')^2]+G_j^v[\jap{\de_v}^{-1}v'']\big)\\
	\notag&\qquad \qquad +\eps^2\jap{t}^{-2s}G_j^v[|\de_v|^s\jap{\de_v}^{-1}v'']\\
	\label{bd:ellipticGZw}
	&=:G_{elliptic}^\eps.
	\end{align}
	Also the following inequality holds true 
	\begin{equation}
	\label{bd:ellp34APsi} \norm{A(-\Delta_L)^{\frac34}|\de_z|^{\frac12}\Psi}^2+\jap{t}^{-1}\norm{\jap{\frac{\de_v}{t\de_z}}^{-1}A((-\Delta_L)\Psi)_{\neq}}^2\lesssim \eps^2.
	\end{equation}
	\end{proposition}
\begin{remark}
	Proposition \ref{prop:elliptic} plays the role of \cite{BM15}*{Proposition 2.4}, from which the main ideas of the following proof come from. The bound \eqref{bd:precellcontr} is the same as in \cite{BM15} but its validity for this problem lies within our underlying time-scale $O(\eps^{-2})$. 
	On the other hand, the inequalities \eqref{bd:ellipticGZw}-\eqref{bd:ellp34APsi} are specific to our problem. The proof of \eqref{bd:ellipticGZw}-\eqref{bd:ellp34APsi} crucially relies on the control of the coefficients with the stronger norm generated by $A^v$, see \eqref{def:ARvintro}. Observe also that we do not need the term $\jap{\de_v/t\de_z}^{-1}$ since, for the bounds \eqref{bd:ellipticGZw}-\eqref{bd:ellp34APsi}, we are able to exploit the control we have on $|\de_v|^{s}\jap{\de_v}^{-1}v''$, see \eqref{bd:Avscoeff}. In fact, also in \eqref{bd:precellcontr} it would be enough to only have $\jap{\de_v/t\de_z}^{-(1-s)}$, but this is not necessary for our purposes.
	\end{remark}
\begin{remark}
Integrating in time the right-hand sides of \eqref{bd:precellcontr} and \eqref{bd:ellipticGZw}, using the bootstrap hypotheses \eqref{boot:E}, \eqref{boot:En}, \eqref{bd:Avcoeff} and \eqref{bd:Avscoeff}, we get	\begin{equation}
	\label{bd:intGell}
		\int_1^t G^\eps_{elliptic}(\tau)\dd\tau\lesssim \eps^2,\qquad \int_1^tG^\delta_{elliptic}(\tau)\dd\tau \lesssim \eps^2 \l t\r\lesssim \delta^2.
	\end{equation}
	\end{remark}
\begin{proof}
	The proof of \eqref{bd:precellcontr} is exactly the same as the one of \cite{BM15}*{Proposition 2.4}, up to a replacement of $\eps$ with $\delta$ (and the use of $A\mathbbm{1}_{|k|\leq |\eta|}\lesssim A^{v}$ when needed), so it is omitted.
	In turn, we present the detailed proof of \eqref{bd:ellipticGZw} which, although being heavily inspired by \cite{BM15}*{Proposition 2.4}, it shows some substantial differences. Taking the Fourier transform of \eqref{def:pertdeltat}, we have 
	\begin{align}
		\label{eq:idPsi0}
		-p \widehat{\Psi}=&\widehat{\Omega} +\mathcal{F}\left(g(\de_v-t\de_z)^2\Psi\right)-\mathcal{F}\left(v''(\de_v-t\de_z)\Psi\right).
	\end{align}
Multiplying by $p^{-\frac14}|k|^{\frac12}$, we obtain 
\begin{align}
	\label{eq:idPsi14}
	-p^\frac34 |k|^{\frac12}\widehat{\Psi}&=Z +p^{-\frac14}|k|^\frac 12\left(\mathcal{F}\left(g(\de_v-t\de_z)^2\Psi\right)-\mathcal{F}\left(v''(\de_v-t\de_z)\Psi\right)\right)\\
	&=:Z+p^{-\frac14}|k|^\frac 12\left(T^g+T^{v''}\right). \notag
\end{align}
Now, define the following multipliers 
\begin{align}
	\label{def:Ms}
	\mathcal{M}^q_k(t,\eta)=\begin{cases}\displaystyle \frac{|k,\eta|^{\frac{s}{2}}}{\jap{t}^q}A_k(t,\eta), \ &\text{if } k\neq 0,\\
	0, \ &\text{if }  k=0,
	\end{cases}
		\qquad \mathcal{M}^w_k(t,\eta)=\begin{cases}\displaystyle\left(\sqrt{\frac{\de_t w_k}{w_k}}\widetilde{A}_k\right)(t,\eta),  \ &\text{if } k\neq 0,\\
	0, \ &\text{if }  k=0.
		\end{cases}
\end{align}
Hence, from \eqref{eq:idPsi14} we have 
\begin{equation}
	\label{eq:pertp34}
			\sum_{\iota\in\{q,w\}}\norm{\mathcal{M}^\iota (-\Delta_L)^\frac34 |\de_z|^\frac 12 \widehat{\Psi}}^2 
			\leq \sum_{j\in \{\lambda,w \}}G_j[Z]+\sum_{\iota \in\{q,w\}}\norm{\mathcal{M}^\iota(-\Delta_L)^{-\frac14}|\de_z|^\frac 12 (T^g+T^{v''})}^2,
\end{equation}
where the last inequality relies on the fact that $\widetilde{A}\leq A$. To prove \eqref{bd:ellipticGZw}, we need to control $T^g$ and $T^{v''}$. Taking into account the decoupling with respect to the $z$ frequencies, we make a paraproduct decomposition of $T^g$ and $T^{v''}$ only in the $v$ variable as 
\begin{align*}
	T^g&=T^g_{HL}+T^g_{LH}+T^g_{HH}, \quad T^{v''}=T^{v''}_{HL}+T^{v''}_{LH}+T^{v''}_{HH},
\end{align*}
with the notation introduced in \eqref{eq:paraprod}. We start with the low-high interactions.

\bullpar{Bounds on $T^{g}_{LH}$ and $T_{LH}^{v''}$}
Among the high-low terms, we only write down the computations for $T^g_{LH}$, 
\begin{equation}
\label{def:TgLH}
\mathcal{M}^\iota_k(\eta)p^{-\frac14}_{k}(\eta)|k|^\frac 12 T^g_{LH}=\mathcal{M}^\iota_k(\eta)p^{-\frac14}_{k}(\eta)\sum_{N\geq 8}\int \widehat{g}(\eta-\xi)_{<N/8}(\xi-kt)^2 |k|^\frac 1 2\widehat{\Psi}_k(\xi)_N \dd\xi.
\end{equation}
On the support of the integral $|\eta|\approx |\xi|$. From the paraproduct decomposition in \eqref{eq:paraprod}, $|\xi-\eta| \le 3/16 |\xi|$, so that Lemma \ref{app:exp:bound} applies and 
 $\e^{\lambda |k,\eta|^s}\leq \e^{\lambda |k,\xi|^s+c\lambda|\eta-\xi|^s}$ for some $c\in (0,1)$.  In addition, since every term of \eqref{def:TgLH} has the same horizontal frequency $k$, we can appeal to \eqref{bd:p/pkk}, \eqref{bd:Jeximp}, so obtaining that
 \begin{align*}
 p^{-\frac 14}_k(\eta) \lesssim \l \eta-\xi \r^\frac 12 p^{-\frac 14}_k(\xi), \quad J_k(\eta) \lesssim \e^{10\mu |\eta-\xi|^\frac 12} J_k(\xi).
 \end{align*}
 Altogether, since $m$ in \eqref{def:m} is bounded, this implies that
 \begin{align*}
 \mathcal{M}^q_k(\eta)p^{-\frac14}_{k}(\eta)\lesssim \jap{\eta-\xi}^{2}\e^{c\lambda|\eta-\xi|^s}\mathcal{M}^q_k(\xi)p^{-\frac14}_{k}(\xi).
 \end{align*}
Turning to $\mathcal{M}_k^w(\eta)$, we use the same frequency exchanges as before, together with \eqref{bd:wexfgen}, to get  
 \begin{align*}
 \mathcal{M}_k^w(\eta) p^{-\frac14}_{k}(\eta) \lesssim \l \eta - \xi \r^4 \e^{c\lambda|\eta-\xi|^s} (\mathcal{M}_k^w(\xi)+\mathcal{M}_k^q(\xi)) p_k^{-\frac 14}(\xi).
 \end{align*}
Since $(\xi-kt)^2p^{-\frac14}_k(\xi)\leq p^{\frac34}_k(\xi)$, by Young's convolution inequality and \eqref{bd:Avcoeff}, we have for $\iota \in \{q, w\}$ that
\begin{equation}
	\notag
	\norm{\mathcal{M}^\iota (-\Delta_L)^{-\frac14} |\de_z|^\frac 12 T^g_{LH}}^2\lesssim \delta^2 \sum_{\iota' \in \{q, w\}} \norm{ \mathcal{M}^{\iota'} (-\Delta_L)^\frac34 |\de_z|^\frac 12 \Psi}^2,
\end{equation}
where in the last line we used that $\eps t^\frac12\leq \delta$. Hence, we can absorb this term on the left-hand side of \eqref{eq:pertp34} for $\delta$ is sufficiently small.
The bound for the term with $T^{v''}$ is analogous and we omit it.

 \bullpar{Bounds on $T^g_{HL}$ and $T^{v''}_{HL}$}
Exchanging the role of $\eta-\xi$ and $\xi$ in \eqref{def:TgLH}, on the support of the integrand we have $|\eta|\approx|\xi|$. 
Notice that the $\Psi$ could be at high frequencies in $k$, therefore we further split these terms as follows
\begin{equation*}
	T_{HL}^f=T_{HL}^f(\mathbbm{1}_{16|k|> |\eta|}+\mathbbm{1}_{16|k|\leq |\eta|})=T_{HL}^{f,z}+T_{HL}^{f,v},
\end{equation*}
where $f\in \{g,v''\}$. When $16|k|>|\eta|$, we claim that there is some $c\in(0,1)$ such that 
\begin{equation}
	\label{bd:exchketa0}
	|k,\eta|^s\leq |k,\xi|^s+c|\eta-\xi|^s.
\end{equation}
This can be proved thanks to \eqref{app:inequality}, by considering separately the cases $\frac{1}{16} |\eta|\leq |k|\leq 16|\eta|$ and $|k|\geq 16|\eta|$. First, for  $\frac{1}{16} |\eta|\leq |k|\leq 16|\eta|$, we have that $|k,\xi-\eta| \le |k| + |\xi-\eta| \le 16 |\eta| + 24 |\xi|,$ where $|\eta| \le 25|\xi |$. Thus, we can use \eqref{app:inequality2} which gives \eqref{bd:exchketa0} and argue as in the low-high case before. Indeed, by \eqref{bd:exchketa0}  we can pay regularity on $g,v''$ and conclude by applying Young's convolution inequality. This way, 
\begin{equation}\label{eq:ellipticbound}
	\norm{\mathcal{M}^\iota (-\Delta_L)^{-\frac14}|\de_z|^\frac 12 (T^{g,z}_{HL}+T^{v'',z}_{HL})}^2\lesssim \delta^2 \sum_{\iota' \in \{q, w\}}\norm{ \mathcal{M}^{\iota'} (-\Delta_L)^{\frac34} |\de_z|^\frac 12 \Psi}^2,
\end{equation}
which can be absorbed in the left-hand side of \eqref{eq:pertp34}.

We now turn our attention to the terms where $16|k|\leq |\eta|$. Here, being the coefficients at high frequencies, we cannot absorb these terms on the left-hand side but we can exploit the integrability properties of $\Psi$. The most difficult term is $T^{v'',v}_{HL}$ since, in view of the bounds \eqref{bd:Avcoeff}-\eqref{bd:Avscoeff}, we need to recover some derivatives for $ v''$. This term is explicitly given by
\begin{equation}
\label{def:v''HL}
\mathcal{M}_k^\iota(\eta)p^{-\frac14}_{k}(\eta)|k|^\frac 12 T^{v'',v}_{HL}= \mathcal{M}_k^\iota(\eta)p^{-\frac14}_{k}(\eta) |k|^\frac 12 \sum_{N\geq 8}\int \mathbbm{1}_{16|k|\leq |\eta|}\widehat{v''}(\xi)_{N}\mathcal{F}((\de_y-t\de_x)\Psi)_k(\eta-\xi)_{<N/8} \dd\xi.
\end{equation}
On the support of the integrand $|\eta-\xi|\leq \frac{3}{16}|\xi|$. Since $16|k|\leq|\eta|$, we have that 
\begin{equation*}
	||k,\eta|-|\xi||\leq |k,\eta-\xi|\leq \frac{1}{16}|\eta|+|\eta-\xi|\leq \frac{1}{16}|\eta|+\frac{3}{16}|\xi|\leq \frac{9}{32}|\xi|.
\end{equation*}
In addition, from $16|k|\leq |\eta|$ we also get $A\lesssim \widetilde{A}$. 
It is now crucial to exploit the definition of the weight $A^v$ given in \eqref{def:ARvintro}. In particular, if $t\in I_{k,\xi}$, by \eqref{bd:p/p}, from Remark \ref{remarkino} and Lemma \ref{lemma:appfreq}, we deduce that
\begin{align}
	\label{bd:ellAv0}
	(p^{-\frac14}_kA)_k(\eta)\mathbbm{1}_{t\in I_{k,\xi}}&\lesssim p^{-\frac14}_k(\xi)\widetilde{A}(\xi)\jap{\eta-\xi}^\frac12\e^{c\lambda|\eta-\xi|^s}\lesssim \left(\frac{|k|}{|\xi|}\right)^\frac12 A^{v}(\xi)\e^{c\lambda|\eta-\xi|^s},
\end{align}
for some $c\in(0,1)$.
On the other hand, if $t\in I_{k,\xi}^c$,  then 
\begin{equation}
	\label{bd:ellAv1}
	p^{-\frac14}_k(\xi)\mathbbm{1}_{t\in I_{k,\xi}^c}\lesssim \left(\frac{|k|}{|\xi|}\right)^\frac12
\end{equation}
and $A\lesssim A^{v}$ when $16|k|\leq |\eta|$. Therefore, appealing again to \eqref{bd:wexfgen}-\eqref{bd:Jeximp}, in general we have
\begin{align}
	\label{bd:TtrueH}
	\mathcal{M}^\iota_k(\eta)p^{-\frac14}_{k}(\eta)\lesssim\frac{1}{|\xi|^\frac12}A^{v}(\xi) \e^{c\lambda|k,\eta-\xi|^s}\begin{cases}
	\displaystyle \sqrt{\frac{\de_t w_k}{w_k}}(t,\xi)+\frac{|\xi|^{\frac{s}{2}}}{\jap{t}^q}\qquad &\text{if }\iota=w,\\
		\displaystyle\frac{|\xi|^{\frac{s}{2}}}{\jap{t}^q}\qquad &\text{if }\iota=q,
	\end{cases} 
 \end{align}
where we have absorbed all the low-frequency Sobolev regularity in the exponential term. As remarked, we need to bound $\l \de_v \r^{-1} v''$, while up to now we only recovered half derivative. If $|\xi|\leq |k t|$, observe that
\begin{equation}
\label{bd:japell}
	1\lesssim {\jap{kt}^\frac12}{\jap{\xi}^{-\frac12}}\mathbbm{1}_{|\xi|\leq |kt|}.
\end{equation}
When $|\xi|\geq |kt|$, since $s>1/2$, we argue as follows 
\begin{equation}
\label{bd:japells}
\frac{1}{|\xi|^{\frac12}}\mathbbm{1}_{|\xi|\geq |kt|}=\frac{1}{|\xi|^{s-\frac12}}\frac{|\xi|^s}{\jap{\xi}}\mathbbm{1}_{|\xi|\geq |kt|}\leq \frac{t^\frac12}{t^s}\frac{|\xi|^s}{\jap{\xi}}\mathbbm{1}_{|\xi|\geq |kt|}.
\end{equation}
Combining the bounds \eqref{bd:japell}-\eqref{bd:japells} with \eqref{bd:TtrueH} and since $\mathcal{F}((\de_y-t\de_x )\Psi)\leq p^{\frac12}\widehat{\Psi}$, we have that
\begin{align}\label{eq:last-prop-elliptic}
		 &\norm{\mathcal{M}^\iota (-\Delta_L)^{-\frac14} |\de_z|^\frac 12 T^{v'',v}_{HL}}^2\notag\\
		 &\qquad\qquad\lesssim \jap{t}\sum_{j\in\{\lambda,w\}}\left(G_j^v[\jap{\de_v}^{-1} v'']+\jap{t}^{-2s}G_j^v[|\de_v|^s\jap{\de_v}^{-1} v'']\right)\norm{(-\Delta_L)^{\frac12}{\Psi}_{\neq}}_{\G^{\lambda,\sigma-5}}^2.
\end{align}
Then, using Proposition \ref{prop:lossyelliptic} and \eqref{bd:OmTh} we get
\begin{equation}
	\norm{(-\Delta_L)^{\frac12}\Psi_{\neq}}_{\G^{\lambda,\sigma-5}}\lesssim \jap{t}\norm{{\Psi}_{\neq}}_{\G^{\lambda,\sigma-4}}\lesssim \frac{\norm{\Omega}_{\G^{\lambda,\sigma-1}}}{\jap{t}}\lesssim\frac{\eps }{\jap{t}^{\frac12}}. \label{eq:ellipticproof-last}
\end{equation}
Hence, from \eqref{eq:last-prop-elliptic} and \eqref{eq:ellipticproof-last} we obtain 
\begin{align}\label{eq:last-prop-elliptic2}
		 \norm{\mathcal{M}^\iota (-\Delta_L)^{-\frac14} |\de_z|^\frac 12 T^{v'',v}_{HL}}^2&\lesssim \eps^2\sum_{j\in\{\lambda,w\}}\left(G_j^v[\jap{\de_v}^{-1} v'']+\jap{t}^{-2s}G_j^v[|\de_v|^s\jap{\de_v}^{-1} v'']\right).
\end{align}
The control of $T^{g,v}_{HL}$ when $16|k|<|\eta|$ follows by a similar argument. The only difference is the case $t\in I_{k,\xi}^c$. Indeed, we do not have to recover derivatives for $ g$ but extra time decay is necessary because one has to deal with the analogous of \eqref{eq:ellipticproof-last} with $p^\frac12 \widehat{\Psi}$ replaced with $p \widehat{\Psi}$. To overcome this problem, it is enough to split the relative size of $\xi$ with respect to $kt$. When $|\xi| \leq |kt|/2$, we have that 
\begin{equation*}
	p^{-\frac14}_k(\xi)\lesssim \jap{kt}^{-\frac12}.
\end{equation*}
When $|kt|/2\leq |\xi|\leq 2|kt|$ or $|\xi|\geq 2|kt|$, we can exchange the factor $|\xi|^{-1/2}$ in \eqref{bd:TtrueH} with $\jap{t}^{-1/2}$. Therefore, we always recover a factor $t^{-1/2}$ which is necessary to close the estimate. In particular, one has
\begin{align}
\notag
		 \norm{\mathcal{M}^\iota (-\Delta_L)^{-\frac14} |\de_z|^\frac 12 T^{g,v}_{HL}}^2&\lesssim \frac{1}{\jap{t}}\sum_{j\in\{\lambda,w\}}G_j^v[1-(v')^2]\norm{\Delta_L\Psi_{\neq}}_{\G^{\lambda,\sigma-4}}^2\lesssim \eps^2\sum_{j\in\{\lambda,w\}}G_j^v[1-(v')^2].
\end{align}

 \bullpar{Bounds on $T^g_{HH}$ and $T^{v''}_{HH}$}
For these terms it is easy to show
\begin{equation*}
	|T^g_{HH}|+|T^{v''}_{HH}|\lesssim \delta^2 \sum_{\iota' \in \{q, w\}} \norm{ \mathcal{M}^{\iota'} (-\Delta_L)^\frac34 |\de_z|^\frac 12 \widehat{\Psi}}^2.
\end{equation*}
Finally, to prove \eqref{bd:ellp34APsi}, for the first term we can follow the arguments done to prove \eqref{bd:ellipticGZw}, since no specific properties of $|\nabla|^{s/2}/\jap{t}^q$ or $\de_t w/w$ have been used. Analogously, the proof for the second term is obtained from the arguments for \eqref{bd:precellcontr}.  The bound then follows by the bootstrap hypotheses \eqref{boot:E}-\eqref{boot:Ev}.
\end{proof}

\section{Bound on the energy functional $E_L$}\label{sec:mainEn}
In this section, we aim at proving the first part \eqref{boot:impE} of the bootstrap Proposition \ref{prop:bootimpr}. In general, 
we have to estimate nonlinear terms of the type  $\jap{\boldsymbol{u}\cdot \nabla f,g}$. To do so, we  use a paraproduct decomposition, see \eqref{eq:paraprod}, where we decompose the nonlinear term in \textit{transport}, \textit{reaction} and \textit{remainder} contributions  
\begin{align}
	\label{def:TRrem}
	\jap{\boldsymbol{u}\cdot \nabla f,g} & = 
	\sum_{N>8} \jap{\boldsymbol{u}_{<N/8}\cdot \nabla f_N,g} +\sum_{N>8} \jap{\boldsymbol{u}_{N}\cdot \nabla f_{<N/8},g}
	+ \sum_{\substack{N,N'\\N/8\leq N'\leq 8N}}\jap{\boldsymbol{u}_{N}\cdot \nabla f_{N'},g} \\ 
	&  =: \sum_{N>8}T_N+\sum_{N>8}R_N+\mathcal{R}. 
\end{align}
Note that if we write e.g. 
\begin{align}
T_N&=\skli \overline{\widehat{g}}_{k}(\eta)\widehat{\nabla f}_{\ell}(\xi)_N\cdot \widehat{\boldsymbol{u}}_{k-\ell}(\eta-\xi)_{<N/8}\dd\eta \dd\xi, 
\end{align}
it is important to note that on the support of the integral we have  \begin{equation}
	\label{bd:freqsupp}
			||k,\eta|-|\ell,\xi||\leq |k-\ell,\eta-\xi|\leq \frac{3}{16}|\ell,\xi|,\qquad \frac{13}{16}|\ell,\xi |\leq |k,\eta|\leq \frac{19}{16} |\ell,\xi|.
\end{equation}
In particular, if \eqref{bd:freqsupp} holds,  thanks to Lemma \ref{lemma:appfreq} we have
\begin{align}\label{app:exp:bound}
\e^{\lambda |k, \eta|^s} \le \e^{\lambda |\ell, \xi|^s+\frac{s}{(13/3)^{s-1}} \lambda |k-\ell,\eta-\xi|^s}.
\end{align} 
In what follows, $c=c(s,\lambda_0,\sigma)\in(0,1)$ will denote a generic constant, independent of $\delta,\eps$. It will be mainly used in
terms of the form $ \e^{c \lambda |k-\ell,\eta-\xi|^s}$ to absorb Sobolev or exponential weights as the one in \eqref{app:exp:bound}.

We also need to distinguish between short (S), intermediate (I) and long (L) times via the cut-offs
\begin{equation}
	\label{def:cutofftime}
	\chi^S=\mathbbm{1}_{t\leq 2\max\{\sqrt{|\eta|},\sqrt{|\xi|}\}},\quad
	\chi^I=\mathbbm{1}_{2\max\{\sqrt{|\eta|},\sqrt{|\xi|}\}\leq t\leq 2\min\{|\eta|,|\xi|\}},
	\quad 	\chi^L=\mathbbm{1}_{t\geq 2\min\{|\eta|,|\xi|\}}.
\end{equation}

\subsection{The energy inequality}\label{app:energyIDZQ}
Recalling the definition of $E_L$ given in \eqref{eq:Energy}, we obtain the following result.
\begin{lemma}
For every $t\geq 0$ we have the energy inequality
\begin{align}
	\label{def:dtEm}
	\ddt E_L + \left(1-\frac{1}{2\beta}\right)\sum_{j\in\{\lambda,w,m\}}\left(G_j[Z]+G_j[Q]\right) \leq L^{Z,Q}+NL^{Z,Q}+\cE^{\div}+\cE^{\Delta_t},
\end{align}
where the $G_j[\cdot]$ are defined in \eqref{def:Gw} and the error terms are given by 
\begin{align}\label{eq:linear-error}
	L^{Z,Q}&=\frac{1}{4\beta}\left|\left\l \de_t\left(\frac{\de_tp}{|k| p^{\frac12} }\right)AZ ,AQ\right\r\right|\\
	{NL}^{Z,Q}&= \left|\left\l \cF\left(\left[A \left(\frac{p}{k^2}\right)^{-\frac{1}{4}} ,\bU\right]\cdot \nabla \Omega\right),AZ+\frac{1}{4\beta}\frac{\de_t p}{|k|p^\frac12}AQ\right\r\right|+\frac{1}{4\beta}\left|\left\l \left[\frac{\de_tp}{|k| p^{\frac12} },\bU\right] \cdot \nabla AZ ,AQ\right\r\right|\notag\\
	&\quad+\left|\left\l \cF\left(\left[A \left(\frac{p}{k^2}\right)^{\frac{1}{4}}k ,\bU\right]\cdot \nabla (i\beta\Theta)\right),AQ+\frac{1}{4\beta}\frac{\de_t p}{|k|p^\frac12}AZ\right\r\right|,\label{def:NL}\\
	\cE^{\div}&=\frac12\left|\l  \cF(\nabla\cdot\bU),|AZ|^2+|AQ|^2\r\right|+\frac{1}{4\beta}\left|\left\l  \cF(\nabla\cdot \bU)\frac{\de_tp}{|k| p^{\frac12} } AZ ,AQ\right\r\right|,\label{eq:div-error}\\
	\cE^{\Delta_t}&=\beta\left|\left\l |k|^{\frac32}p^{-\frac34}A\mathcal{F}\big((\Delta_t-\Delta_L)\Psi\big) ,AQ+\frac{1}{4}\frac{\de_tp}{|k| p^{\frac12} }AZ\right\r\right|.\label{eq:lasterror}
\end{align} 

\end{lemma}

\begin{proof}
The proof follows from the cancellations observed in \cite{BCZD20} 
together with the definition of $A$. 
Commutators have been introduced to better handle the transport structure. We recall briefly from \cite{BCZD20} that the Miles-Howard condition arises from using $|\de_tp| \leq 2|k| p^{\frac12} $, to obtain
\begin{align}
\frac{1}{2\beta }\Re\left\l \frac{\de_tp}{|k| p^{\frac12} }\frac{\de_tA}{A}AZ ,AQ\right\r\leq \frac{1}{2\beta} \sum_{j\in\{\lambda,w,m\}}\left(G_j[Z]+G_j[Q]\right).
\end{align}
\end{proof}

\subsection{Enumerating nonlinear terms}\label{sub:parapd}
We decompose the nonlinear term $NL^{Z,Q}$ in \eqref{def:NL} as in \eqref{def:TRrem},  
where the transport nonlinearity is further split as $T_N=\sum_{i=1}^2T_N^{\Omega,i}+T_{N}^{\Theta,i}+T_{N}^{err}$, where
\begin{align}
T_N^{\Omega,1}&=\left|\jap{\mathcal{F}\left(\left[A \left(\frac{p}{k^2}\right)^{-\frac{1}{4}} ,\bU_{<N/8}\right]\cdot \nabla \Omega_N\right),AZ }\right|,\\
T_N^{\Omega,2}&=\left|\jap{\mathcal{F}\left(\left[A \left(\frac{p}{k^2}\right)^{-\frac{1}{4}} ,\bU_{<N/8}\right]\cdot \nabla \Omega_N\right),\frac{\de_tp}{4\beta| k|p^\frac12}AQ}\right|,\\
T_{N}^{\Theta,1}&= \left|\jap{\mathcal{F}\left(\left[A \left(\frac{p}{k^2}\right)^{\frac{1}{4}}k ,\bU_{<N/8}\right]\cdot \nabla (i\beta \Theta)_N\right),AQ} \right|,\\
T_{N}^{\Theta,2}&= \left|\jap{\mathcal{F}\left(\left[A \left(\frac{p}{k^2}\right)^{\frac{1}{4}}k ,\bU_{<N/8}\right]\cdot \nabla (i\beta \Theta)_N\right), \frac{\de_tp}{4\beta |k|p^\frac12}AZ}\right|,\\
T_{N}^{err}&=\left|\jap{\mathcal{F}\left(\left[\frac{\de_t p}{|k|p^\frac12},\bU_{<N/8}\right]\cdot \nabla AZ_N\right),AQ}\right|. 
\end{align}
Since $|\de_tp|\leq 4\beta| k|p^\frac12$ for $\beta>1/2$,  it is enough to show how to deal with $T_N^{\Omega,1}$ and $T_{N}^{\Theta,1}$, as
$T_N^{\Omega,2}$ and $T_{N}^{\Theta,2}$ are completely analogous ($T_{N}^{err}$ will be dealt with separately).
Similarly, the reaction nonlinearity is given by
$R_N=\sum_{i=1}^2R_N^{\Omega,i}+R_{N}^{\Theta,i}+R_{N}^{err}$, where
\begin{align}
R_N^{\Omega,1}&=\left|\jap{\mathcal{F}\left(\left[A \left(\frac{p}{k^2}\right)^{-\frac{1}{4}} ,\bU_{N}\right]\cdot \nabla \Omega_{<N/8}\right),AZ}\right|,\label{de:react1}\\
R_N^{\Omega,2}&=\left|\jap{\mathcal{F}\left(\left[A \left(\frac{p}{k^2}\right)^{-\frac{1}{4}} ,\bU_{N}\right]\cdot \nabla \Omega_{<N/8}\right),\frac{\de_tp}{4\beta kp^\frac12}AQ}\right|,\label{de:react2}\\
R_{N}^{\Theta,1}&= \left|\jap{\mathcal{F}\left(\left[A \left(\frac{p}{k^2}\right)^{\frac{1}{4}}k ,\bU_{N}\right]\cdot \nabla (i\beta\Theta)_{<N/8}\right),AQ} \right|,\label{de:react3}\\
R_{N}^{\Theta,2}&= \left|\jap{\mathcal{F}\left(\left[A \left(\frac{p}{k^2}\right)^{\frac{1}{4}}k ,\bU_{N}\right]\cdot \nabla (i\beta\Theta)_{<N/8}\right),\frac{\de_tp}{4\beta |k|p^\frac12}AZ}\right|,\label{de:react4}\\
R_{N}^{err}&=\left|\jap{\mathcal{F}\left(\left[\frac{\de_t p}{|k|p^\frac12},\bU_N\right]\cdot \nabla AZ_{<N/8}\right),AQ}\right|. \label{de:react5}
\end{align}
Finally, the remainder reads as 
\begin{align}
	\mathcal{R}=\sum_{N\in \boldsymbol{D}}\sum_{N/8\leq N'\leq N}&\left|\jap{\mathcal{F}\left(\left[A \left(\frac{p}{k^2}\right)^{-\frac{1}{4}} ,\bU_{N}\right]\cdot \nabla \Omega_{N'}\right),AZ+\frac{\de_tp}{4\beta kp^\frac12}AQ}\right|\notag\\
	&+\left|\jap{\mathcal{F}\left(\left[A \left(\frac{p}{k^2}\right)^{\frac{1}{4}}k ,\bU_{N}\right]\cdot \nabla (i\beta\Theta)_{N'}\right),AQ+\frac{\de_tp}{4\beta kp^\frac12}AZ}\right|\notag\\
	&+\left|\jap{\mathcal{F}\left(\left[\frac{\de_t p}{|k|p^\frac12},\bU_N\right]\cdot \nabla AZ_{N'}\right),AQ}\right|.\label{eq:remaind}
\end{align}
In this section we prove the following.
\begin{proposition}
	Under the bootstrap hypothesis one has
	\begin{align}
		\label{bd:TNprop} \sum_{N>8}T_N&\lesssim  \sum_{j\in \{\lambda, w\}}\delta(G_j[Z]+G_j[Q])+\delta^{-1}\eps^2\big(G_j[\Omega]+G_j[\nabla_L\Theta])+\delta\frac{\eps^2}{t^{\frac32}}  \\
		\label{bd:RNprop} \sum_{N>8} R_N&\lesssim  \delta G^{\eps}_{elliptic}+\sum_{j\in \{\lambda, w\}}\delta(G_j[Z]+G_j[Q])+\eps^2G^v_{\lambda}[h] +\delta^{-1}\eps^2 t^{2+2s}G_{\lambda}[\jap{\de_v}^{-s}\mathcal{H}]\\
		\notag&\quad +\eps^4+\frac{\eps^3}{t^{\frac12}}-\dot{\lambda}(t)\delta^{-1}\eps^4,\\
		\label{bd:Rem} \mathcal{R} &\lesssim  \delta \frac{\eps^2}{t^{\frac32}}, \\
		\label{bd:rest}L^{Z,Q}&\leq  \frac12(1-\frac{1}{2\beta})(G_m[Z]+G_m[Q]),\\
		\label{bd:Ediv}\mathcal{E}^{\div}&\lesssim \frac{\eps^3}{t^2},\\
		\label{bd:EDeltat}\mathcal{E}^{\Delta_t}&\lesssim \delta G^{\eps}_{elliptic}+\frac{\eps^3}{t^\frac12}\\
		&\quad+\sum_{j\in \{\lambda,w\}}\delta G_j[Q]+\delta^{-1}\eps^2\left(G_j^v[1-(v')^2]+G_j^v[\jap{\de_v}^{-1} v'']+\jap{t}^{-2s}G_j^v[|\de_v|^s\jap{\de_v}^{-1} v'']\right)\notag.
		\end{align}
	\end{proposition}
Using the above bounds in the energy inequality \eqref{def:dtEm}, the proof of \eqref{boot:impE}  
follows from an integration of \eqref{def:dtEm} on 
$(1,t)$, the use of  the bootstrap hypothesis \eqref{boot:E}-\eqref{boot:Ev}, \eqref{bd:intGell}, and  $\eps t^{1/2}\leq \delta$.

\subsection{Transport nonlinearities}
In this section, we control the transport nonlinearities, defined in Section \ref{sub:parapd}, to prove \eqref{bd:TNprop}. 
First, by the bootstrap hypothesis \eqref{boot:En}, Proposition \ref{prop:lossyelliptic} and \eqref{boot:vdot} we get
\begin{align}\label{bd:bootUlow} 
  \norm{\bU(t)}_{\G^{\lambda,\sigma-6}}&\leq \norm{\nabla^\perp \Psi_{\neq}(t)}_{\G^{\lambda,\sigma-6}}+\norm{h\nabla^\perp \Psi_{\neq}(t)}_{\G^{\lambda,\sigma-6}} +\norm{\dot{v}}_{\G^{\lambda,\sigma-6}} \lesssim \frac{\eps}{\l t \r^{{\frac32}}}.
\end{align}
As mentioned already, we present only the proof for the terms $T_N^{\Omega,1}, T^{\Theta,1}_N$ and $T_N^{err}$.

 \subsubsection{{Bound on} $T_N^{\Omega,1}$}\label{subsec:TNOmega}
For any Fourier multiplier $m_1, m_2$ and any function $f$ one has
\begin{equation}
	\label{eq:trivcomm}
	[m_1m_2,f]=m_1[m_2,f]+[m_1,f]m_2.
\end{equation}
Since $p^{-1/4}A=m^{-1}p^{-1/4}(mA)$, using \eqref{eq:trivcomm} we rewrite $T_N^{\Omega,1}\leq T_N^{\Omega,A}+T_N^{\Omega,p}+T_N^{\Omega,m}$ 
with 
\begin{align}
	T_N^{\Omega,A}&=\left|\jap{\mathcal{F}(m^{-1} \left(\frac{p}{k^2}\right)^{-\frac{1}{4}} ([mA,\bU_{<N/8}]\cdot \nabla \Omega_N)),AZ}\right|,\\
	T_N^{\Omega,p}&=\left|\jap{\mathcal{F}(m^{-1}\left[ \left(\frac{p}{k^2}\right)^{-\frac{1}{4}} ,\bU_{<N/8}\right]\cdot \nabla (mA\Omega)_N),AZ}\right|,\\
	T_N^{\Omega,m}&=\left|\jap{\mathcal{F}([m^{-1},\bU_{<N/8}]\cdot \nabla (mAZ)_N),AZ}\right|.
\end{align}
Recall that on the support of the integral we always have \eqref{bd:freqsupp}.

\bullpar{Bound on $T_N^{\Omega,p}$}
Writing down this term and using that $|m|\approx 1$, we have 
\begin{align*}
	T_N^{\Omega,p}
	\lesssim&\skli \mathcal{T}_N^{p,1}+\mathcal{T}_N^{p,2},
	\end{align*} 
where we define
\begin{align*}
		\mathcal{T}_N^{p,1}&=  \frac{|  \left(p_\ell(\xi)/\ell^2\right)^{\frac{1}{4}}-\left(p_k(\xi)/k^2\right)^{\frac{1}{4}}|}{(p_k(\eta)/k^2)^\frac14} |\ell,\xi||\widehat{\bU}|_{k-\ell}(\eta-\xi)_{<N/8}|AZ|_{\ell}(\xi)_N |AZ|_k(\eta),\\
				\mathcal{T}_N^{p,2}&=  \frac{|\left(p_k(\eta)/k^2\right)^{\frac{1}{4}} - \left(p_k(\xi)/k^2\right)^{\frac{1}{4}}|}{(p_k(\eta)/k^2)^\frac14} |\ell,\xi||\widehat{\bU}|_{k-\ell}(\eta-\xi)_{<N/8}|AZ|_{\ell}(\xi)_N |AZ|_k(\eta)
			.\end{align*}
We claim that $\mathcal{T}^{p,1}_N$ and $\mathcal{T}^{p,2}_N$ are bounded in a way that is consistent with \eqref{bd:TNprop}.
To control $\mathcal{T}^{p,1}_N$, by the elementary identity 
\begin{equation}
a^{\frac14}-b^\frac14=\frac{a-b}{(a^\frac14+b^{\frac14})(a^\frac12+b^\frac12)},
\end{equation}
we deduce
\begin{align}
\left| \left(\frac{p_\ell(\xi)}{\ell^2}\right)^\frac 14 -  \left(\frac{p_k(\xi)}{k^2}\right)^\frac 14 \right| &= \frac{|(\frac{\xi}{\ell}-t)+(\frac{\xi}{k}-t)|}{((p_\ell(\xi)/\ell^2)^\frac12+(p_k(\xi)/k^2)^\frac12)}   \frac{|\frac{\xi}{\ell}-\frac{\xi}{k}|}{((p_\ell(\xi)/\ell^2)^\frac14+(p_k(\xi)/k^2)^\frac14)}\notag\\
	\label{bd:p14kell0}&\lesssim |k-\ell|\frac{|\xi|}{|k\ell|}  \frac{1}{(1+|\frac{\xi}{k}-t|)} \left(\frac{p_{k}(\xi)}{k^2}\right)^\frac14.
\end{align}
For $\mathcal{T}^{p,2}_N$, by the mean value theorem, there is $\xi'$ between $\eta$ and $\xi$ such that 
\begin{align}
\label{bd:p14etaxi}	\left| \left(\frac{p_k(\eta)}{k^2}\right)^\frac 14 -  \left(\frac{p_k(\xi)}{k^2}\right)^\frac 14 \right| 	\lesssim |\eta-\xi| \frac{1}{|k|(1+|\frac{\xi'}{k}-t|)}  \left(\frac{p_k(\xi')}{k^2}\right)^\frac14.
\end{align}
Therefore, the most dangerous term will appear in $\mathcal{T}^{p,1}_N$, since there is a loss of order $|\xi|/|\ell|$. Hence, we only only deal with $\mathcal{T}^{p,1}_N$.
By means of \eqref{bd:p14kell0} and \eqref{bd:p/pkk}, we have
\begin{align}
\label{bd:p14kell}\left| \left(\frac{p_\ell(\xi)}{\ell^2}\right)^\frac 14 -  \left(\frac{p_k(\xi)}{k^2}\right)^\frac 14 \right| \left(\frac{p_k(\eta)}{k^2}\right)^{-\frac14} |\ell, \xi| 
	&\lesssim  \mathbbm{1}_{k\neq \ell} \jap{\eta-\xi,k-\ell}^3 \frac{|\xi|^2}{|k|^2}\frac{1}{(1+|\frac{\xi}{k}-t|)}.
\end{align}
We now have to consider different cases, depending on intermediate, short and long times.

\diampar{Intermediate times} 
If $t\in I_{k,\eta}\cap I_{k,\xi}$, combining \eqref{bd:p14kell} with Lemma \ref{lemma:detw/wfar} and \eqref{bd:wexfaway}, we get
\begin{align}
\label{bd:tp1idf}\frac{ \left|\left(\frac{p_\ell(\xi)}{\ell^2}\right)^\frac 14 -  \left(\frac{p_k(\xi)}{k^2}\right)^\frac 14\right|}{\left(\frac{p_k(\eta)}{k^2}\right)^{\frac14}}  |\ell, \xi| \mathbbm{1}_{t\in I_{k,\eta}\cap I_{k,\xi}}\chi^I&\lesssim \mathbbm{1}_{k\neq \ell } \jap{k-\ell,\eta-\xi}^7 t^2 \sqrt{\frac{\de_t w_k(\eta)}{w_{k}(\eta)}}\sqrt{\frac{\de_t w_\ell(\xi)}{w_{\ell}(\xi)}},
\end{align}
where in the last line we used  $t\approx |\xi/k|$. Applying \eqref{bd:Ulow} we obtain
\begin{align}
	\label{bd:Tp1RNR}	\sum_{N>8} \skli \mathcal{T}^{p,1}_N \mathbbm{1}_{t\in I_{k,\eta}\cap I_{k,\xi}}\chi^I \lesssim t^2 \norm{\bU_{\neq}}_{\G^{\lambda,\sigma-6}} \norm{\sqrt{\frac{\de_t w}{w}}AZ}^2\lesssim \eps t^{\frac12}G_w[Z]\lesssim \delta G_w[Z].
\end{align}
 We now have to consider $t \notin I_{k,\eta}\cap I_{k,\xi}$. In this case, either $|\xi/k-t|\gtrsim |\xi/k^2|$ or $|\eta/k-t|\gtrsim |\eta/k^2|$. Hence, upon paying Sobolev regularity for the low frequencies, we can always recover a derivative in $v$ in \eqref{bd:p14kell}. More precisely, from \eqref{bd:p14kell}, using \eqref{bd:p/pkk} if necessary, notice that  
\begin{align*}
	\frac{ \left|\left(\frac{p_\ell(\xi)}{\ell^2}\right)^\frac 14 -  \left(\frac{p_k(\xi)}{k^2}\right)^\frac 14\right|}{\left(\frac{p_k(\eta)}{k^2}\right)^{\frac14}}  |\ell, \xi| \mathbbm{1}_{t\in I_{k,\eta}^c\cup I_{k,\xi}^c}\chi^I&\lesssim  \mathbbm{1}_{k\neq \ell} \jap{k-\ell,\eta-\xi}^5 |\xi|.
\end{align*}
Since $\TS\leq t$, then $|\xi|\leq t^{2-2s}|\xi|^s\leq t^{2-2s}|\ell,\xi|^\frac{s}{2}|k,\eta|^{\frac{s}{2}}$. Thus, again from   \eqref{bd:Ulow} we get
\begin{align}
	\label{bd:Tp1nRNR}	\sum_{N>8} \skli \mathcal{T}^{p,1}_N \mathbbm{1}_{t\in I_{k,\eta}^c\cup I^c_{\ell,\xi}}\chi^I \lesssim t^{2-2s}\norm{\bU_{\neq}}_{\G^{\lambda,\sigma-6}} \norm{|\nabla|^\frac{s}{2}AZ}^2\lesssim \delta G_\lambda[Z],
\end{align}
which is consistent with \eqref{bd:TNprop}.

\diampar{Short times} 
For $t \in I_{k,\xi}$, since $|\xi/k|^2\approx t^2 \lesssim t^{2-2s}|k,\eta|^{\frac{s}{2}}|\ell,\xi|^{\frac{s}{2}}$ from \eqref{bd:p14kell} and \eqref{bd:Ulow} we deduce 
\begin{align}
	\label{bd:Tp1SRNR}	\sum_{N>8} \skli \mathcal{T}^{p,1}_N \mathbbm{1}_{t\in I_{k,\xi}}\chi^S &\lesssim t^{2-2s} \norm{\bU_{\neq}}_{\G^{\lambda,\sigma-6}} \norm{|\nabla|^{\frac{s}{2}}AZ}^2 \lesssim \delta G_\lambda[Z].
\end{align}
When $t\in I_{k,\xi}^c$, the gain of one derivative in \eqref{bd:p14kell} is not enough. Thus we crucially use the bounds on $\Omega$. First, we notice that 
\begin{equation}
\frac{|\xi|}{|\ell|}\frac{|\frac{\xi}{k}-t+t|}{1+|\frac{\xi}{k}-t|}\leq \frac{|\xi|}{|\ell|}\left( \frac{t}{1+|\frac{\xi}{k}-t|}+1\right).
\end{equation}
Then, using the inequality above in  \eqref{bd:p14kell}, we get
\begin{align}
	\notag\mathcal{T}^{p,1}_{N}\mathbbm{1}_{t\in I_{k,\xi}^c}\chi^S\lesssim&\ \mathbbm{1}_{k\neq \ell}\mathbbm{1}_{t\in I_{k,\xi}^c}\chi^S( \mathcal{I}^1+\mathcal{I}^2),
\end{align}
where we define
\begin{align}\notag 
	\mathcal{I}^1&=\frac{|\xi|}{|\ell|}\frac{t}{1+|\frac{\xi}{k}-t|}|AZ|_\ell(\xi)_N|AZ|_k(\eta)\jap{k-\ell,\eta-\xi}^8|\widehat{\bU}|_{k-\ell}(\eta-\xi)_{<N/8},\\
		\notag \mathcal{I}^2&=\frac{|\xi|}{|\ell|}\frac{1}{(1+|\frac{\xi}{\ell}-t|)^\frac12}|A\widehat{\Omega}|_\ell(\xi)_N|AZ|_k(\eta)\jap{k-\ell,\eta-\xi}^8|\widehat{\bU}|_{k-\ell}(\eta-\xi)_{<N/8}.
\end{align}
Notice that to write $\mathcal{I}^2$ we have used the definition of $Z$.
To control $\mathcal{I}^1$ we observe that 
\begin{equation}
\label{bd:wow}
	\frac{|\xi|}{|\ell|}\frac{t}{1+|\frac{\xi}{k}-t|}\mathbbm{1}_{t\in I_{k,\xi}^c}\chi^S\lesssim \frac{|\xi|}{|\ell|}\frac{t}{(1+|\frac{\xi}{k}-t|)^\frac12}\mathbbm{1}_{t\in I_{k,\xi}^c}\chi^S\lesssim t\frac{|\xi|}{|\ell|}\frac{|k|}{|\xi|^\frac12}\chi^S\lesssim t^{2-2s}\jap{k-\ell}|k,\eta|^{\frac{s}{2}}|\ell,\xi|^{\frac{s}{2}}.
\end{equation}
In this way, 
\begin{equation}
	\sum_{N>8}\skli\mathcal{I}^1 \mathbbm{1}_{k\neq \ell}\mathbbm{1}_{t\in I_{k,\xi}^c}\chi^S \lesssim t^{2-2s}\norm{\bU_{\neq}}_{\G^{\lambda,\sigma-6}}\norm{|\nabla|^{\frac{s}{2}}AZ}^2\lesssim \delta G_\lambda[Z].
\end{equation}
Similarly, to control $\mathcal{I}^2$, if $t\in I_{\ell,\xi}$ then $|\xi/\ell|\lesssim t\lesssim t^{1-2s}|k,\eta|^\frac{s}{2}|\ell,\xi|^{\frac{s}{2}}$. If $t\in I_{\ell,\xi}^c$ we can argue as in \eqref{bd:wow}. Therefore, using \eqref{bd:Ulow}, we  have
\begin{align}
	\sum_{N>8}\skli\mathcal{I}^2 \mathbbm{1}_{k\neq \ell}\mathbbm{1}_{t\in I_{k,\xi}^c}\chi^S &\lesssim t^{1-2s}\norm{\bU_{\neq}}_{\G^{\lambda,\sigma-6}}\norm{|\nabla|^{\frac{s}{2}}A\Omega}\norm{|\nabla|^{\frac{s}{2}}AZ}
	&\lesssim\delta G_\lambda[Z] +\delta^{-1}\eps^2G_\lambda[\Omega],
\end{align}
which agrees with \eqref{bd:TNprop}.

\diampar{Long times} 
In this case, if $|\xi|\leq |\eta|$ we know that $|\xi/k-t|\gtrsim t$. If $|\xi|\geq |\eta|$, using \eqref{bd:p/pkk}  we have 
\begin{equation}
	\label{bd:Long}
	\frac{1}{1+|\frac{\xi}{k}-t|}\lesssim \jap{\eta-\xi}\frac{1}{1+|\frac{\eta}{k}-t|}\lesssim \jap{\eta-\xi}\frac{1}{\jap{t}}.
\end{equation}
Therefore, we always gain factor of times from the denominator in \eqref{bd:p14kell}. This implies 
\begin{equation}
\frac{|\xi|^2}{|k|^2}\frac{1}{(1+|\frac{\xi}{k}-t|)}\chi^L\lesssim \jap{\eta-\xi}^2t^{1-s}|k,\eta|^{\frac{s}{2}}|\ell,\xi|^{\frac{s}{2}}.
\end{equation}
Thus, using the bound above in \eqref{bd:p14kell} we get
\begin{align}
	\label{bd:Tp1LRNR}	\sum_{N>8} \skli \mathcal{T}^{p,1}_N \chi^L &\lesssim t^{1-s} \norm{\bU_{\neq}}_{\G^{\lambda,\sigma-6}} \norm{|\nabla|^{\frac{s}{2}}AZ}^2 \lesssim \frac{\eps}{t^{\frac12+s}}\norm{|\nabla|^{\frac{s}{2}}AZ}^2 \lesssim \eps G_\lambda[Z].
\end{align}
This shows that $\mathcal{T}^{p,1}_N$  produces the bound \eqref{bd:TNprop}. As a consequence, the bound on $T_N^{\Omega,p}$ is proven.

\bullpar{Bound on $T^A_N$}  
Turning to $T^A_N$, we write  the commutator in each of its components as 
 \begin{equation}
 [A,\bU_{<N/8}]=[\e^{\lambda |\cdot|^s},\bU_{<N/8}]\jap{\cdot}^\sigma J+\e^{\lambda|\cdot|^s}[\jap{}^\sigma,\bU_{<N/8}]J+\e^{\lambda |\cdot |^s}\jap{\cdot}^\sigma[J,\bU_{<N/8}]
 \end{equation}
and bound
\begin{align}\label{def:splitTNOm}
	|T_N^{\Omega,A}|\lesssim  \skli 	\mathcal{T}^1_N+\mathcal{T}^{2}_N+\mathcal{T}^J_N,
\end{align}
where
\begin{align}
\mathcal{T}^1_N &= \left(\frac{p_\ell(\xi)}{p_k(\eta)}\right)^\frac14\frac{|k|^{\frac12}}{|\ell|^{\frac12}} \left|\e^{\lambda(|k,\eta|^s-|\ell,\xi|^s)}-1\right|| |\ell,\xi||\widehat{\bU}_{k-\ell}(\eta-\xi)|_{<N/8}|AZ_\ell(\xi)|_{N}|AZ_k(\eta)| ,\\
\mathcal{T}^2_N &=\left(\frac{p_\ell(\xi)}{p_k(\eta)}\right)^\frac14\frac{|k|^{\frac12}}{|\ell|^{\frac12}}\e^{\lambda(|k,\eta|^s-|\ell,\xi|^s)}\left|\frac{\jap{k,\eta}^\sigma}{\jap{\ell,\xi}^\sigma}-1\right| |\ell,\xi||\widehat{\bU}_{k-\ell}(\eta-\xi)|_{<N/8}|AZ_\ell(\xi)|_{N}|AZ_k(\eta)| ,\\
\mathcal{T}^J_N &= \left(\frac{p_\ell(\xi)}{p_k(\eta)}\right)^\frac14\frac{|k|^{\frac12}}{|\ell|^{\frac12}}  \e^{\lambda(|k,\eta|^s-|\ell,\xi|^s)}\frac{\jap{k,\eta}^\sigma}{\jap{\ell,\xi}^\sigma}\left|\frac{J_k(\eta)}{J_\ell(\xi)}-1\right||\ell,\xi||\widehat{\bU}_{k-\ell}(\eta-\xi)|_{<N/8}|AZ_\ell(\xi)|_{N}|AZ_k(\eta)|.
\end{align}
That $\mathcal{T}^1_N $ and $\mathcal{T}^2_N$ satisfy bounds that comply with \eqref{bd:TNprop} follows from an argument  analogous to \cite{BM15}*{Section 5}, thanks to \eqref{bd:p/pgen} and the fact that $\eps t^\frac12\leq \delta$. We therefore only focus on the more problematic term $\mathcal{T}^J_N $.
It is convenient to split $\mathcal{T}_N^J$ as 
\begin{equation}
	\label{def:TNJ}
		\mathcal{T}_N^J= \mathcal{T}_N^{J}\mathbbm{1}_{t\leq \frac12 \min\{\sqrt{|\eta|},\sqrt{|\xi|}\}}+\mathcal{T}_N^J(\mathbbm{1}_{t\in D}+\mathbbm{1}_{t\in D^c}),
\end{equation}
where the \textit{difficult} domain is defined as
\begin{equation}
	\label{def:D}
	D=I_{k,\eta}\cap I_{k,\xi}\cap\{k\neq \ell\}\cap \{t\geq \frac12 \min\{\sqrt{|\eta|},\sqrt{|\xi|}\}\}.
\end{equation} 
Using Lemma \ref{lemma:commJ}, the bounds \eqref{bd:p/pgen} and \eqref{bd:bootUlow},  as in 
\cite{BM15}*{Section 5} we
get
\begin{align}
	\sum_{N>8}\skli \mathcal{T}_N^{J}\mathbbm{1}_{t\leq \frac12 \min\{\sqrt{|\eta|},\sqrt{|\xi|}\}}&\lesssim t^\frac12\norm{\bU}_{\G^{\lambda,\sigma-6}}\norm{|\nabla|^{\frac{s}{2}}AZ}^2 \lesssim \delta G_\lambda[Z],
\end{align}
which is consistent with  \eqref{bd:TNprop}.
We now turn our attention to the  terms where $t\geq \frac12 \min\{\sqrt{|\eta|},\sqrt{|\xi|}\}$. 

The term with $t\in D$ is the most delicate one. In this case, we cannot gain anything from the commutator. 
Notice that in this interval we may have a loss in the bound \eqref{bd:p/p} and \eqref{bd:JexTDs}. Combining \eqref{bd:p/p} with \eqref{bd:JexTDs} and Lemma \ref{lemma:detw/wfar} we get 
  \begin{align}
\left(\frac{p_{\ell}(\xi)}{p_k(\eta)}\right)^\frac14\frac{|k|^{\frac12}}{|\ell|^{\frac12}}&\e^{\lambda(|k,\eta|^s-|\ell,\xi|^s)}\frac{\jap{k,\eta}^\sigma}{\jap{\ell,\xi}^\sigma}\left|\frac{J_k(\eta)}{J_\ell(\xi)}-1\right||\ell,\xi|\mathbbm{1}_{t\in D}\notag\\
&\lesssim  |k,\eta|\frac{\eta}{|k|^2}\frac{1}{1+|\frac{\eta}{k}-t|}\jap{k-\ell,\eta-\xi}^3 \e^{c\lambda |k-\ell,\eta-\xi|^s}\notag\\
&\lesssim t^2\sqrt{\frac{\de_t w_\ell(\xi)}{w_\ell(\xi)}}\sqrt{\frac{\de_t w_k(\eta)}{w_k(\eta)}}\jap{k-\ell,\eta-\xi}^3 \e^{c\lambda |k-\ell,\eta-\xi|^s},
  \end{align}
where in the last line we have used that $t\approx |\eta/k|$.  Then, using \eqref{bd:towerA} and  \eqref{bd:Ulow} we have
\begin{equation}
	\label{bd:TNJrr}
		\skli \mathcal{T}_N^{J}\mathbbm{1}_{t\in D}\lesssim t^2 \norm{\bU_{\neq}}_{\G^{\lambda,\sigma-6}}\norm{\sqrt{\frac{\de_t w}{w}}\tA Z}^2\lesssim \delta G_w[Z],
\end{equation}
which works well with \eqref{bd:TNprop}.
For the remaining terms we need to consider two subcases, namely $|\ell|>100|\xi|$ and $|\ell|<100|\xi|$. In the first scenario, 
using \eqref{lemma:maxgrowth} and  \eqref{bd:p/pgen}, we can repeat the argument in \cite{BM15}*{Section 5} and obtain
\begin{equation}
	\label{bd:TNJothers1}
		\skli \mathcal{T}_N^{J}\mathbbm{1}_{\{t\in D^c\}\cap\{|\ell|>100|\xi|\}}
		\lesssim \delta G_\lambda[Z].
\end{equation}
When $|\ell|\leq 100|\xi|$, we can again ignore any gain from the commutator. Indeed, for the terms we are considering we can always apply \eqref{bd:Jeximp}. Then, if $t\in I_{k,\eta}\cap I_{\ell,\xi}^c$, by \eqref{bd:p/p} and  since $|\ell,\xi|\lesssim |\xi|\lesssim |\eta|$, we have 
\begin{align*}
	\left(\frac{p_{\ell}(\xi)}{p_k(\eta)}\right)^\frac14\frac{|k|^{\frac12}}{|\ell|^{\frac12}}&\e^{\lambda(|k,\eta|^s-|\ell,\xi|^s)}\frac{\jap{k,\eta}^\sigma}{\jap{\ell,\xi}^\sigma}\left|\frac{J_k(\eta)}{J_\ell(\xi)}-1\right||\ell,\xi|\mathbbm{1}_{\{t\in I_{k,\eta}\cap I_{\ell,\xi}^c\cap D^c\}\cap \{|\ell|\leq 100|\xi|\}}\\
	&\lesssim \jap{k-\ell,\eta-\xi}^\frac32 \e^{c\lambda|k-\ell,\eta-\xi|^s}\frac{|\eta|}{|k|}\frac{|\eta|^\frac12}{(1+|\frac{\eta}{k}-t|)^\frac12}.
\end{align*}
We can now apply the same reasoning done for the terms $\mathcal{T}^{p,1}_N\mathbbm{1}_{t\in I_{k,\eta}\cap I_{\ell,\xi}^c}$, see e.g. \eqref{bd:tp1idf}. Namely, as in \eqref{bd:Tp1RNR}, \eqref{bd:Tp1SRNR} and \eqref{bd:Tp1LRNR} we get 
\begin{equation}
	\label{bd:TNJnrr}
	\sum_{N>8}	\skli \mathcal{T}_N^{J}\mathbbm{1}_{\{t\in D^c\cap I_{k,\eta}\cap I_{\ell,\xi}^c\}\cap \{|\ell|\leq 100|\xi|\}} \lesssim \delta (G_w[Z]+G_\lambda[Z]).
\end{equation}
For the remaining terms, since we always have $t\in I_{k,\eta}^c \cup I_{\ell,\xi}$ we do not lose anything from  $p_\ell(\xi)/p_k(\eta)$. Therefore, we simply exploit the fact that $|\xi|\lesssim t^2$, to get $|\ell,\xi|\lesssim t^{2-2s}|k,\eta|^\frac{s}{2}|\ell,\xi|^{\frac{s}{2}}$, which implies 
\begin{equation}
		\label{bd:TNJothers2}
		\sum_{N>8}\skli \mathcal{T}_N^{J}\mathbbm{1}_{\{t\in D^c\cap I_{k,\eta}^c\}\cap \{|\ell|\geq 100|\xi|\}}		\lesssim \eps t^\frac12 \frac{1}{t^{2s}}\norm{|\nabla|^{\frac{s}{2}}AZ}^2\lesssim \delta G_\lambda[Z].
\end{equation}
As a consequence, $T^A_N$ is bounded as in  \eqref{bd:TNprop}, as we wanted.

\bullpar{Bound on $T_N^m$}
Writing this term explicitly and using that $|m|\approx 1$ we find 
\begin{align}
T_N^m\lesssim  \skli \mathcal{T}_N^{m,1}+\mathcal{T}_N^{m,2},
\end{align} 
where
\begin{align}
\mathcal{T}_N^{m,1} &= \left|m^{-1}_k(\xi)-m^{-1}_\ell(\xi)\right| |\ell,\xi||\widehat{\bU}|_{k-\ell}(\eta-\xi)_{<N/8}  |AZ|_\ell(\xi)_{N}|A{Z}|_k(\eta),\\
\mathcal{T}_N^{m,2} &= \left|m^{-1}_k(\eta)-m_k^{-1}(\xi)\right||\ell,\xi||\widehat{\bU}|_{k-\ell}(\eta-\xi)_{<N/8} |AZ|_\ell(\xi)_{N}|A{Z}|_k(\eta) .
\end{align} 
Again, we want to show that these terms satisfy \eqref{bd:TNprop}.
We are going to proceed in analogy to what was done for the term $T_N^{\Omega,p}$. Recalling the definition of $m$ \eqref{def:m}, by the mean value theorem we have that  there is $\xi'$ between $\eta$ and $\xi$ such that\begin{align}
	\left|m^{-1}_k(\eta)-m_k^{-1}(\xi)\right|
	&\lesssim |\eta-\xi| \frac{1}{|k|(1+(t-\frac{\xi'}{k})^2)}.
\end{align}
Thus, for $\mathcal{T}_N^{m,2}$ we can repeat the arguments done to handle $T_N^{\Omega,p}$. For what concerns $\mathcal{T}_N^{m,1}$, since $|e^x-1|\leq |x|e^{x}$, notice that
\begin{equation}
\left|m^{-1}_k(\xi)-m^{-1}_\ell(\xi)\right|\lesssim\left|\arctan\left(\frac{\xi}{k}-t\right)-\arctan\left(\frac{\xi}{\ell}-t\right)\right|=\left|\arctan\left(\frac{\frac{\xi}{k\ell}(k-\ell)}{1+(\frac{\xi}{k}-t)(\frac{\xi}{\ell}-t)}\right)\right|,
\end{equation}
where the last equality follows by trigonometric identities. Now we claim that 
\begin{equation}
\label{bd:commmkell}
\left|m^{-1}_k(\xi)-m^{-1}_\ell(\xi)\right|\lesssim |k-\ell|\frac{|\xi|}{|k\ell|}\left(\frac{1}{1+|\frac{\xi}{k}-t|}+\frac{1}{1+|\frac{\xi}{\ell}-t|}\right).
\end{equation}
To prove the inequality above, if $|1+(\xi/k-t)(\xi/\ell-t)|\geq 16 |\xi/(k\ell)||k-\ell|$ and $|\xi/k-t|\geq 16, \ |\xi/\ell-t|\geq 16$, then \eqref{bd:commmkell} follows by Taylor's expansion of the $\arctan$. If $|1+(\xi/k-t)(\xi/\ell-t)|\leq 16 |\xi/(k\ell)||k-\ell|$ and $|\xi/k-t|\geq 16, \ |\xi/\ell-t|\geq 16$, then it is enough to use the uniform boundedness of the $\arctan$ and multiply and divide by $1+|\xi/k-t|$. When $|\xi/k-t|\leq 16$ or $|\xi/\ell-t|\leq 16$, \eqref{bd:commmkell} follows by $|\arctan(x)-\arctan(y)|\leq |x-y|$. From \eqref{bd:commmkell}, we can now obtain an estimate as in\eqref{bd:p14kell} and argue as done to control $\mathcal{T}^{p,1}_N$ to get 
\begin{equation}
\sum_{N>8}\skli\mathcal{T}^{m,1}_N\lesssim \delta (G_w[Z]+G_{\lambda}[Z]).
\end{equation}

 \subsubsection{{Bound on} $T_N^{\Theta,1}$}
\label{subsec:TTheta}
Using \eqref{eq:trivcomm}, we rewrite $T_N^{\Theta,1}=T_N^{\Theta,A}+T_N^{\Theta,p}+T_N^{\Theta,m}+T_N^{\Theta,z}$ where
\begin{align}
	T_N^{\Theta,A} &=\left|\jap{\mathcal{F}(m^{-1} \left(\frac{p}{k^2}\right)^{\frac{1}{4}} k([mA,\bU_{<N/8}]\cdot \nabla (i\beta \Theta)),AQ}\right|,\\
	T_N^{\Theta,p} &= \left|\jap{\mathcal{F}(m^{-1}k\left[ \left(\frac{p}{k^2}\right)^{\frac{1}{4}} ,\bU_{<N/8}\right]\cdot \nabla (imA\beta \Theta)_N),AQ}\right|,\\
	T_N^{\Theta,m}& =\left|\jap{\mathcal{F}([m^{-1},\bU_{<N/8}]\cdot \nabla (mAQ)_N),AQ}\right|,\\
	T_N^{\Theta,z}&=\left|\jap{\mathcal{F}(m^{-1}[k,\bU_{<N/8}]\cdot \nabla \left(imA\beta \left(\frac{p}{k^2}\right)^\frac14\Theta_N\right)),AQ}\right|.\end{align}
The terms $T_N^{\Theta,A}$ and $T_N^{\Theta,m}$ can be bounded with exactly the same arguments used for $T_N^{\Omega,A}$ and $T_N^{\Omega,m}$. The term $T_N^{\Theta,p}$ is equivalent to $T_N^{\Omega,p}$ with the role of $(k,\eta)$ and $(\ell,\xi)$ switched. Just notice that the extra factor of $k$ out of the commutator can be easily moved onto the high-frequency part by paying Sobolev regularity on $\bU$. In addition, we need to replace the bounds on $\Omega$ with the ones for $\nabla_L\Theta$. We do not detail more the bounds for these three terms. 
On the other hand, we present the bounds for $T_N^{\Theta,z}$. Here, we again need to use the bounds available for $\nabla_L\Theta$.

\bullpar{Bound on $T_N^{\Theta,z}$}
Writing explicitly this term in the Fourier space and using that $m\approx 1$ we have 
\begin{align*}
	T_N^{\Theta,z}\lesssim\skli \mathcal{T}^{\Theta,z}_N= \skli  |k-\ell||\widehat{\bU}|_{k-\ell}(\eta-\xi)_{<N/8}|\ell,\xi|\left(\frac{p_\ell(\xi)}{\ell^2}\right)^{\frac14}|A\widehat{\Theta}|_\ell(\xi)_N|A{Q}|_k(\eta)\dd \eta \dd \xi.
\end{align*}
Observe that 
\begin{equation}
	\label{eq:trividTheta}
	\left(\frac{p_\ell(\xi)}{\ell^2}\right)^{\frac14}|\ell,\xi|=\frac{|\ell,\xi|}{|\ell|^\frac12 p^{\frac14}_\ell(\xi)}p^{\frac12} _\ell(\xi)\lesssim \jap{\frac{\xi}{\ell}}\frac{1}{(1+|\frac{\xi}{\ell}-t|)^\frac12}p^{\frac12} _\ell(\xi).
\end{equation}
Again, our goal is to bound the above term as in \eqref{bd:TNprop}.

\diampar{Intermediate times}
When   $t\in I_{\ell,\xi}$ we need to use  Lemma \ref{lemma:detw/wfar}. In addition, observe that in this interval $\jap{\xi/\ell}\lesssim t^{1-\frac{s}{2}}|k,\eta|^{\frac{s}{2}}$. Thus, combining   \eqref{bd:Ulow} with the previous inequality we have 
\begin{align}
	\sum_{N>8}\skli \mathcal{T}^{\Theta,z}_N\mathbbm{1}_{t\in I_{\ell,\xi}}\chi^I&\lesssim t^{1-\frac{s}{2}}\norm{\bU_{\neq}}_{\G^{\lambda,\sigma-6}}\norm{\sqrt{\frac{\de_tw}{w}}\tA\nabla_L\Theta}\norm{|\nabla|^{\frac{s}{2}}AQ}\notag\\
	&\lesssim\delta G_\lambda[Q]+ \delta^{-1}\eps^2 G_w[\nabla_L\Theta].
\end{align}
Then, if $|\xi|\geq |\ell|$ and $t\in I_{\ell,\xi}^c$, we know that 
\begin{equation*}
	\jap{\frac{\xi}{\ell}}\frac{1}{(1+|\frac{\xi}{\ell}-t|)^\frac12} \mathbbm{1}_{t\in I_{\ell,\xi}^c}\chi^I \lesssim \frac{|\xi|}{|\ell|}\frac{|\ell|}{|\xi|^{\frac12}}\lesssim t^{1-2s} |k,\eta|^{\frac{s}{2}}|\ell,\xi|^{\frac{s}{2}},
\end{equation*}
where in the last inequality we used $\TS\leq t$. When $|\xi|\leq |\ell|$ one simply observes that since $t\leq \TL$ then $1\leq t^{-s}|k,\eta|^{\frac{s}{2}}|\ell,\xi|^{\frac{s}{2}}$. In this way, since $s<1$ from \eqref{eq:trividTheta} and \eqref{bd:Ulow} we have 
 \begin{align}
 	\notag \sum_{N>8}\skli \mathcal{T}^{\Theta,z}_N\mathbbm{1}_{t\in I_{\ell,\xi}^c}\chi^I&\lesssim t^{1-2s}\norm{\bU_{\neq}}_{\G^{\lambda,\sigma-6}}\norm{|\nabla|^\frac{s}{2}A\nabla_L\Theta}\norm{|\nabla|^{\frac{s}{2}}AQ},\\
 	\label{bd:TthetazNR}&\lesssim \delta G_\lambda[Q]+ \delta^{-1}\eps^2 G_\lambda[\nabla_L\Theta],
 \end{align}
 as needed for   \eqref{bd:TNprop}.
 
\diampar{Short times}
 If $t\in I_{\ell,\xi}$ then $\jap{\xi/\ell}\approx t \leq t^{1-2s}|k,\eta|^{\frac{s}{2}}|\ell,\xi|^{\frac{s}{2}}$ so that we can repeat the argument done to obtain \eqref{bd:TthetazNR}. When $t\in I_{\ell,\xi}^c$ then 
\begin{equation*}
	\jap{\frac{\xi}{\ell}}\frac{1}{(1+|\frac{\xi}{\ell}-t|)^\frac12} \mathbbm{1}_{t\in I_{\ell,\xi}^c}\chi^S \lesssim \max\{1,\frac{|\xi|}{|\ell|}\frac{|\ell|}{|\xi|^{\frac12}}\}\lesssim |k,\eta|^{\frac{s}{2}}|\ell,\xi|^{\frac{s}{2}},
\end{equation*}
where we also used $s>1/2$. Hence
 $$
	\sum_{N>8}\skli \mathcal{T}^{\Theta,z}_N\mathbbm{1}_{t\in I_{\ell,\xi}^c}\chi^S\lesssim \norm{\bU_{\neq}}_{\G^{\lambda,\sigma-6}}\norm{|\nabla|^\frac{s}{2}A\nabla_L\Theta}\norm{|\nabla|^{\frac{s}{2}}AQ}
\lesssim\delta G_\lambda[Q]+ \delta^{-1}\eps^2 G_\lambda[\nabla_L\Theta],
$$
consistent with   \eqref{bd:TNprop}.

\diampar{Long times} 
Since $t\geq \TL$, arguing as in \eqref{bd:Long} we get 
\begin{equation*}
	\jap{\frac{\xi}{\ell}}\frac{1}{(1+|\frac{\xi}{\ell}-t|)^\frac12}\chi^L \lesssim t^{-\frac12}\jap{\xi} \chi^L\lesssim t^{\frac12-s}|k,\eta|^{\frac{s}{2}}|\ell,\xi|^{\frac{s}{2}}.
\end{equation*}
This implies 
 \begin{align}
	\sum_{N>8}\skli \mathcal{T}^{\Theta,z}_N\chi^L&\lesssim t^{\frac12-s}\norm{\bU_{\neq}}_{\G^{\lambda,\sigma-6}}\norm{|\nabla|^\frac{s}{2}A\nabla_L\Theta}\norm{|\nabla|^{\frac{s}{2}}AQ}\lesssim \delta G^{Q}_\lambda[Q]+ \delta^{-1}\eps^2 G_\lambda[\nabla_L\Theta].
\end{align}
implying  \eqref{bd:TNprop} for this term as well.

\subsubsection{{Bound on} $T_N^{err}$}
We split this term as 
\begin{align}
T^{err}_N\lesssim  \skli \mathcal{T}_N^{err,1}+\mathcal{T}_N^{err,2},
\end{align} 
where
\begin{align}
\mathcal{T}_N^{err,1} &= \left|\frac{\de_tp_k(\xi)}{|k|p^{\frac12}_k(\xi)}-\frac{\de_tp_\ell(\xi)}{|\ell|p^{\frac12}_\ell(\xi)}\right| |\ell,\xi||\widehat{\bU}|_{k-\ell}(\eta-\xi)_{<N/8}  |AZ|_\ell(\xi)_{N}|AQ|_k(\eta),\\
\mathcal{T}_N^{err,2} &= \left|\frac{\de_tp_k(\eta)}{|k|p^{\frac12}_k(\eta)}-\frac{\de_tp_k(\xi)}{|k|p^{\frac12}_k(\xi)}\right||\ell,\xi||\widehat{\bU}|_{k-\ell}(\eta-\xi)_{<N/8} |AZ|_\ell(\xi)_{N}|AQ|_k(\eta) .
\end{align} 
To control $T^{err,1}_N$, we first notice that 
\begin{align}
\left|\frac{\de_tp_k(\xi)}{|k|p^{\frac12}_k(\xi)}-\frac{\de_tp_\ell(\xi)}{|\ell|p^{\frac12}_\ell(\xi)}\right|&=\left|\frac{\frac{\xi}{k}-t}{(1+(\frac{\xi}{k}-t)^2)^{\frac12}}-\frac{\frac{\xi}{\ell}-t}{(1+(\frac{\xi}{\ell}-t)^2)^{\frac12}}\right|.
\end{align}
Then, using the bound 
$$
\left|\frac{a}{(1+a^2)^{\frac12}}-\frac{b}{(1+b^2)^{\frac12}}\right|=\left|\frac{(a-b)(1+b^2)^\frac12-b((1+a^2)^\frac12-(1+b^2)^\frac12)}{(1+a^2)^\frac12(1+b^2)^\frac12}\right|\lesssim |a-b|\frac{1}{(1+a^2)^\frac12},
$$
we get 
\begin{align}
\left|\frac{\de_tp_k(\xi)}{|k|p^{\frac12}_k(\xi)}-\frac{\de_tp_\ell(\xi)}{|\ell|p^{\frac12}_\ell(\xi)}\right|&\lesssim |k-\ell|\frac{|\xi|}{|k\ell|}\frac{1}{1+|\frac{\xi}{k}-t|},
\end{align}
which means we can repeat the arguments done for $T^{p,1}_{N}$. For $T^{err,2}_N$, since 
\begin{equation}
	\de_\eta \left(\frac{\de_t p}{|k|p^\frac12}\right)= 2\frac{1}{k(1+|\frac{\eta}{k}-t|^2)^{\frac32}},
\end{equation}
we can again apply the mean value theorem. Overall, we obtain
\begin{equation}
	T^{err}_N\lesssim \delta\sum_{j\in \{\lambda,w\}}(G_j[Z]+G_j[Q]).
\end{equation}
Therefore, the proof  of  \eqref{bd:TNprop} is completed.

\subsection{Reaction nonlinearities}\label{sec:ReactionEn}
We will decompose a function $F=F(t,k,\ell,\eta,\xi)$ as
\begin{align}\label{def:resnonresdecom}
F=F^{(R,R)}+F^{(NR,R)}+F^{(R,NR)}+F^{(NR,NR)}+F^S+F^L,
\end{align}
where
\begin{alignat}{3}
&F^{(R,R)}=F \chi^I\mathbbm{1}_{t\in I_{k,\eta}\cap I_{\ell,\xi}}, &\qquad  &F^{(NR,R)}=F \chi^I \mathbbm{1}_{t\in I_{k,\eta}^c\cap I_{\ell,\xi}}, &\qquad  &F^{(R,NR)}=F \chi^I\mathbbm{1}_{t\in I_{k,\eta}\cap I_{\ell,\xi}^c}, \\
&F^{(NR,NR)}=F \chi^I\mathbbm{1}_{t\in I_{k,\eta}^c\cap I_{\ell,\xi}^c}, &\qquad  &F^S=F \chi^S,  &\qquad  &F^L=F \chi^L.
\end{alignat}
This decomposition will be used in essentially all the reaction terms appearing in \eqref{de:react1}-\eqref{de:react5}. 
Besides the distinction between intermediate, short and long times, among the intermediate times we need to separate the \emph{resonant} (R) versus the \emph{non-resonant} (NR) interactions. 
As we shall see, the hardest terms to treat are those of the form $R^{\Omega,i}_N$ in  \eqref{de:react1}-\eqref{de:react2}, on which the toy model has been constructed. The terms
$R^{\Theta,i}_N$ in \eqref{de:react3}-\eqref{de:react4}  will be simpler to handle.  

The goal of this section is to prove that the reaction term satisfies the bound \eqref{bd:RNprop}.
Recall that throughout this section, on the support of the integral we have \eqref{bd:freqsupp}.

\subsubsection{{Bound on} $R^{\Omega,i}_N$}\label{sec:RZiN}
The bounds for $R^{\Omega,1}_N$ and $R^{\Omega,2}_N$ are analogous  since $|\de_t p|/(|k|p^{\frac12})\leq2$. We will then consider just the first one. We  split this term as 
\begin{align}
	R_N^{\Omega,1}
	\leq  R^{\Omega}_{N,\Psi}+{R}^{\Omega}_{N,\dot{v}}+{R}^{\Omega}_{N,\delta}+{R}^{\Omega}_{N,com},
\end{align} 
where we use that $\bU=v'\nabla^\perp \Psi_{\neq}+(0,\dot{v})$ and we define 
\begin{align}
 &R^{\Omega}_{N,\Psi}=\skli |AZ|_k(\eta)(Ap^{-\frac14})_k(\eta)|k|^{\frac12}|\eta\ell-k\xi||\widehat{\Psi}_{\neq}|_\ell(\xi)_N|\widehat{\Omega}|_{k-\ell}(\eta-\xi)_{<N/8}\dd\eta \dd\xi 	\label{def:RZPsi}\\
&R^{\Omega}_{N,\delta}=\skli |AZ|_k(\eta)(Ap^{-\frac14})_k(\eta)|k|^{\frac12}|\mathcal{F}(h\nabla^\perp\Psi_{\neq})|_\ell(\xi)_N|\widehat{\nabla\Omega}|_{k-\ell}(\eta-\xi)_{<N/8}\dd\eta \dd\xi 	\label{def:RZdelta}\\
&R^{\Omega}_{N,\dot{v}}=\skli |AZ|_k(\eta)(Ap^{-\frac14})_k(\eta)|k|^{\frac12}|\dot{v}|(\xi)_N|\widehat{\de_v\Omega}|_{k}(\eta-\xi)_{<N/8}\dd\eta \dd\xi \label{def:ROmegaNdotv}\\
&R^{\Omega}_{N,com}=\skli |AZ|_k(\eta)|\bU|_\ell(\xi)_N |A\nabla Z|_{k-\ell}(\eta-\xi)_{<N/8}\dd\eta \dd\xi.	\label{def:ROmegaNcom}
\end{align}
The main contribution will be the one given by $R^{\Omega}_{N,\Psi}$ and the term $R^{\Omega}_{N,\delta}$ can be considered, roughly speaking, as a perturbation of it. The term $R^{\Omega}_{N,com}$ comes from the commutator which we had to introduce to deal with the transport nonlinearities. When the velocity is at high frequencies, we do not need to gain anything from the commutator and we can deal with this term separately.

\bullpar{Bound on ${R}_{N,\Psi}^\Omega$}
 In view of the notation introduced in \eqref{def:resnonresdecom}, we split the term as 
\begin{equation}
	{R}_{N,\Psi}^\Omega=\skli \mathcal{R}_{N,\Psi}^{\Omega,(R,R)}+\mathcal{R}_{N,\Psi}^{\Omega,(NR,R)}+\mathcal{R}_{N,\Psi}^{\Omega,(R,NR)}+\mathcal{R}_{N,\Psi}^{\Omega,(NR,NR)}+\mathcal{R}_{N,\Psi}^{\Omega,S}+\mathcal{R}_{N,\Psi}^{\Omega,L},
\end{equation} 
and bound each term in a way that is consistent with \eqref{bd:RNprop}. 

\diampar{Bound on $\mathcal{R}_{N,\Psi}^{\Omega,(R,NR)}$}   
First of all, observe that on the support of the integral, we have
\begin{equation}
	\label{bd:frRNRr}
	|k|\leq \frac14|\eta|, \qquad A_k(\eta)\lesssim \tA_k(\eta),\qquad |\eta\ell-k\xi|\leq |k,\eta||k-\ell,\eta-\xi|\lesssim  |\eta||k-\ell,\eta-\xi|.
\end{equation}
Then, since $t\geq \TS$,  from \eqref{bd:detw/w} we have 
\begin{equation*}
	p^{-\frac14}_k(\eta)|k|^\frac12\mathbbm{1}_{t\in I_{k,\eta}}\chi^I\lesssim \sqrt{\frac{\de_tw_k}{w_k}}(\eta).
\end{equation*} 
 Using this and appealing to \eqref{bd:Jexgen}, \eqref{bd:frRNRr}  we deduce   
\begin{align}
	(Ap^{-\frac14})_{k}(\eta)|k|^\frac12 |\eta\ell-k\xi|\mathbbm{1}_{t\in I_{k,\eta}\cap I^c_{\ell,\xi}}\chi^I&\lesssim
\sqrt{\frac{\de_tw_k(\eta)}{w_k(\eta)}}\frac{|\eta|^\frac32}{|k|(1+|t-\frac{\eta}{k}|)^{\frac12}}\tA_\ell(\xi)\e^{c\lambda|k-\ell,\eta-\xi|^s}\notag\\
	\label{bd:RRNRk}&\lesssim\frac{|\eta|^\frac32}{|k|^\frac32}\frac{\de_tw_k(\eta)}{w_k(\eta)}p^{-\frac34}_\ell(\xi)|\ell|^\frac12 (p^\frac34\tA)_\ell(\xi)\e^{\lambda|k-\ell,\eta-\xi|^s}.
\end{align} 
Since $t\in I_{\ell,\xi}^c$, we observe that 
\begin{align}
\left(\frac{|\eta|}{|k|}\right)^\frac32p^{-\frac34}_\ell(\xi) \mathbbm{1}_{t\in I_{\ell,\xi}^c}\lesssim\left(\frac{|\eta|}{|k|}\right)^\frac32 \jap{\frac{\xi}{\ell}}^{-\frac32}\lesssim \jap{k-\ell,\eta-\xi}^{\frac32}.\label{eq:rnrbdd1}
\end{align}
Combining \eqref{bd:wexfaway},  \eqref{bd:RRNRk} , \eqref{eq:rnrbdd1}, \eqref{bd:ellipticGZw} and the bootstrap hypothesis \eqref{boot:En}
we get 
\begin{align}
\notag \sum_{N>8}\skli \mathcal{R}^{\Omega,(R,NR)}_{N,\Psi}&\lesssim \norm{\sqrt{\frac{\de_tw}{w}}\tA Z}\norm{\sqrt{\frac{\de_t w}{w}}\widetilde{A}(-\Delta_L)^\frac34|\de_z|^\frac12 \Psi}\norm{A\Omega}_{\G^{\lambda,\sigma-5}}\\
&\lesssim \delta (G_w[Z]+G^\eps_{elliptic}),
\end{align}
as required by \eqref{bd:RNprop}.

\diampar{Bound on $\mathcal{R}^{\Omega,(NR,R)}_{N,\Psi}$} 
For this term we know  that $4|\ell|^2\leq |\xi|$. Hence, from \eqref{bd:freqsupp} we deduce
\begin{align}
\label{bd:stupid}
|k|\leq |\ell|+\frac{3}{16}|\ell,\xi|\leq \frac{31}{64}|\xi|\leq \frac{31}{624}(|k|+|\eta|)\Longrightarrow |k|\leq \frac{31}{593}|\eta|.
\end{align}
Hence, we always have $A\lesssim \tA$.
We now have to exploit the fact that when exchanging $(Ap^{-\frac14})_k(\eta)$ with $(Ap^{-\frac14})_\ell(\xi)$ we gain derivatives in $v$, namely $1/2$ from $A$ and $1/2$ from $p^{-\frac14}$. More precisely, since $t\in I_{k,\eta }^c\cap I_{\ell,\xi}$, by \eqref{bd:p/p} and \eqref{bd:Jexgood} we get 
\begin{equation}
	\mathbbm{1}_{t\in I_{k,\eta}^c\cap I_{\ell,\xi}}(Ap^{-\frac14})_k(\eta)|k|^\frac12\lesssim  \frac{|\ell|^2(1+|t-\frac{\xi}{\ell}|)}{|\xi|}|\ell|^\frac12(Ap^{-\frac14})_\ell(\xi)\e^{c\lambda |k-\ell,\eta-\xi|^s}.
\end{equation}
Appealing to Lemma \ref{lemma:detw/wfar} and the fact that $|\eta\ell-k\xi|\lesssim |\ell,\xi||k-\ell,\eta-\xi|\lesssim |\xi||k-\ell,\eta-\xi|$, we have
 \begin{align}
	\mathbbm{1}_{t\in I_{k,\eta}^c\cap I_{\ell,\xi}}(Ap^{-\frac14})_k(\eta)|k|^\frac12|\eta \ell-k\xi|&\lesssim \frac{|\ell|}{p^{\frac12}_\ell(\xi)}|\ell|^\frac12(\tA p^{\frac34})_{\ell}(\xi)\e^{c\lambda |k-\ell,\eta-\xi|^s}\notag\\
	&\lesssim\frac{\de_t w_\ell}{w_\ell}(\xi) |\ell|^\frac12(\tA p^{\frac34})_\ell(\xi)\e^{c\lambda |k-\ell,\eta-\xi|^s}.
\end{align}
Combining the inequality above with \eqref{bd:wexfaway}, using the bootstrap hypothesis \eqref{boot:En} and the elliptic estimate \eqref{bd:ellipticGZw} we get
\begin{align*}
	\sum_{N>8}\skli \mathcal{R}^{\Omega,(NR,R)}_{N,\Psi}&\lesssim  \norm{\sqrt{\frac{\de_t w}{w}}\tA Z}\norm{ \sqrt{\frac{\de_t w}{w}}\tA(-\Delta_L)^{\frac34}|\de_z|^\frac12\Psi}\norm{\Omega}_{\G^{\lambda,\sigma-5}}\lesssim\delta (G_w[Z]+ G^\eps_{elliptic}),
\end{align*}
whence proving \eqref{bd:RNprop}.

\diampar{Bound on $\mathcal{R}_{N,\Psi}^{\Omega,(NR,NR)}$}
For this term, since $t\in I_{k,\eta}^c\cap I_{\ell,\xi}^c$ from \eqref{bd:Jeximp} and \eqref{bd:p/p} we deduce
\begin{equation}
	\label{bd:RNRNRp54}
	\mathbbm{1}_{t\in I_{k,\eta}^c\cap I_{\ell,\xi}^c}(Ap^{-\frac14})_k(\eta)|k|^\frac12|\eta\ell-k\xi|\lesssim |k,\eta|^{\frac{s}{2}}|\ell,\xi|^{1-\frac{s}{2}} p^{-1}_{\ell}(\xi)|\ell|^\frac12(p^\frac34A)_\ell(\xi)\e^{c\lambda |k-\ell,\eta-\xi|^s}.
\end{equation}
Appealing to \eqref{bd:ketapc}, the bootstrap hypothesis \eqref{boot:En}  and \eqref{bd:ellipticGZw}, from the inequality above we  then deduce that 
\begin{align}
	\sum_{N>8}\skli \mathcal{R}^{\Omega,(NR,NR)}_{N,\Psi}&\lesssim  \frac{1}{t^{s}} \norm{|\nabla|^\frac{s}{2}AZ}\norm{  \frac{|\nabla|^{\frac{s}{2}}}{\jap{t}^s}A(-\Delta_L)^{\frac34}|\de_z|^\frac12\Psi}\norm{\Omega}_{\G^{\lambda,\sigma-5}}\lesssim \delta (G_\lambda[Z]+ G^\eps_{elliptic}),\label{bd:RNRNR>}
\end{align}
as needed for  \eqref{bd:RNprop}.

\diampar{Bound on $R^{\Omega,(R,R)}_{N,\Psi}$}
When $t\in I_{k,\eta}\cap I_{\ell,\xi}$ we necessarily have $4|k|\leq |\eta|$ and $4|\ell|\leq |\xi|$, meaning that $|k,\eta|\approx |\eta|$ and $|\ell,\xi|\approx |\xi|$. Then, on the support of the integrand $|k,\eta|\approx|\ell,\xi|$ so that $|\eta|\approx |\xi|$.
Therefore we can apply the trichotomy Lemma \ref{lemma:trichotomy}. If case (b) holds, then we repeat the same argument done for $R^{\Omega,(NR,NR)}_{N,\Psi}$ and we omit the details. 

If we are in the case (a), namely $k=\ell$, then appealing to \eqref{bd:p/pkk} and \eqref{bd:Jeximp} we get 
\begin{equation}
	(Ap^{-\frac14})_k(\eta)|k|^\frac12|\eta\ell-k\xi|\mathbbm{1}_{\{t\in I_{k,\eta}\cap I_{\ell,\xi}\}\cap \{ k=\ell\}}\lesssim |\ell||\eta-\xi|p^{-1}_{\ell}(\xi)|\ell|^\frac12 (Ap^{\frac34})_\ell (\xi) \e^{c\lambda |\eta-\xi|^s}.
\end{equation}
Now observe that since $t\geq \TS$, \eqref{bd:detw/w} implies
\begin{equation*}
	|\ell|p^{-1}_{\ell}(\xi)\chi^I \mathbbm{1}_{t\in I_{\ell,\xi}}\leq \frac{1}{|\ell|(1+|t-\frac{\xi}{\ell}|^2)}\chi^I\mathbbm{1}_{t\in I_{\ell,\xi}}\lesssim \frac{\de_t w_\ell}{w_\ell}(\xi).
\end{equation*}
We also have $A_{k}(\eta)\lesssim \widetilde{A}_{k}(\eta)$ and $A_{\ell}(\xi)\lesssim \widetilde{A}_{\ell}(\xi)$. 
Using \eqref{bd:wexfaway} and \eqref{boot:En} we deduce 
\begin{align*}
	\sum_{N>8}\skli \mathcal{R}^{\Omega,(R,R)}_{N,\Psi}\mathbbm{1}_{k=\ell}&\lesssim \norm{\sqrt{\frac{\de_t w}{w}}\tA Z}\norm{\sqrt{\frac{\de_tw}{w}}\widetilde{A}(-\Delta_L)^{\frac34}|\de_z|^\frac12\Psi}\norm{\Omega}_{\G^{\lambda,\sigma-4}}\\
		&\lesssim \delta (G_w[Z] +G^\eps_{elliptic}).
\end{align*}
In case (c) of Lemma \ref{lemma:trichotomy} one has 
\begin{equation*}
	|\eta\ell-k\xi|\leq |\ell||\eta-\xi|+|\xi||k-\ell|=|\ell||\eta-\xi|+\frac{|\xi|}{|k|}|k||k-\ell|\lesssim |\ell|\jap{\eta-\xi,k-\ell}^2.
\end{equation*}
Therefore we can repeat the same argument as above.

\diampar{Bound on $\mathcal{R}^{\Omega,S}_{N,\Psi}$}
When $t\leq \TS$, if $t\in I_{k,\eta}$ then $|k|\gtrsim \sqrt{|\eta|}$. Hence, from \eqref{bd:p/p} we  have 
\begin{equation*}
	p^{-\frac14}_{k}(\eta)\chi^S\lesssim \jap{k-\ell,\eta-\xi}^{\frac32}p^{-\frac14}_\ell(\xi).
\end{equation*}
Analogously, we can always apply \eqref{bd:Jeximp}. Then, if $t\in I_{\ell,\xi}^c$ we argue as done for $\mathcal{R}^{\Omega,(NR,NR)}_{N,\Psi}$ to obtain
\begin{align}
	\sum_{N>8}\skli \mathcal{R}^{\Omega,S}_{N,\Psi}\mathbbm{1}_{t\in I_{\ell,\xi}^c}&\lesssim \delta (G_\lambda[Z]+G^\eps_{elliptic}).
\end{align}
If $t\in I_{\ell,\xi}$ and $t\leq \TS$, then $|\xi|\lesssim |\ell|^2$. Therefore $|\eta\ell-k\xi|\lesssim |\ell|^2|k-\ell,\eta-\xi|$, hence
\begin{align*}
	\mathcal{R}_{N,\Psi}^{\Omega,S}\mathbbm{1}_{t\in I_{\ell,\xi}}\lesssim &\  (|AZ|)_{k}(\eta)|\ell|^2p^{-1}_\ell(\xi)|\ell|^\frac12({A}p^{\frac34}|\widehat{\Psi}|)_{\ell}(\xi)_N \e^{\lambda|k-\ell,\eta-\xi|^s}|\widehat{\Omega}|_{k-\ell}(\eta-\xi)_{<N/8}.
\end{align*}
Since $|\ell|^2p^{-1}_\ell(\xi)\leq 1 \lesssim t^{-2s}|k,\eta|^{\frac{s}{2}}|\ell,\xi|^{\frac{s}{2}}$, appealing to \eqref{boot:En} we find
\begin{align}
	\sum_{N>8}\skli \mathcal{R}^{\Omega,S}_{N,\Psi}\mathbbm{1}_{t\in{I}_{\ell,\xi}}\lesssim \delta (G_\lambda[Z]+G^\eps_{elliptic}),
\end{align}
as required for  \eqref{bd:RNprop}. 

\diampar{Bound on $\mathcal{R}^{\Omega,L}_{N,\Psi}$}
In this case, from \eqref{bd:p/pkk} we know that  
\begin{equation}
	\label{bd:longp}
	p^{-1/4}_{k}(\eta)+p^{-1/4}_{\ell}(\xi) \lesssim \frac{1}{t^\frac12}\jap{\eta-\xi}^{1/2}.
\end{equation}
 When $t\in I_{\ell,\xi}^c$, in view of \eqref{bd:Jeximp} we find 
\begin{equation}
	(Ap^{-\frac14})_k(\eta)|k|^\frac12|\eta\ell-k\xi|\chi^L \mathbbm{1}_{t\in I_{\ell,\xi}^c}\lesssim \mathbbm{1}_{t\in I_{\ell,\xi}^c}\frac{1}{t^\frac12}|k,\eta|^{\frac{s}{2}}|\ell,\xi|^{1-\frac{s}{2}} p^{-\frac34}_{\ell}(\xi)|\ell|^\frac12(p^\frac34A)_\ell(\xi)\e^{c\lambda |k-\ell,\eta-\xi|^s},
\end{equation}
whence, using \eqref{bd:ketapct12} we get 
\begin{align}
	\sum_{N>8}\skli \mathcal{R}^{\Omega,L}_{N,\Psi}\mathbbm{1}_{t\in I_{\ell,\xi}^c}&\lesssim \delta (G_\lambda[Z]+G^\eps_{elliptic}).
\end{align}
If $t\in I_{\ell,\xi}$ then $|\ell,\xi|\lesssim |\xi|$ so that from \eqref{bd:longp} we have
\begin{align}
	(Ap^{-\frac14})_k(\eta)|k|^\frac12|\eta\ell-k\xi|\chi^L \mathbbm{1}_{t\in I_{\ell,\xi}}&\lesssim \frac{1}{t^\frac12}|\xi| p^{-\frac34}_{\ell}(\xi)|\ell|^\frac12(p^\frac34A)_\ell(\xi)\e^{c\lambda |k-\ell,\eta-\xi|^s}\notag\\
	&\lesssim \frac{\min\{|\xi|,|\eta|\}^{1-s}}{t^{2}}|k,\eta|^{\frac{s}{2}}|\ell,\xi|^{\frac{s}{2}}|\ell|^\frac12(p^\frac34A)_\ell(\xi)\e^{c\lambda |k-\ell,\eta-\xi|^s}.
\end{align}
Since $\TL\leq t$ we have $t^{-2}|\min\{|\xi|,|\eta|\}|^{1-s}\lesssim t^{-1-s}$, so that
\begin{align}
	\sum_{N>8}\skli \mathcal{R}^{\Omega,L}_{N,\Psi}\mathbbm{1}_{t\in I_{\ell,\xi}}&\lesssim \delta (G_\lambda[Z]+G^\eps_{elliptic}),
\end{align}
as we wanted.

\bullpar{Bound on $R^{\Omega}_{N,\delta}$}
First observe that 
\begin{equation}
	\label{eq:intpartsv'}
	h\nabla^{\perp}\Psi=\nabla^{\perp}(h\Psi)+(\Psi\de_vh,0).
\end{equation}
Hence, by the definition \eqref{def:RZdelta} we have 
\begin{align*}
{R}^{\Omega}_{N,\delta}\leq& \ \skli |AZ|_k(\eta)(Ap^{-\frac14})_k(\eta)|k|^\frac12|\eta \ell-k\xi||\mathcal{F}(h\Psi)|_\ell(\xi)_N|\widehat{\Omega}|_{k-\ell}(\eta-\xi)_{<N/8}\dd\eta \dd\xi\\
	&+\skli |AZ|_k(\eta)(Ap^{-\frac14})_k(\eta)|k|^\frac12|\mathcal{F}(\Psi\de_vh)|_\ell(\xi)_N|\widehat{\de_z\Omega}|_{k-\ell}(\eta-\xi)_{<N/8}\dd\eta \dd\xi\\
	=& {R}^{1}_{N,\delta}+{R}^{2}_{N,\delta}.
\end{align*}
One would like to directly treat these terms as a $\delta$-perturbation of $R^{\Omega}_{N,\Psi}$, however, this is not true in general. More precisely, as done in the proof of Proposition \ref{prop:elliptic}, we first consider the following paraproduct decomposition in the $v$-variable (since $v'$ does not depend on $z$):
\begin{align}
	h\Psi=h_{Hi}\Psi_{lo}+h_{lo}\Psi_{Hi} +h_{Hi}\Psi_{Hi}, \qquad 	\Psi	\de_vh=\Psi_{lo}(\de_vh)_{Hi}+\Psi_{Hi}(\de_vh)_{lo} +\Psi_{Hi}(\de_vh)_{Hi}.
\end{align} 
and we define 
\begin{align}
	\label{def:splitRidelta}
	{R}^{i}_{N,\delta}={R}^{i}_{\delta, LH}+{R}^{i,z}_{\delta, HL}+{R}^{i,v}_{\delta, HL}+{R}^{i}_{\delta, HH},
\end{align}
 where the term ${R}^{i,z}_{\delta,HL}$ denote the part of ${R}^{i}_{\delta,HL}$ with the cut-off $\chi^z=\mathbbm{1}_{|\ell|\geq 16|\xi'|}$. Analogously, $R^{i,v}_{\delta,HL}$ has the cut-off $\chi^v=\mathbbm{1}_{|\ell|\leq 16|\xi'|}$. With a slight abuse of notation we omitted the subscript $N$.

\diampar{Coefficients in (relatively) low frequencies}
In the proof of Proposition \ref{prop:elliptic} we have seen that we can treat in the same way the \textit{low-high} case or \textit{high-low} with $|k|\geq 16|\eta|$. This because we can always pay derivatives on the coefficients. Therefore, the most problematic term will be will ${R}^1_{N,\delta}$ since more derivatives are hitting $\Psi$. However, the case under consideration can be treated by reasoning as done for the term $R^{\Omega}_{N,\Psi}$. More precisely, we first claim that the following inequality holds true
\begin{align}
	\label{bd:Ip34}\mathcal{I}^{p^{3/4}}&=\norm{\left(\frac{|\nabla|^{\frac{s}{2}}}{\jap{t}^q}A+\sqrt{\frac{\de_tw}{w}}\widetilde{A}\right)(-\Delta_L)^{\frac34}|\de_z|^\frac12\left(h_{lo} \Psi_{Hi}+\chi^{z}h_{Hi} \Psi_{lo})\right)_{\neq}}^2\lesssim \delta^2 G^\eps_{elliptic},
\end{align}
where $\chi^z_k(t,\eta)=\mathbbm{1}_{|k|\geq 16|\eta|}$.
Indeed, first observe that since $v$ does not depend on $z$, the factor $|\de_z|^\frac12$ is always on the stream function. To prove \eqref{bd:Ip34}, one can always use \eqref{bd:p/pkk}, \eqref{bd:Jeximp}  to move the multipliers onto $\Psi$ by paying regularity on $h$. We omit the details of this argument since it has been done in the proof of Proposition \ref{prop:elliptic}. Hence, from Young's convolution inequality, the bootstrap hypothesis \eqref{boot:Ev} and Proposition \ref{prop:elliptic} we infer
\begin{align*}
\mathcal{I}^{p^{3/4}}\lesssim\norm{h}_{\G^{\lambda,\sigma-4}}^2\norm{\left(\frac{|\nabla|^{\frac{s}{2}}}{\jap{t}^q}A+\sqrt{\frac{\de_tw}{w}}\widetilde{A}\right)(-\Delta_L)^\frac34 |\de_z|^\frac12 \Psi)_{\neq}}^2\lesssim \delta^2 G^\eps_{elliptic}, 
\end{align*}
where in the last inequality  we used \eqref{bd:precellcontr}. Having at hand \eqref{bd:Ip34}, we can repeat exactly the same argument done for ${R}^{\Omega}_{N,\Psi}$ to conclude that 
\begin{equation}
	\sum_{N, M>8}{R}^{\Omega}_{\delta,LH}=	\sum_{N, M>8}\sum_{i=1}^2 {R}^{i}_{\delta,LH}+{R}^{i,z}_{\delta,HL}\lesssim \delta 	\sum_{N>8}{R}^{\Omega}_{N,\Psi}.
\end{equation}
Using the bounds done for $R^{\Omega}_{N,\Psi}$, we see that also in this case we get a bound consistent with \eqref{bd:RNprop}.

\diampar{Coefficients (truly) at high-frequencies}
We now have to deal with the \textit{high-low} case when $|k|\leq 16|\eta|$. In this case, as evident from ${R}^2_{N,\delta}$ we need to recover some derivative for the term $\de_vh$. We will treat only ${R}^2_{N,\delta}$ since the term ${R}^1_{N,\delta}$ is analogous in this high-low regime. Writing down the term explicitly one has 
\begin{align}
{R}^{2,v}_{\delta,HL}\leq\sum_{N, M>8}\sum_{k,\ell}\int_{\mathbb{R}^3}& |AZ|_k(\eta)(Ap^{-\frac14})_k(\eta)|k|^\frac12\rho_N(\xi)_\ell\mathbbm{1}_{|\ell|\leq 16|\xi'|}|\xi'|\widehat{h}(\xi')_M \notag\\
	&\times|\widehat{\Psi}|_\ell(\xi-\xi')_{<M/8} |{\widehat{\de_z\Omega}}|_{k-\ell}(\eta-\xi)_{<N/8}\dd\eta \dd\xi \dd \xi',\label{def:R2vdelta}
\end{align}
where $\rho_N$ is the cut-off of the paraproduct decomposition as defined in \eqref{eq:paraprod}.
We now have to exploit the fact that we control the coefficients with a stronger norm. More precisely, by the definition of $A^{v}$, see \eqref{def:ARvintro}, reasoning as in \eqref{bd:ellAv0}-\eqref{bd:ellAv1} we have 
\begin{equation}\label{eq:est6130}
	(Ap^{-\frac14})_k(\eta)|\xi'|\lesssim |\xi'|\jap{\frac{k}{\xi'}}^{\frac12}A^{v}(\xi')\e^{c\lambda |\ell,\xi'-\xi|^s+c\lambda |\eta-\xi|^s}\lesssim |\xi'|^{\frac12}A^{v}(\xi')\e^{c\lambda |\ell,\xi'-\xi|^s+c\lambda |k-\ell,\eta-\xi|^s}, 
\end{equation}
where we also used the fact that $A\lesssim \widetilde{A}$ when $|\ell|\leq 16|\xi'|$. Hence, since $s>1/2$, appealing to Proposition \ref{prop:lossyelliptic} and the bootstrap hypotheses \eqref{boot:En}, we infer
\begin{align}
		\sum_{N, M>8}{R}^{2,v}_{\delta,HL}\lesssim \norm{|\nabla|^{\frac{s}{2}}AZ}\norm{|\de_v|^{\frac{s}{2}}A^{v}h}\norm{\Psi}_{\G^{\lambda,\sigma-4}}\norm{\Omega}_{\G^{\lambda,\sigma-4}} \lesssim\delta ( G_{\lambda}[Z]+\eps^2 G^{v}_\lambda[h]),
\end{align}
as needed for \eqref{bd:RNprop}.

\diampar{The high-high term}
Applying the same reasoning of \cite{BM15}, we have
\begin{equation}
\sum_{N>8}\sum_{i=1}^2R^{i}_{\delta,HH}\lesssim \delta^2 \frac{\eps^2}{t^\frac32}.
\end{equation}

\bullpar{Bound on ${R}^{\Omega}_{N,\dot{v}}$}
To control this term, we are going to exploit the consequences of the bootstrap hypotheses \eqref{bd:dvdotv}-\eqref{bd:Gwdotv}. Indeed, we recall that to we control $\dot{v}$ via \eqref{boot:vdot} and bounds on  $\mathcal{H}$. Notice that we need to recover $s$-derivatives but this will be balanced by the extra decay in time available for $\mathcal{H}$ (or $\de_v{\dot{v}}$).
From  \eqref{def:ROmegaNdotv} we have
\begin{align}
R^{\Omega}_{N,\dot{v}}&\lesssim \skli |AZ|_k(\eta)(Ap^{-\frac14})_k(\eta)|k|^{\frac12}|\widehat{\dot{v}}|(\xi)_N|{\widehat{\de_v\Omega}}|_{k}(\eta-\xi)_{<N/8}\dd\eta \dd\xi\\
&=\skli \mathcal{R}^{\Omega}_{N,\dot{v}}(\mathbbm{1}_{t\in I_{k,\eta}\cap I_{k,\xi}}\chi^I+(1-\mathbbm{1}_{t\in I_{k,\eta}\cap I_{k,\xi}}\chi^I)).
\end{align}
Since $\Omega$ is at low frequencies we can always move all the factors $|k|$ to this term. Here, the most dangerous case is when $t\in I_{k,\eta}\cap I_{k,\xi}$. Indeed, we know that $\dot{v}$ always has non-resonant regularity since does not depend on $z$, whereas the weight $A$ is at resonant regularity. Due to the regularity gap between $w_R$ and $w_{NR}$, we will lose $1/2$ derivatives in $v$. In particular, when $t\in I_{k,\eta}\cap  I_{k,\xi}$ and $\TS\leq t \leq \TL$, from \eqref{bd:JexTDs},  Lemma \ref{lemma:detw/wfar} and the definition of the weight $w$, see \eqref{def:w}, we have 
\begin{align}
\label{bd:Ap-14kk}	(Ap^{-\frac14})_k(\eta)\mathbbm{1}_{t\in I_{k,\eta}\cap I_{k,\xi}} \chi^I
	&\lesssim \frac{|\eta|^{\frac12+s}}{|k|^\frac32}\frac{\de_t w_k(\eta)}{w_k(\eta)}\frac{A_0(\xi)}{\jap{\xi}^s}\e^{c\lambda|k,\eta-\xi|^s}.
\end{align}
 Since $t\in I_{k,\eta}$ (and $s<1$) we know that $|\eta|^{1/2+s}/|k|^{3/2}\lesssim t^{1/2+s}$. Hence, combining the inequality above with \eqref{bd:wexfaway} and \eqref{bd:Gwdotv} we get
\begin{align}
	\sum_{N>8} \skli \mathcal{R}^{\Omega}_{N,\dot{v}}\mathbbm{1}_{t\in I_{k,\eta}\cap I_{k,\xi}}\chi^I &\lesssim t^{\frac12+s}\norm{\sqrt{\frac{\de_t w}{w}}\tA Z}\norm{\sqrt{\frac{\de_t w}{w}}\frac{A_0}{\jap{\de_v}^s}\dot{v}}\norm{\Omega}_{\G^{\lambda,\sigma-6}}\\
	&\lesssim \delta G_w[Z]+\delta^{-1}\eps^2t^{2+2s}G_w[\jap{\de_v}^{-s}\mathcal{H}]),
\end{align} 
as required by \eqref{bd:RNprop}.  
For the remaining term, thanks to \eqref{bd:Jeximp} we and \eqref{bd:wNRexgen} we know that we never lose derivatives from $A_k(\eta)/A_0(\xi)$. In particular, we have
\begin{align*}
	(Ap^{-\frac14})_k(\eta)(1-\mathbbm{1}_{t\in I_{k,\eta}^c} \chi^I)\lesssim |k,\eta|^{\frac{s}{2}}|\xi|^{\frac{s}{2}}\frac{A_0(\xi)}{\jap{\xi}^s}\e^{c\lambda|k,\eta-\xi|^s}.
\end{align*}
This way, we conclude that
\begin{align}
	\sum_{N>8} \skli \mathcal{R}^{\Omega}_{N,\dot{v}}(1-\mathbbm{1}_{t\in I_{k,\eta}\cap I_{k,\xi}}\chi^I ) &\lesssim \eps t^{\frac12}\norm{|\nabla|^{\frac{s}{2}}AZ}\norm{|\de_v|^{\frac{s}{2}}\frac{A_0}{\jap{\de_v}^s}\dot{v}}\\
	&\lesssim\delta \frac{1}{t^{\frac12+s}}\norm{|\nabla|^{\frac{s}{2}}AZ}^2+\delta^{-1}\frac{\eps^2}{t^{\frac12+s}}t^{2+2s}\norm{|\de_v|^{\frac{s}{2}}\frac{A_0}{\jap{\de_v}^s}\dot{v}}^2.
	\end{align} 
Since $s>1/2$, appealing to \eqref{bd:Glambdadotv} we get
\begin{align}
	\sum_{N>8} \skli \mathcal{R}^{\Omega}_{N,\dot{v}}(1-\mathbbm{1}_{t\in I_{k,\eta}\cap I_{k,\xi}}\chi^I ) &\lesssim \delta G_\lambda[Z]+\delta^{-1}\eps^2 t^{2+2s}G_\lambda[\jap{\de_v}^{-s}\dot{v}]\\
	\notag &\lesssim \delta G_\lambda[Z]+\delta^{-1}\eps^2 (\jap{t}^{2+2s}G_\lambda[\jap{\de_v}^{-s}\mathcal{H}]-\dot{\lambda}(t)\eps^2t^{2+2s-4}),
\end{align} 
that is in agreement with \eqref{bd:RNprop} since $s<1$.

\bullpar{Bound on $\mathcal{R}^{\Omega}_{com}$}
This is the easiest term to control since we have room to pay regularity on $\bU$ in order to get integrability. In particular, since on the support of the integrand $|k-\ell,\eta-\xi|\lesssim |\ell,\xi|$, by Cauchy-Schwarz and Young's convolution inequality we get 
\begin{equation}
	\sum_{N>8} R^{\Omega}_{N,com}\leq \norm{AZ}^2 \norm{\bU}_{H^{\sigma-5}}\lesssim \frac{\eps^3}{t^{\frac32}},
\end{equation}
as needed in \eqref{bd:RNprop}.

\subsubsection{{Bound on} $R^{\Theta,i}_N$}
\label{sub:RTheta1}
As done in Section \ref{sec:RZiN}, we will consider just the bound for $R^{\Theta,1}_N$ and we split this term as
\begin{align}
	R_N^{\Theta,1}
	\lesssim \ R^{\Theta}_{N,\Psi}+R^{\Theta_0}_{N,\Psi}+{R}^{\Theta}_{N,\delta}+{R}^{\Theta}_{N,\dot{v}}+{R}^{Q}_{N,com}.
\end{align}
where we define 
\begin{align}
	\label{def:RQPsi}&{R}^{\Theta}_{N,\Psi}=\skli |AQ|_k(\eta)(Ap^{\frac14})_k(\eta)|k|^\frac12|\eta\ell-k\xi||\widehat{\Psi}_{\neq}|_\ell(\xi)_N|\widehat{\Theta}_{\neq}|_{k-\ell}(\eta-\xi)_{<N/8}\dd\eta \dd\xi\\
		\label{def:RTheta0}&{R}^{\Theta_0}_{N,\Psi}=\skli |AQ|_k(\eta)(Ap^{\frac14})_k(\eta)|k|^\frac32|\widehat{\Psi}_{\neq}|_k(\xi)_N|\widehat{\de_v\Theta}_0|(\eta-\xi)_{<N/8}\dd\eta \dd\xi\\
	\label{def:RQdelta}&{R}^{\Theta}_{N,\delta}=\skli |AQ|_k(\eta)(Ap^{\frac14})_k(\eta)|k|^\frac12|\mathcal{F}(h\nabla^\perp\Psi_{\neq})|_\ell(\xi)_N|\widehat{\nabla\Theta}|_{k-\ell}(\eta-\xi)_{<N/8}\dd\eta \dd\xi \\
			\label{def:RQdotv}&{R}^{\Theta}_{N,\dot{v}}=\skli |AQ|_k(\eta)(Ap^{\frac14})_k(\eta)|k|^\frac12|\dot{v}|(\xi)_N|\widehat{\de_v\Theta}|_{k}(\eta-\xi)_{<N/8}\dd\eta \dd\xi\\
	\label{def:RQcom}&{R}^{Q}_{N,com}=\skli |AQ|_k(\eta)|\widehat{\bU}|_\ell(\xi)_N |A\nabla Q|_{k-\ell}(\eta-\xi)_{<N/8}\dd\eta \dd\xi
\end{align}
Throughout this section, we will make use of the bound \eqref{bd:Thneqlow} and, as a direct consequence of 
\eqref{boot:En}, we also have
\begin{align}
	\label{bd:Thneqlow0}	\norm{\de_v\Theta_0(t)}_{\G^{\lambda,\sigma}}=\beta^{-1}\norm{\beta(\nabla_L\Theta)_0(t)}_{\G^{\lambda,\sigma}}\lesssim \eps t^{\frac12}\lesssim \delta.
\end{align}
In view of the inequality \eqref{bd:Thneqlow}, the bounds for $R^{\Theta,1}_N$ will be easier with respect to the ones for $R^{\Omega,1}_N$. In particular, we immediately have a factor $\eps t^{-1/2}$ from the low frequency part, whereas in Section \ref{sec:RZiN} the low frequency contribution give us $\eps t^{1/2}$.

\bullpar{Bound on $R^\Theta_{N,\Psi}$}
Following the notation introduced in \eqref{def:resnonresdecom}, we split the term as 
\begin{equation}
	{R}_{N,\Psi}^\Theta=\skli \mathcal{R}_{N,\Psi}^{\Theta,(R,R)}+\mathcal{R}_{N,\Psi}^{\Theta,(NR,R)}+\mathcal{R}_{N,\Psi}^{\Theta,(R,NR)}+\mathcal{R}_{N,\Psi}^{\Theta,(NR,NR)}+\mathcal{R}_{N,\Psi}^{\Theta,S}+\mathcal{R}_{N,\Psi}^{\Theta,L}.
\end{equation}
We now control each term separately.

\diampar{Bound on $\mathcal{R}_{N,\Psi}^{\Theta,(R,NR)}$} 
Using $|\eta\ell-k\xi|\lesssim |\ell,\xi||k-\ell,\eta-\xi|$ and $t\approx |\eta/k|$, from \eqref{bd:Jexgen} and \eqref{bd:ketapct12}, we get
\begin{align}
\notag (Ap^\frac14)_k(\eta)|k|^\frac12|\eta\ell-k\xi|\mathbbm{1}_{t\in I_{k,\eta}\cap I_{\ell,\xi}^c}\chi^I&\lesssim |\ell,\xi|\frac{|\eta|^\frac12}{|k|(1+|t-\frac{\eta}{k}|)^\frac12}p^{\frac14}_{k}(\eta)p^{-\frac34}_{\ell}(\xi)|\ell|^\frac12(p^\frac34A)_\ell(\xi)\e^{c\lambda|k-\ell,\eta-\xi|^s}\\
\notag &\lesssim t^\frac12|k,\eta|^{\frac{s}{2}} |\ell,\xi|^{1-\frac{s}{2}}p^{-\frac34}_{\ell}(\xi)|\ell|^\frac12(p^\frac34 A)_\ell(\xi)\e^{c\lambda|k-\ell,\eta-\xi|^s}\mathbbm{1}_{t\in  I_{\ell,\xi}^c}\\
&\lesssim t^{1-2s} |k,\eta|^{\frac{s}{2}} |\ell,\xi|^{\frac{s}{2}}|\ell|^\frac12(p^\frac34 A)_\ell(\xi)\e^{\lambda|k-\ell,\eta-\xi|^s}.
\end{align} 
This way, combining the inequality above with \eqref{bd:Thneqlow} and Proposition \ref{prop:elliptic}, we have 
\begin{align}
\notag
	\sum_{N>8}\skli \mathcal{R}_{N,\Psi}^{\Theta,(R,NR)}&\lesssim t^{1-s} \norm{|\nabla|^{\frac{s}{2}} A Q}\norm{\frac{|\nabla|^{\frac{s}{2}}}{\l t \r^s}A (-\Delta_L)^\frac34 |\de_z|^\frac12\Psi}\norm{\Theta_{\neq}}_{\G^{\lambda,\sigma-4}}	\\
	\label{bd:RZPsiRNR1}&\lesssim \delta (G_\lambda[Q]+G^\eps_{elliptic}),
\end{align}
implying \eqref{bd:RNprop} for this term.

\diampar{Bound on $\mathcal{R}_{N,\Psi}^{\Theta,(NR,R)}$}
 This is the most dangerous term since we lose to exchange $p_{k}(\eta)$ with $p_{\ell}(\xi)$ (but we gain when we exchange $A_k(\eta)$ with $A_\ell(\xi)$). On the support of the integrand we know $4|\ell|^2\leq |\xi|$ and \eqref{bd:stupid}, meaning that we have $A\lesssim \tA$.
 Then, appealing to \eqref{bd:Jexgood} and \eqref{bd:p/p} (with the role of $(k,\eta)$ and $(\ell,\xi)$ switched), we get 
\begin{align}
	(Ap^\frac14)_k(\eta)|k|^{\frac12}\mathbbm{1}_{t\in I_{k,\eta}^c\cap I_{\ell,\xi}}\chi^I&=p^{-\frac12}_k(\eta)(Ap^\frac34)_k(\eta)|k|^{\frac12}\mathbbm{1}_{t\in I_{k,\eta}^c\cap I_{\ell,\xi}}\chi^I\\
	&\lesssim p^{-\frac12}_k(\eta)\frac{|\ell|(1+|t-\frac{\xi}{\ell}|)^\frac12}{|\xi|^\frac12} \frac{|\xi|^\frac32}{|\ell|^3(1+|t-\frac{\xi}{\ell}|)^\frac32}|\ell|^{\frac12}(\widetilde{A}p^\frac34)_\ell(\xi)\e^{c\lambda|k-\ell,\eta-\xi|^s}\\
	&\lesssim \frac{|k|}{\jap{\eta}} \frac{|\xi|}{|\ell|^2} \frac{\de_t w_\ell(\xi)}{w_\ell(\xi)}|\ell|^{\frac12} (\widetilde{A}p^\frac34)_\ell(\xi)\e^{c\lambda|k-\ell,\eta-\xi|^s}\\
	&\lesssim  \frac{t}{|\xi|}  \frac{\de_t w_\ell(\xi)}{w_\ell(\xi)} |\ell|^{\frac12}(\widetilde{A}p^\frac34)_\ell(\xi)\e^{\lambda|k-\ell,\eta-\xi|^s}.
\end{align}
Since $|\eta\ell-k\xi|\lesssim |\xi||k-\ell,\eta-\xi|$, combining the inequality above with \eqref{bd:wexfaway}, \eqref{bd:Thneqlow} and Proposition \ref{prop:elliptic}, we infer 
\begin{align}
	\sum_{N>8}\skli \mathcal{R}_{N,\Psi}^{\Theta,(NR,R)}&\lesssim t\norm{\sqrt{\frac{\de_t w}{w}}\tA Q}\norm{\sqrt{\frac{\de_t w}{w}}\widetilde{A} (-\Delta_L)^\frac34|\de_z|^{\frac12}\Psi}\norm{\Theta_{\neq}}_{\G^{\lambda,\sigma-4}}\\
\label{bd:RThetaNRR}	&\lesssim \delta (G_w[Q]+G^{\eps}_{elliptic}),
\end{align}
consistent with \eqref{bd:RNprop}.

\diampar{Bound on $\mathcal{R}^{\Theta,(NR,NR)}_{N,\Psi}$}
From \eqref{bd:Jeximp} and \eqref{bd:p/p} we have 
\begin{equation*}
	(Ap^{\frac14})_{k}(\eta)|k|^{\frac12}|\eta \ell-k\xi|\mathbbm{1}_{t\in I_{k,\eta}^c\cap I_{\ell,\xi}^c}\chi^I\lesssim \mathbbm{1}_{t\in I_{\ell,\xi}^c}|k,\eta|^{\frac{s}{2}} |\ell,\xi|^{1-\frac{s}{2}}p^{-\frac12}_{\ell}(\xi) |\ell|^{\frac12}(p^{\frac34}A)_\ell(\xi)\e^{c\lambda|k-\ell,\eta-\xi|^s}.
\end{equation*}
From \eqref{bd:ketapct}, \eqref{bd:Thneqlow} and the elliptic estimate \eqref{bd:ellipticGZw} we have 
\begin{align}
	\sum_{N>8}\skli \mathcal{R}_{N,\Psi}^{\Theta,(NR,NR)}&\lesssim t^{1-s}\norm{|\nabla|^{\frac{s}{2}}AQ}\norm{\frac{|\nabla|^{\frac{s}{2}}}{\jap{t}^s}A (-\Delta_L)^\frac34|\de_z|^\frac12\Psi}\norm{\Theta_{\neq}}_{\G^{\lambda,\sigma-4}}\\
\label{bd:RThetaNRNR}	&\lesssim \delta (G_\lambda[Q] +G^{\eps}_{elliptic}),
\end{align}
as required for \eqref{bd:RNprop}.

\diampar{Bound on $\mathcal{R}^{\Theta,(R,R)}_{N,\Psi}$}
On the support of the integral we know that $4|k|^2\leq |\eta|$ and $4|\ell|^2\leq |\xi|$. Then, since $k\neq \ell$, we know that we can only have cases $(b)$ or $(c)$ in Lemma \ref{lemma:trichotomy}. The case $(c)$ is straightforward. In case $(b)$ we are in a situation similar to $\mathcal{R}^{\Theta,(NR,NR)}_{N,\Psi}$. Indeed, we know that we can apply \eqref{bd:Jeximp}, so that from \eqref{bd:p/p} we get 
\begin{equation}
	(Ap^\frac14)_{k}(\eta)|\eta\ell-k\xi||k|^{\frac12}\mathbbm{1}_{I_{k,\eta}\cap I_{\ell,\xi}}\chi^I\lesssim |\xi|p^{-\frac12}_{\ell}(\xi)|\ell|^{\frac12}(\widetilde{A}p^{\frac34})_{\ell}(\xi)\e^{c\lambda|k-\ell,\eta-\xi|^s}.
\end{equation}
We also have 
\begin{equation}
	\label{bd:trivRR}
	|\xi|p^{-\frac12}_{\ell}(\xi)\lesssim \jap{\frac{\xi}{\ell}} \frac{\de_t w_\ell(\xi)}{w_\ell(\xi)}\lesssim t\frac{\de_t w_\ell(\xi)}{w_\ell(\xi)}. 
\end{equation}
Therefore, thanks to \eqref{bd:Thneqlow} and \eqref{bd:ellipticGZw} we conclude that 
\begin{align}
	\sum_{N>8}\skli \mathcal{R}_{N,\Psi}^{\Theta,(R,R)}&\lesssim t\norm{\sqrt{\frac{\de_t w}{w}}\tA Q}\norm{\sqrt{\frac{\de_t w}{w}}\widetilde{A} (-\Delta_L)^\frac34|\de_z|^{\frac12}\Psi}\norm{\Theta_{\neq}}_{\G^{\lambda,\sigma-4}}\\
	\label{bd:RThetaRR}	&\lesssim \delta (G_w[Q]+G^{\eps}_{elliptic}),
\end{align}
in agreement with \eqref{bd:RNprop}.

\diampar{Bound on $\mathcal{R}^{\Theta,S}_{N,\Psi}$}
From \eqref{bd:Jeximp} and \eqref{bd:p/p} we deduce 
\begin{equation*}
	(Ap^{\frac14})_k(\eta)|k|^\frac12|\eta\ell-k\xi|\chi^S\lesssim |\ell,\xi|p^{-\frac12}_\ell(\xi) |\ell|^\frac12 (Ap^{\frac34})_{\ell}(\xi)\e^{c\lambda|k-\ell,\eta-\xi|^s}.
\end{equation*}
If $t\in I_{\ell,\xi}$ one has $	|\ell,\xi|p^{-\frac12}_\ell(\xi)\lesssim t^{1-2s}|k,\eta|^{\frac{s}{2}}|\ell,\xi|^{\frac{s}{2}}$. If $t\in I_{\ell,\xi}^c$ we use \eqref{bd:ketapct}, so that, in general we have 
\begin{align}
	\sum_{N>8}\skli \mathcal{R}_{N,\Psi}^{\Theta,S}\mathbbm{1}_{t\in I_{\ell,\xi}^c}&\lesssim t^{1-s}\norm{|\nabla|^{\frac{s}{2}}AQ}\norm{\frac{|\nabla|^{\frac{s}{2}}}{\jap{t}^s}A (-\Delta_L)^\frac34|\de_z|^{\frac12}\Psi}\norm{\Theta_{\neq}}_{\G^{\lambda,\sigma-4}}\\
	\label{bd:RThetaS1}&\lesssim \delta (G_\lambda[Q]+G^{\eps}_{elliptic}),
\end{align}
as needed for \eqref{bd:RNprop}.

\diampar{Bound on $\mathcal{R}^{\Theta,L}_{N,\Psi}$}
When $t\in I_{\ell,\xi}^c$, in view of \eqref{bd:p/p} and \eqref{bd:Jeximp} we find 
\begin{equation}
	(Ap^{\frac14})_k(\eta)|k|^\frac12|\eta\ell-k\xi|\chi^L \mathbbm{1}_{t\in I_{\ell,\xi}^c}\lesssim \mathbbm{1}_{t\in I_{\ell,\xi}^c}|k,\eta|^{\frac{s}{2}}|\ell,\xi|^{1-\frac{s}{2}} p^{-\frac12}_{\ell}(\xi)|\ell|^\frac12(p^\frac34A)_\ell(\xi)\e^{c\lambda |k-\ell,\eta-\xi|^s},
\end{equation}
whence, using \eqref{bd:ketapct} we get 
\begin{align}
		\label{bd:RThetaL1}
	\sum_{N>8}\skli \mathcal{R}^{\Theta,L}_{\Psi}\mathbbm{1}_{t\in I_{\ell,\xi}^c}&\lesssim \delta (G_\lambda[Q]+G^\eps_{elliptic}).
\end{align}
If $t\in I_{\ell,\xi}$, then $|\ell,\xi|\lesssim |\xi|$. Since $t\geq \TL$ we also have $p^{\frac14}_{k}(\eta)\lesssim t^\frac12 \jap{k}^\frac12$ and $p^{-3/4}_\ell(\xi)\lesssim (|k| \jap{t})^{-3/2}$. Hence 
\begin{align}
	(Ap^{\frac14})_k(\eta)|k|^\frac12|\eta\ell-k\xi|\chi^L \mathbbm{1}_{t\in I_{\ell,\xi}}&\lesssim t^\frac12\jap{k}^{\frac12}|\xi| p^{-\frac34}_{\ell}(\xi)|\ell|^\frac12(p^\frac34A)_\ell(\xi)\e^{c\lambda |k-\ell,\eta-\xi|^s}\\
	&\lesssim \frac{|\min\{|\xi|,|\eta|\}|^{1-s}}{t}|k,\eta|^{\frac{s}{2}}|\ell,\xi|^{\frac{s}{2}}|\ell|^\frac12(p^\frac34A)_\ell(\xi)\e^{c\lambda |k-\ell,\eta-\xi|^s}.
\end{align}
Exploiting the bound $t^{-1}|\min\{|\xi|,|\eta|\}|^{1-s}\lesssim t^{-s}$, we get
\begin{align}
\label{bd:RThetaL2}	\sum_{N>8}\skli \mathcal{R}^{\Theta,L}_{\Psi}\mathbbm{1}_{t\in I_{\ell,\xi}}&\lesssim \norm{|\nabla|^{\frac{s}{2}}AQ}\norm{\frac{|\nabla|^\frac{s}{2}}{\l t\r^s}(-\Delta_L)^\frac34|\de_z|^\frac12\Psi}\norm{\Theta_{\neq}}_{\G^{\lambda,\sigma-4}}\\
	\notag &\lesssim \frac{\delta}{t}\norm{|\nabla|^{\frac{s}{2}}AQ}\norm{\frac{|\nabla|^\frac{s}{2}}{\l t\r^s}(-\Delta_L)^\frac34|\de_z|^\frac12\Psi}\lesssim \delta (G_\lambda[Q]+G^\eps_{elliptic}),
\end{align}
as required in \eqref{bd:RNprop}.

\bullpar{Bound on $R^{\Theta_0}_{N,\Psi}$}
For this term, since $\Psi$ is at the same frequency $k$, we can move the multiplier $Ap^\frac14$ withouth losing derivatives in the high-frequency part. In addition, we only need to recover one derivative in $z$. Observe that,  appealing to \eqref{bd:p/pkk} and \eqref{bd:Jeximp}, we have 
\begin{equation}
	\label{bd:RTheta0first}
	(Ap^{\frac14})_{k}(\eta)|k|^\frac32\lesssim \frac{|k|}{(|k|^2+|\xi-kt|^2)^\frac12}|k|^{\frac12}(Ap^{\frac34})_k(\xi)\e^{c\lambda|k-\ell,\eta-\xi|^s}. 
\end{equation}
Then, we split this term as follows 
\begin{equation}
\label{eq:splitT0}
	R^{\Theta_0}_{N,\Psi}= \skli\mathcal{R}^{\Theta_0}_{N,\Psi}(\chi^S+(1-\chi^S)(\mathbbm{1}_{t\in I_{k,\xi}}+\mathbbm{1}_{t\in I_{k,\xi}^c})) , 
\end{equation} 
where the cut-off are defined in \eqref{def:cutofftime}. Then, when $t\in I_{k,\xi}$, combining \eqref{bd:RTheta0first} with Lemma \ref{lemma:detw/wfar} we have
\begin{equation}
	(Ap^{\frac14})_{k}(\eta)|k|^\frac32\mathbbm{1}_{t\in I_{k,\xi}}(1-\chi^S)\lesssim \frac{\de_t w_k(\xi)}{w_k(\xi)}|k|^{\frac12}(\tA p^{\frac34})_k(\xi)\e^{c\lambda|k-\ell,\eta-\xi|^s}. 
\end{equation}
Since on the support of the integrand $|\eta|\approx|\xi|$, in view of \eqref{bd:wexfgen} and \eqref{bd:Thneqlow0} we get 
\begin{align}
	\notag	\sum_{N>8}\mathcal{R}^{\Theta_0}_{N,\Psi}\mathbbm{1}_{t\in I_{k,\xi}}(1-\chi^S) &\lesssim  \norm{\left(\sqrt{\frac{\de_tw }{w}} \tA+\frac{|\nabla|^{\frac{s}{2}}}{\l t\r^s}A\right) Q}\norm{\sqrt{\frac{\de_tw }{w}}\widetilde{A}(-\Delta_L)^{\frac34}|\de_z|^\frac12 \Psi}\norm{\de_v\Theta_0}_{\G^{\lambda,\sigma-6}}\\
	&\lesssim \delta(G_w[Q]+G_\lambda[Q]+G^{\eps}_{elliptic}).
\end{align}
When $t\in I_{k,\xi}^c$, $t\geq \TS$ and $|kt|/2\leq |\xi|\leq 2|kt|$, since $s>1/2$ observe that 
\begin{equation}
\frac{|k|}{(|k|^2+|\xi-kt|^2)^\frac12}\lesssim \frac{|k|}{(|k|^2+|\xi|^2/|k|^2)^\frac12} \lesssim \frac{|k|}{|\xi|^{\frac12}}\lesssim \frac{|\xi|^{\frac12}}{t}\lesssim t^{-\frac12-s} |\xi|^s,
\end{equation} 
which, since $q\leq 1/4+s/2$, implies
\begin{align}
	\notag	\sum_{N>8}\mathcal{R}^{\Theta_0}_{N,\Psi}\mathbbm{1}_{\{t\in I_{k,\xi}^c\}\cap\{|kt|/2\leq |\xi|\leq 2|kt|\}}(1-\chi^S) &\lesssim  t^{-q}\norm{|\nabla|^{\frac{s}{2}}A Q}\norm{\frac{|\nabla|^\frac{s}{2}}{\l t\r^q}A(-\Delta_L)^{\frac34}|\de_z|^\frac12 \Psi}\norm{\de_v\Theta_0}_{\G^{\lambda,\sigma-6}}\\
	\label{bd:stoq}&\lesssim \delta(G_\lambda[Q]+G^{\eps}_{elliptic}).
\end{align}
If $|\xi|\leq |kt|/2$ or $|\xi|\geq 2|kt|$ we have
\begin{equation}
\label{bd:nice}
\frac{|k|}{(|k|^2+|\xi-kt|^2)^\frac12}\mathbbm{1}_{|\xi|\leq|kt|/2}\lesssim \frac{1}{t}, \qquad \frac{|k|}{(|k|^2+|\xi-kt|^2)^\frac12}\mathbbm{1}_{|\xi|\geq 2|kt|}\lesssim \frac{1}{|\xi/k|}\lesssim\frac{1}{t}\jap{\frac{\xi}{kt}}^{-1}.
\end{equation} 
Then, from \eqref{bd:ellp34APsi}, \eqref{bd:RTheta0first} and \eqref{boot:E} we obtain
\begin{align}
	\notag	\sum_{N>8}\mathcal{R}^{\Theta_0}_{N,\Psi}\mathbbm{1}_{t\in I_{k,\xi}^c\cap\{|\xi|\leq|kt|/2\}\cap\{|\xi|\geq2|kt|\}}(1-\chi^S) &\lesssim  \frac{1}{t}\norm{A Q}\norm{ A(-\Delta_L)^{\frac34}|\de_z|^\frac12 \Psi}\norm{\de_v\Theta_0}_{\G^{\lambda,\sigma-6}}\lesssim \frac{\eps^3}{t^\frac12}.
\end{align}
When $t\leq \TS$ we use $1\lesssim t^{-2s}|k,\eta|^{\frac{s}{2}}|k,\xi|^{\frac{s}{2}}$, to get
\begin{align}
	\notag	\sum_{N>8}\mathcal{R}^{\Theta_0}_{N,\Psi}\chi^S&\lesssim  t^{-s}\norm{|\nabla|^{\frac{s}{2}} AQ}\norm{\frac{|\nabla|^{\frac{s}{2}}}{\jap{t}^s}{A}(-\Delta_L)^{\frac34}|\de_z|^\frac12 \Psi}\norm{\de_v\Theta_0}_{\G^{\lambda,\sigma-6}}\lesssim \delta(G_\lambda[Q]+G^{\eps}_{elliptic}).
\end{align}
Therefore, also for the term $R^{\Theta_0}_{N,\Psi}$ we have bounds consistent with \eqref{bd:RNprop}.

\bullpar{Bound on $R^{\Theta}_{N,\delta}$}
We again consider a paraproduct decomposition in the $v$-variable and, similarly to \eqref{def:splitRidelta}, we write 
\begin{equation}
	R^\Theta_{N,\delta}={R}^\Theta_{\delta,LH}+{R}^{\Theta,z}_{\delta,HL}+{R}^{\Theta,v}_{\delta,HL}+{R}^\Theta_{\delta,HH}.
\end{equation}
Since in the bounds for ${R}_{N,\Psi}^\Theta$ we never used the fact we had $|\eta\ell-k\xi|$ instead of $|\ell,\xi|$\footnote[$\dagger$]{In the bounds for $\mathcal{R}^{\Omega,(R,R)}_{N,\Psi}$ one exploit the presence of $|\eta \ell-k\xi|$ when $k=\ell$. }, for $k\neq \ell$ we can repeat the arguments done for $\mathcal{R}^\Theta_{N,\Psi}$ to deal with the case when the coefficients are in (relatively) low frequencies, namely $R^\Theta_{\delta,LH}$ and $R^{\Theta,z}_{\delta,HL}$. The same holds true if $k=\ell$, where the bounds are done as the ones for $R^{\Theta_0}_{N,\Psi}$. In particular, one has 
\begin{equation}
\label{bd:RTdelta0}
	\sum_{N, M>8}{R}^\Theta_{\delta,LH}+{R}^{\Theta,z}_{\delta,HL}\lesssim \delta \sum_{N>8}{R}^\Theta_{N,\Psi}+{R}^{\Theta_0}_{N,\Psi}.
\end{equation}
 We are then left with the high-low case when the coefficients are truly at high frequencies. In analogy with the notation used in \eqref{def:R2vdelta} we have to control
\begin{align}
		\label{def:RQdeltav}{R}^{\Theta,v}_{\delta,HL}=&\sum_{k,\ell}\int_{\mathbb{R}^3} |AQ|_k(\eta)(Ap^{\frac14})_k(\eta)|k|^{\frac12}\rho_N(\xi)_\ell\mathbbm{1}_{|\ell|\leq 16|\xi'|}|\widehat{h}|(\xi')_M\\
		&\qquad \qquad \times |\widehat{\nabla^\perp\Psi_{\neq}}|_\ell(\xi-\xi')_{<M/8}|\widehat{\nabla\Theta}|_{k-\ell}(\eta-\xi)_{<N/8}\dd\eta \dd\xi \dd\xi'\\
		=&\skli \mathcal{R}^{\Theta,v}_{\delta,HL}.
\end{align}
For this term we can always move derivatives in $z$ onto the streamfunction. Then, from the definition of the weight $A^{v}$, see \eqref{def:ARvintro}, since $16|\ell|\leq |\xi'|$ we know that on the support of the integral $A_k(\eta)\lesssim A^{v}(\eta)\lesssim A^{v}(\xi')$. Hence, appealing to \eqref{bd:p/pkk} and \eqref{bd:Jeximp} we have
\begin{equation}
	(Ap^{\frac14})_k(\eta)\lesssim p^{\frac14}_k(\xi') A^{v}(\xi')\e^{c\lambda|k-\ell,\xi'-\xi|^s+c\lambda|\eta-\xi|^s}.
\end{equation}
Then, notice that 
\begin{equation*}
	p^{\frac14}_k(\xi')\lesssim (t^\frac12\mathbbm{1}_{|\xi'|\leq t}+|\xi'|^\frac12\mathbbm{1}_{|\xi'|\geq t})\jap{k}^\frac12\lesssim (t^\frac12\mathbbm{1}_{|\xi'|\leq t}+|\xi'|^\frac12\mathbbm{1}_{|\xi'|\geq t})\jap{k-\ell}^\frac12\jap{\ell}^\frac12.
\end{equation*}
Since in general we do not have $k\neq \ell$, we can only use the worst bound \eqref{bd:Thneqlow0}. However, combining the two bounds above with the bootstrap hypotheses and Proposition \ref{prop:lossyelliptic}, when $|\xi'|\leq t$ we have 
\begin{align}
\label{bd:RTdelta1}
	\sum_{N, M>8}\skli\mathcal{R}^{\Theta,v}_{\delta,HL}\mathbbm{1}_{|\xi'|\leq t}\lesssim t^{\frac12} \norm{AQ}\norm{A^vh}\norm{\Psi_{\neq}}_{\G^{\lambda,\sigma-5}}\norm{\Theta}_{\G^{\lambda,\sigma-5}}\lesssim \eps^4.
\end{align}
If $|\xi'|\geq t$, since $s>1/2$ one has $|\xi|^{1/2}= |\xi'|^{1/2-s}|\xi'|^s\lesssim t^{1/2-s}|k,\eta|^{s/2}|\xi'|^{s/2}$. This way, from the bootstrap hypotheses and \eqref{bd:Thneqlow0} we get
\begin{align}
\label{bd:RTdelta2}
	\sum_{N, M>8}\skli\mathcal{R}^{\Theta,v}_{\delta,HL}\mathbbm{1}_{|\xi'|\geq t}&\lesssim t^{\frac12-s} \norm{|\nabla|^{\frac{s}{2}}AQ}\norm{|\de_v|^{\frac{s}{2}}A^{v}h}\norm{\Psi_{\neq}}_{\G^{\lambda,\sigma-5}}\norm{\Theta}_{\G^{\lambda,\sigma-5}}\\
	&\lesssim \frac{\eps^2}{t^{s+\frac12}}\norm{|\nabla|^{\frac{s}{2}}AQ}\norm{|\de_v|^{\frac{s}{2}}A^{v}h}\lesssim \eps^2 (G_\lambda[Q]+ G^v_\lambda[h]).
\end{align}
For the high-high term, we get 
\begin{equation}
\sum_{N>8}R^{\Theta}_{\delta,HH}\lesssim \norm{AQ} \norm{h}_{\G^{\lambda,\sigma-4}}\norm{A(-\Delta_L)^{\frac14}\Psi_{\neq}}_{\G^{\lambda,\sigma-4}}\norm{\Theta}_{\G^{\lambda,\sigma-4}}\lesssim\eps^4,
\end{equation}
meaning that $R^{\Theta}_{N,\delta}$ satisfies bounds in agreement with \eqref{bd:RNprop}.

\bullpar{Bound on $R^{\Theta}_{N,\dot{v}}$}
First split this term as 
\begin{align*}
	R^{\Theta}_{N,\dot{v}}=\skli \mathcal{R}^{\Theta}_{N,\dot{v}}(\mathbbm{1}_{t\in I_{k,\eta}\cap I_{k,\xi}}\chi^I+(1-\mathbbm{1}_{t\in I_{k,\eta}\cap I_{k,\xi}}\chi^I)).
\end{align*}
Notice that we always have $\Theta_{\neq}$ since if $k=0$ the term above vanishes. This is crucial since we can always recover some time-decay from \eqref{bd:Thneqlow}. The treatment of this term will be similar to $R^{\Omega}_{N,\dot{v}}$, however, in view of the $p^{\frac14}$ the worst case will be when $|\eta|\geq t$. In contrast to \eqref{bd:Ap-14kk}, for $t\in I_{k,\eta}\cap  I_{k,\xi}$, appealing to \eqref{bd:Jeximp}, Lemma \ref{lemma:detw/wfar} and \eqref{def:w}, we now argue as follows  
\begin{align}
		(Ap^{\frac14})_k(\eta)\mathbbm{1}_{t\in I_{k,\eta}\cap I_{k,\xi}} \chi^I&\lesssim  |k|^\frac12(1+|t-\frac{\eta}{k}|)^\frac12\frac{A_k(\xi)}{A_0(\xi)}A_0(\xi)\e^{c\lambda|k,\eta-\xi|^s} ,\\
              &\lesssim \left(\frac{|\eta|}{|k|}\right)^\frac12 A_0(\xi)\e^{c\lambda|k,\eta-\xi|^s} \lesssim t^{\frac12}|k,\eta|^{\frac{s}{2}}|\xi|^{\frac{s}{2}}\frac{A_0(\xi)}{\jap{\eta}^s}\e^{c\lambda|k,\eta-\xi|^s},
\end{align}
where in the last line we used $t\approx |\eta/k|$. Then, using \eqref{bd:Thneqlow} and \eqref{bd:Glambdadotv} we get
\begin{align}
	\label{bd:RTdotv0} \sum_{N>8} \skli \mathcal{R}^{\Theta}_{N,\dot{v}}\mathbbm{1}_{t\in I_{k,\eta}\cap I_{k,\xi}}\chi^I &\lesssim t^{\frac12}\norm{|\nabla|^{\frac{s}{2}}AQ}\norm{|\de_v|^{\frac{s}{2}}\frac{A_0}{\jap{\de_v}^s}\dot{v}}\norm{\Theta_{\neq}}_{\G^{\lambda,\sigma-6}}\\
	&\lesssim \eps t^{-
		(1+s)}t^{1+s}\norm{|\nabla|^{\frac{s}{2}}AQ}\norm{|\de_v|^{\frac{s}{2}}\frac{A_0}{\jap{\de_v}^s}\dot{v}}\\
	&\lesssim\delta  G_\lambda[Q]+\delta^{-1}\eps^2\left(t^{2+2s}G_\lambda[\jap{\de_v}^{-s}\mathcal{H}]+\frac{\eps^2}{t^{\frac52-s}}\right).
\end{align} 
For the remaining term, if $t\geq|\eta|$ then $p^{1/4}_k(\eta)\lesssim t^{1/2}\jap{k}^{1/2}$ but $A_k(\eta)\lesssim A_0(\eta)\e^{c\lambda|k,\eta-\xi|^s}$. Hence we can repeat exactly the same argument above. When $|\eta|\geq t$ we have $p^{1/4}_{k}(\eta)\lesssim |\eta|^{1/2}$. Therefore, 
\begin{align}
	\notag \sum_{N>8} \skli \mathcal{R}^{\Theta}_{N,\dot{v}}(1-\mathbbm{1}_{t\in I_{k,\eta}\cap I_{k,\xi}}\chi^I )\mathbbm{1}_{|\eta|\geq t} &\lesssim \eps t^{-\frac12}\norm{|\nabla|^{\frac{s}{2}}AQ}\norm{|\de_v|^{\frac{s}{2}}\frac{A_0}{\jap{\de_v}^s}|\de_v|^{\frac12}\dot{v}_{>t}}.
\end{align}
Arguing as done to prove \eqref{bd:Glambdadotv}, since $\mathcal{H}=v'\de_v\dot{v}$ and $t\geq 1$, we deduce 
\begin{equation*}
	\norm{|\de_v|^{\frac{s}{2}}\frac{A_0}{\jap{\de_v}^s}|\de_v|^{\frac12}\dot{v}_{>t}}\lesssim 	\norm{|\de_v|^{\frac{s}{2}}\frac{A_0}{\jap{\de_v}^s}\de_v\dot{v}}\lesssim 	\norm{|\de_v|^{\frac{s}{2}}\frac{A_0}{\jap{\de_v}^s}\mathcal{H}}.
\end{equation*}
Combining the two inequalities above we have
\begin{align}	\label{bd:RTdotv1} \sum_{N>8} \skli \mathcal{R}^{\Theta}_{N,\dot{v}}(1-\mathbbm{1}_{t\in I_{k,\eta}\cap I_{k,\xi}}\chi^I )\mathbbm{1}_{|\eta|\geq t} &\lesssim \eps t^{-\frac12}\norm{|\nabla|^{\frac{s}{2}}AQ}\norm{|\de_v|^{\frac{s}{2}}\frac{A_0}{\jap{\de_v}^s}\mathcal{H}}\\
\notag &\lesssim \delta G_\lambda[Q]+\delta^{-1}\eps^2t^{2+2s}G_{\lambda}[\jap{\de_v}^{-s}\mathcal{H}].
\end{align}
Thus, for $R^{\Theta}_{N,\dot{v}}$ we have bounds required in \eqref{bd:RNprop}.

\bullpar{Bound on $R^{Q}_{N,com}$}
On the support of the integrand $|k-\ell,\eta-\xi|\lesssim |\ell,\xi|$, hence we get
\begin{equation}
	\sum_{N>8} R^{Q}_{N,com}\lesssim \norm{AQ}^2 \norm{\bU}_{H^{\sigma-6}}\lesssim \frac{\eps^3}{t^\frac32}.
\end{equation}
	This concludes the proof of \eqref{bd:RNprop}.

\subsection{Remainders}
Considering 
\eqref{eq:remaind}, we note that on the support of the integrand, $|\ell,\xi|\approx |k-\ell,\eta-\xi|$ and hence by  \eqref{app:inequality2}  we find 
\begin{align}
|k,\eta|^s\leq c|k-\ell,\eta,\xi|^s+c|\ell,\xi|^s.
\end{align}
Hence, we can always pay regularity to move the multipliers. Arguing as in \cite{BM15}*{Section 7}, we deduce that
\begin{align}
	\mathcal{R}\lesssim \norm{\bU}_{\cG^{\lambda,\sigma-6}} \left(\norm{AZ}+\norm{AQ}\right) \left(\norm{Z}_{\cG^{\lambda,\sigma-1}}+\norm{Q}_{\cG^{\lambda,\sigma-1}}+
	\norm{\de_v\Theta_0}_{\cG^{\lambda,\sigma-1}}\right),
\end{align}
which can be bounded as in  
\eqref{bd:Rem} thanks to the bootstrap assumptions \eqref{boot:E}-\eqref{boot:En} and \eqref{bd:bootUlow}.

\subsection{Remaining error terms}\label{sub:remerr}

We estimate the remainder terms \eqref{eq:linear-error}, \eqref{eq:div-error} and \eqref{eq:lasterror}, starting from the linear error term $L^{Z,Q}$ in \eqref{eq:linear-error}.
Simply using the definition of the linear weight $m$ in \eqref{def:m}, we have
\begin{align}
\frac{1}{4\beta} \left|\de_t \left( \frac{\de_t p}{|k| p^\frac 12} \right)\right| = \frac{1 }{2\beta C_\beta} \frac{ C_\beta |k|^3}{p^\frac 32} \le \frac{1}{2\beta C_\beta}  \frac{C_\beta |k|^2}{p} \le  \frac{1}{2\beta C_\beta}  \frac{\de_t m}{m},
\end{align}
where $C_\beta$ was introduced in \eqref{def:m}. This readily implies that
\begin{align}
L^{Z,Q}&\le \frac{1}{2 \beta C_\beta} \norm{\sqrt{\frac{\de_t m}{m}} AZ} \norm{\sqrt{\frac{\de_t m}{m}} AQ} \leq \frac{1}{2 }\left(1-\frac{1}{2\beta}\right)(G_m[Z]+G_m[Q]).
\end{align} 
Next, consider the divergence error term $\mathcal{E}^{\div}$ in \eqref{eq:div-error}.
As $|\de_t p|/|k| p^\frac 12 \le 2$, we apply   \eqref{boot:vdot} to get
\begin{align}
\mathcal{E}^{\div} & \lesssim \norm{ \dot v }_{H^3}\norm{AZ}\norm{AQ}\lesssim \frac{\eps^3}{t^{\frac32}}.
\end{align}
It remains to estimate $\mathcal{E}^{\Delta_t}$ in \eqref{eq:lasterror}. 
As $|\de_t p|/ |k| p^\frac 12 \le 2$, the two addends of \eqref{eq:lasterror} reduce to the control of 
\begin{align*}
\mathcal{E}^{\Delta_t, 1}:=\beta  \left| \left\l |k|^\frac 32 p^{-\frac34}A\mathcal{F}((\Delta_t - \Delta_L)  \Psi) ,AQ\right\r\right|.
\end{align*} 
We now write $\mathcal{F}((\Delta_t - \Delta_L)\Psi)$ explicitly as in \eqref{eq:idPsi0} and separate the part involving $g$ and the one involving $v''$ as
 $\mathcal{E}^{\Delta_t,1}= \mathcal{E}^{\Delta_t, g} + \mathcal{E}^{\Delta_t, v''}$.  
Similarly to \eqref{def:splitRidelta}, we consider the following decomposition 
\begin{align}
\label{eq:splitrem}
\mathcal{E}^{\Delta_t, j}= \mathcal{E}^{\Delta_t, j}_{LH}+\mathcal{E}^{\Delta_t, j,z}_{HL}+\mathcal{E}^{\Delta_t, j,v}_{HL}+\mathcal{E}^{\Delta_t, j}_{HH}, \qquad j\in \{g,v''\}.
\end{align}
 For the first two terms in the right-hand side of \eqref{eq:splitrem}, the most dangerous case is for $j=g$ since more derivatives hit $\Psi$. However, notice that from \eqref{bd:Jeximp} and \eqref{bd:p/pkk} we have
 \begin{align}
|k|^{\frac32}(p^{-\frac34}A)_{k}(\eta)\lesssim \frac{|k|}{(k^2+|\xi-kt|^2)^{\frac12}}|k|^{\frac12}(p^{-\frac14}A)_{k}(\xi)\e^{c\lambda|\eta-\xi|^s}.
 \end{align}
 Hence, we are in a situation analogous to $R^{\Theta_0}_{N,\Psi}$, see \eqref{bd:RTheta0first}. Proceeding as done for $R^{\Theta_0}_{N,\Psi}$, with the use of \eqref{bd:Thneqlow0} replaced by \eqref{bd:OmTh}, we get 
 \begin{align}
 \sum_{j\in\{g,v''\}}\mathcal{E}^{\Delta_t, j}_{LH}+\mathcal{E}^{\Delta_t, j,z}_{HL}\lesssim \delta (G_w[Q]+G_\lambda[Q]+G^{\eps}_{elliptic})+t^{-\frac12}\eps^3.
 \end{align}
We now consider the remaining high-low term and we only deal with $\mathcal{E}^{\Delta_t, v'',v}_{HL}$, for which 
\begin{align}
\mathcal{E}^{\Delta_t, v'',v}_{HL}\lesssim \sum_{M>8}\sum_{k} \int_{\RR^2} |AQ|_k(\eta)  |k|^{\frac32}(Ap^{-\frac34})_k(\eta) \mathbbm{1}_{|k|\leq 16|\xi|} |\widehat{v''}(\xi)|_M| \cF((\de_v-t\de_z)\Psi)|_k(\eta-\xi)_{<M/8} .
\end{align}
Using \eqref{bd:ellAv0}-\eqref{bd:ellAv1}, \eqref{bd:p/pkk} and absorbing all factors of $|k|$ in the exponential, we have
\begin{align}
|k|^{\frac32}(p^{-\frac34}A)_k(\eta)\lesssim \frac{1}{1+|\frac{\xi}{k}-t|}\frac{1}{\l\xi\r^{\frac 12}} A^{v}(\xi) \e^{c\lambda|k,\eta-\xi|^s}.
\end{align}
We need to recover another half derivative to use the bounds available on $\l\de_v\r^{-1}v''$. Observe that, combining the bound above with \eqref{bd:japell}-\eqref{bd:japells} we have  
\begin{align}
|k|^{\frac32}(p^{-\frac34}A)_k(\eta)\lesssim \frac{1}{1+|\frac{\xi}{k}-t|}(t^{\frac12}\mathbbm{1}_{|\xi|\leq |kt|}+t^{\frac12-s}|\xi|^s\mathbbm{1}_{|\xi|\geq |kt|})\frac{A^v(\xi)}{\l\xi\r}  \e^{c\lambda|k,\eta-\xi|^s}.
\end{align}
To make use of the factor $(1+|t-\xi/k|)^{-1}$, we can consider the same splitting as in \eqref{eq:splitT0}. Hence, using \eqref{eq:ellipticproof-last}, in a similar way as in \eqref{eq:last-prop-elliptic2}, we finally obtain 
\begin{equation}
\mathcal{E}^{\Delta_t, v'',v}_{HL}\lesssim \sum_{j\in \{\lambda,w\}}\delta G_j[Q]+\delta^{-1}\eps^2\left(G_j^v[\jap{\de_v}^{-1} v'']+\jap{t}^{-2s}G_j^v[|\de_v|^s\jap{\de_v}^{-1} v'']\right),
\end{equation}
whence concluding the proof of \eqref{bd:EDeltat}.

\section{Bounds on the energy functional $E_n$}\label{sec:naturalenergy}
In this section, we aim at proving bounds on $E_n$, defined in \eqref{def:Omega}.

\subsection{The energy inequality}\label{app:energyIDVD}
The time-derivative of $E_n$ is computed in the following lemma.
\begin{lemma}
	\label{lemma:EIdEn}
For every $t\geq 0$ we have the energy inequality
\begin{align}
	\label{def:dtEn}
	\ddt E_n+\sum_{j\in\{\lambda,w,m\}}\left(G_j[\Omega]+\beta^2 G_j[\nabla_L\Theta]\right) \leq \frac12 \left\l \frac{\de_tp}{|k|p^\frac12}AQ,AQ\right\r+NL^{\Omega,\Theta}+\widetilde{\mathcal{E}}^{\div}+\widetilde{\mathcal{E}}^{\Delta_t}.
	\end{align}
where the $G_j[\cdot]$ are defined in \eqref{def:Gw} and the error terms are given by 
\begin{align}
	\label{def:NLVD}
	NL^{\Omega,\Theta}&=|\left\l [A, \bU]\cdot \nabla \Omega,A \Omega\right\r|+\beta^2\big|\left\l \cF([A p^\frac12,\bU]\cdot \nabla \Theta),A p^\frac12\hTheta\right\r\big|,\\
		\widetilde{\mathcal{E}}^{\div}&=\frac12 \big|\left\l \nabla \cdot \bU,|A \Omega|^2+\beta^2|A \nabla_L\Theta|^2\right\r\big|, \label{err-div-omega}\\
		\widetilde{\mathcal{E}}^{\Delta_t}&=\beta^2\big|\left\l Ak p^{-\frac12} \cF\left((\Delta_t-\Delta_L)\Psi\right),Ap^{\frac12}\hTheta\right\r\big|. \label{err-omega}			
\end{align} 
\end{lemma}

\begin{remark} 
The term in \eqref{def:dtEn} involving $Q$ can be bounded as
\begin{equation}
	\label{bd:LQ}
	\frac12 \left|\left\l \frac{\de_tp}{|k|p^\frac12}AQ,AQ\right\r\right| \leq \norm{AQ}^2\leq \frac{4\beta}{2\beta-1}E_L(t)\leq  \frac{32\beta}{2\beta-1}\eps^2,
\end{equation}
thanks to  the bootstrap hypothesis \eqref{boot:E} and the coercivity properties of $E_L$, analogous to \eqref{eq:coercive-pointwise}. In particular,  
the bound of order $\eps^2 t$ in \eqref{boot:impEn} cannot be improved in this setting. 
In addition, the weight $m$ has been introduced to control the term $L^{Z,Q}$ in \eqref{def:dtEm}, while here we exploit a direct control
in $Q$. Hence  the terms $G_m[\cdot]$ are superfluous. However, we decided not to introduce a further modification of the weight $A$, since this would not imply any significant simplifications.
\end{remark}

\begin{proof}[Proof of Lemma \ref{lemma:EIdEn}]
The computations for the time-derivative of the functional in \eqref{def:Omega} exploit some cancellations which are similar to those for the energy $E_L(t)$ in Section \ref{app:energyIDZQ}. 
The energy inequality \eqref{def:dtEn} is then obtained using \eqref{eq:Adt} and \eqref{def:LQEn}. 
\end{proof}

Now we control $NL^{\Omega,\Theta}$ in a similar way as Section \ref{sec:mainEn}. Thus, we define the transport nonlinearities as 
\begin{align*}
	\widetilde{T}_N&=|\jap{\mathcal{F}\left([A,\bU_{<N/8}]\cdot \nabla \Omega_N\right),A\Omega}|+\beta^2\big|\big\l \cF([A p^\frac12,\bU_{<N/8}]\cdot \nabla \Theta_N),A p^\frac12\hTheta\big\r\big|=	\widetilde{T}_N^{\Omega}+\widetilde{T}_N^{\Theta}.
\end{align*}
The reaction nonlinearities are given by 
\begin{align*}
	\widetilde{R}_N&=|\jap{\mathcal{F}\left([A,\bU_{N}]\cdot \nabla \Omega_{<N/8}\right),A\Omega}|+\beta^2\big|\big\l \cF([A p^\frac12,\bU_{N}]\cdot \nabla \Theta_{<N/8}),A p^\frac12\hTheta\big\r\big|
	=\widetilde{R}_N^{\Omega}+\widetilde{R}_N^{\Theta}.
\end{align*}
The remainder is 
\begin{align*}
	\widetilde{\mathcal{R}}&=\sum_{N\in \boldsymbol{D}}\sum_{N/8\leq N'\leq N}|\jap{\mathcal{F}\left([A,\bU_{N}]\cdot \nabla \Omega_{N'}\right),A\Omega}|+\beta^2\big|\big\l \cF([A p^\frac12,\bU_{N}]\cdot \nabla \Theta_{N'}),A p^\frac12\hTheta\big\r\big|.
\end{align*}
The main result of this section is the following proposition.
\begin{proposition}
\label{prop:recaperrEn}
Let $t\geq 1$ and $\beta>1/2$. Under the bootstrap hypotheses, 
\begin{align}
\label{bd:tilT} \sum_{N>8} \widetilde{T}_N&\lesssim  \sum_{j\in\{\lambda,w\}}\delta\big(G_j[\Omega]+G_j[\nabla_L\Theta]\big)+\frac{\eps^2}{t^{\frac32}},\\
\notag \sum_{N>8} \widetilde{R}_N&\lesssim  \sum_{j\in\{\lambda,w\}}\eps t^{\frac32}G_j[Z]+\delta G_j[\Omega]+\delta G_j[\nabla_L\Theta]+\delta t \big(G^{\eps}_{elliptic}+\eps^2 \jap{t}^{-2s}G_\lambda^v[|\de_v|^s\jap{\de_v}^{-1}v'']\big)\\
\label{bd:tilR} &\quad +\delta G^{\delta}_{elliptic} +\delta G_\lambda[\l\de_v\r^{-s}\mathcal{H}]+ \eps^2G^v_{\lambda}[h]+\delta\eps^2,\\
\label{bd:tilRem} \widetilde{\mathcal{R}}&\lesssim \delta\frac{\eps^2}{t},\\
\label{bd:tilEdiv} \widetilde{\mathcal{E}}^{\div}&\lesssim\frac{\eps^3}{t},\\
\label{bd:tilEDt} \widetilde{\mathcal{E}}^{\Delta_t}&\lesssim  \delta\sum_{j\in\{\lambda,w\}}(G_j[\nabla_L\Theta]+ G_{j}^v[1-(v')^2]+ G_j^v[\l \de_v\r^{-1}v'']+\jap{t}^{-2s} G_j^v[|\de_v|^s\l \de_v\r^{-1}v''])\notag\\
&\quad +\delta G^{\delta}_{elliptic}+\delta \eps^2.
\end{align}
\end{proposition}
Having at hand Proposition \ref{prop:recaperrEn}, we first show that \eqref{boot:impEn} holds true.
\begin{proof}[Proof of \eqref{boot:impEn}]
Observe that from \eqref{boot:E}-\eqref{boot:En}, \eqref{bd:Avscoeff} and \eqref{bd:intGell} one has
\begin{align}
\notag \int_1^t\sum_{j\in\{\lambda,w\}}&\eps \tau^{\frac32}G_j[Z]+\delta \tau \big(G^{\eps}_{elliptic}+\eps^2 \jap{\tau}^{-2s}G_\lambda^v[|\de_v|^s\jap{\de_v}^{-1}v'']\big) \dd \tau\lesssim\eps^3t^{\frac32}+\delta \eps^2 t+\delta^3\eps^2t\lesssim \delta \eps^2 t.
\end{align}
Also all the remaining terms on the right-hand side of \eqref{bd:tilT}-\eqref{bd:tilEDt} are at most of size $\delta \eps^2 t$ when integrated on $[1,t]$. Hence, combining \eqref{def:dtEn} with \eqref{bd:LQ}, Proposition \ref{prop:recaperrEn} and the local well-posedness Proposition \ref{lem:loc}, choosing $\delta$ sufficiently small the bound \eqref{boot:impEn} is proved.
\end{proof}

In the following subsections we show the proof of \eqref{bd:tilT}-\eqref{bd:tilRem}. For some term we can directly deduce the bounds from the one given in Section \ref{sec:mainEn}. In these cases, we will only highlight the arguments that need to be used. Moreover, \eqref{bd:tilRem}-\eqref{bd:tilEDt} are analogous to those in Section \ref{sub:remerr} and therefore omitted.

\subsection{Transport nonlinearities}
\label{sub:TNLEn}
To handle these terms, we have to exploit the commutation properties of the weights involved. Most of the bounds that we need are already done in Sections \ref{subsec:TNOmega}-\ref{subsec:TTheta}.

\subsubsection{Bound on $\widetilde{T}_N^\Omega$}
In view of \eqref{eq:trivcomm}, we rewrite this term as
\begin{align}
	\widetilde{T}_N^{\Omega}&\leq \big|\jap{\mathcal{F}m^{-1}\left([mA,\bU_{<N/8}]\cdot \nabla \Omega_N\right),A\Omega}\big|+\big|\jap{\mathcal{F}\left([m^{-1},\bU_{<N/8}]\cdot \nabla (mA\Omega)_N\right),A\Omega}\big|\notag\\
	&=\widetilde{T}^{\Omega,A}_N+\widetilde{T}^{\Omega,m}_N.
\end{align}
The term $\widetilde{T}^{\Omega,m}_N$ can be controlled exactly as done in Section \ref{subsec:TNOmega} to get 
\begin{equation}
	\sum_{N>8}\widetilde{T}^{\Omega,m}_N\lesssim \delta (G_w[\Omega]+G_\lambda[\Omega]).
\end{equation}
Regarding the bound on $\widetilde{T}^{\Omega,A}_N$, we do not have the factor $(p_k(\eta)/p_\ell(\xi))^{1/4}$ as for ${T}^{\Omega,A}_N$ in Section \ref{subsec:TNOmega}, meaning that we never lose a power $t^\frac12$. One can split this term as done in \eqref{def:splitTNOm} and proceed in the same way. The only difference is for the term $\widetilde{\mathcal{T}}_N^J\mathbbm{1}_D$, defined as in \eqref{def:TNJ}-\eqref{def:D}, where we need to use \eqref{bd:JexTD} instead of \eqref{bd:JexTDs}. In particular, the following inequality holds 
\begin{equation}
	\sum_{N>8}\widetilde{T}^{\Omega,A}_N\lesssim \delta G_w[\Omega]+\eps G_\lambda[\Omega]+\frac{\eps^2}{t^{\frac32}}.
\end{equation}
Thus, the two bounds above are consistent with \eqref{bd:tilT}.

\subsubsection{Bound on $\widetilde{T}^{\Theta}_N$}
Similarly to what was done in Section \ref{subsec:TTheta}, we rewrite $\widetilde{T}^{\Theta}_N$ as
\begin{align}
\notag \widetilde{T}_N^{\Theta}&\lesssim \big|\big\l \cF(m^{-1}p^{\frac12}[mA,\bU_{<N/8}]\cdot \nabla \Theta_N),A p^\frac12\hTheta\big\r\big|+\big|\big\l \cF(m^{-1}[p^{\frac12},\bU_{<N/8}]\cdot \nabla (mA\Theta)_N),A p^\frac12\hTheta\big\r\big|\notag\\
&\quad+\big|\big\l \cF([m^{-1},\bU_{<N/8}]\cdot \nabla (mAp^{\frac12}\Theta)_N),A p^\frac12\hTheta\big\r\big|\notag\\
&=\widetilde{T}^{\Theta,A}_N+\widetilde{T}^{\Theta,p}_N+\widetilde{T}^{\Theta,m}_N.
\end{align}
Also in this case the bounds on $\widetilde{T}^{\Theta,A}_N$ and $\widetilde{T}^{\Theta,m}_N$ can be done in the same way as $\widetilde{T}^{\Theta,A}_N$ and $\widetilde{T}^{\Theta,m}_N$ and one obtain 
\begin{equation}
		\sum_{N>8}\widetilde{T}^{\Theta,m}_N+\widetilde{T}^{\Theta,A}_N\lesssim \delta G_w[\nabla_L\Theta]+\eps G_\lambda[\nabla_L\Theta]+\frac{\eps^2}{t^{\frac32}}.
\end{equation}
It thus remain to control $\widetilde{T}^{\Theta,p}_N$.

\bullpar{Bound on $\widetilde{T}_N^{\Theta,p}$}
For the control of $\widetilde{T}^{\Theta,p}_N$ we need to present some technical differences in more detail.
In analogy to what was done for the term $T_N^{\Omega,p}$, we  have
$	\widetilde{T}_N^{\Theta,p}\lesssim \skli \widetilde{\mathcal{T}}_N^{p,1}+\widetilde{\mathcal{T}}_N^{p,2},$
where we define
\begin{align*}
	\widetilde{\mathcal{T}}_N^{p,1}&=  \frac{|p_k(\xi)^{\frac{1}{2}} - p_\ell(\xi)^{\frac{1}{2}}|}{p_\ell(\xi)^\frac12} |\ell,\xi||\widehat{\bU}|_{k-\ell}(\eta-\xi)_{<N/8}|Ap^{\frac12}\widehat{\Theta}|_{\ell}(\xi)_N |Ap^{\frac12}\widehat{\Theta}|_k(\eta),\\
	\widetilde{\mathcal{T}}_N^{p,2}&=  \frac{|p_k(\eta)^{\frac{1}{2}} - p_k(\xi)^{\frac{1}{2}}|}{p_\ell(\xi)^\frac12} |\ell,\xi||\widehat{\bU}|_{k-\ell}(\eta-\xi)_{<N/8}|Ap^{\frac12}\widehat{\Theta}|_{\ell}(\xi)_N |Ap^{\frac12}\widehat{\Theta}|_k(\eta).
 \end{align*}
To control $\widetilde{\mathcal{T}}_N^{p,1}$, notice that
\begin{equation}
|p_k(\xi)^{\frac{1}{2}} - p_\ell(\xi)^{\frac{1}{2}}|=t|k-\ell|\frac{|(k+(kt-\xi))+(\ell+(\ell t-\xi))|}{p_k(\xi)^{\frac{1}{2}} + p_\ell(\xi)^{\frac{1}{2}}}\lesssim t|k-\ell|\frac{1}{|k|(1+|t-\frac{\xi}{k}|)}p^{\frac12}_{k}(\xi).
\end{equation}
Therefore, the analogous of \eqref{bd:p14kell} become 
\begin{equation}
	\label{bd:p12com}
	\frac{|p_k(\xi)^{\frac{1}{2}} - p_\ell(\xi)^{\frac{1}{2}}|}{p_\ell(\xi)^\frac12}|\ell,\xi|\lesssim t\mathbbm{1}_{k\neq \ell}\jap{\eta-\xi,k-\ell}^2\frac{|\xi|}{|k|}\frac{1}{(1+|\frac{\xi}{k}-t|)}.
\end{equation}
As in Section \ref{subsec:TNOmega}, also in this case the term $\widetilde{\mathcal{T}}_N^{p,2}$ is simpler and therefore omitted. To control $\widetilde{\mathcal{T}}_N^{p,1}$,  for intermediate and long times we can argue as done for $T^{p,1}_N$. Instead, for the short times, if $t\in I_{k,\xi}$ we use $|\xi/k|\lesssim t^{1-2s}|k,\eta|^{s/2}|\ell,\xi|^{s/2}$. When $t\in I_{k,\xi}^c$ we argue as in \eqref{bd:wow}. Therefore, appealing to \eqref{boot:En}, we conclude that 
\begin{align}
	\sum_{N>8}\skli \widetilde{\mathcal{T}}^{p,1}_N\lesssim \delta (G_w[\nabla_L\Theta]+G_\lambda[\nabla_L\Theta]).
\end{align}
This is consistent with \eqref{bd:tilT}.

\subsection{Reaction nonlinearities} 
\label{sub:RNLEn}We now turn our attention to the control of $\widetilde{R}_N$ in \eqref{bd:tilR}.
\subsubsection{Bound on $\widetilde{R}^\Omega_N$}
We rewrite this term as follows
\begin{align}
	|\widetilde{R}_N^{\Omega}|&\leq   \bigg|\jap{A\left(\frac{p}{k^2}\right)^{\frac14}\mathcal{F}(\nabla^\perp (\Psi_{\neq})_N\cdot \nabla \Omega_{<N/8}),AZ}\bigg|+\bigg|\jap{A\left(\frac{p}{k^2}\right)^{\frac14}\mathcal{F}(\dot{v}_N (\de_v \Omega)_{<N/8}),AZ}\bigg|\notag\\
	& \quad + \bigg|\jap{A\left(\frac{p}{k^2}\right)^{\frac14}\mathcal{F}((1-v')\nabla^\perp (\Psi_{\neq})_N\cdot \nabla \Omega_{<N/8}),AZ}\bigg|
	+\bigg|\jap{\mathcal{F}\left(\bU_N\cdot \nabla A\Omega_{<N/8}\right),A\Omega} \bigg|\notag\\
	& \quad+\bigg|\jap{A_0\mathcal{F}\left(\bU_N\cdot \nabla \Omega_{<N/8}\right),(A\Omega)_0} \bigg|\notag\\
	&=  \widetilde{R}^{\Omega}_{N,\Psi}+\widetilde{R}^{\Omega}_{N,\dot{v}}+\widetilde{R}^{\Omega}_{N,\delta}+\widetilde{R}^{\Omega}_{N,com}+\widetilde{R}^{\Omega}_{N,0}.
\end{align}
Notice that, besides the term $\widetilde{R}^{\Omega}_{N,0}$ and the factor $k$ in the first three terms, $\widetilde{R}^{\Omega}_N$ has the same structure of $R^{\Theta,1}_N$ studied in Section \ref{sub:RTheta1}. Hence, we will be able to directly recover most of  the bounds from the ones in Section \ref{sub:RTheta1} with the  change 
\begin{equation}
\label{changes}
	 (i\beta \Theta)_{<N/8}\to \Omega_{<N/8}, \qquad AQ\to  AZ.
\end{equation}

\bullpar{Bound on $\widetilde{R}^{\Omega}_{N,\Psi}$}
In view of \eqref{changes}, proceeding as done to obtain \eqref{bd:RZPsiRNR1}, \eqref{bd:RThetaNRR}, \eqref{bd:RThetaNRNR}, \eqref{bd:RThetaRR},  \eqref{bd:RThetaS1}, \eqref{bd:RThetaL1} and \eqref{bd:RThetaL2}, one has 
\begin{align}
\sum_{N>8}\skli \widetilde{R}_{N,\Psi}^{\Omega}	&\lesssim \eps t^{\frac32} (G_w[Z]+G_\lambda[Z]+G^\eps_{elliptic}).
\end{align}

\bullpar{Bound on the remaining terms} 
Proceeding as done for the term $R^{\Theta}_{N,\delta}$ with the change \eqref{changes}, thanks to \eqref{bd:RTdelta0}, \eqref{bd:RTdelta1} and \eqref{bd:RTdelta2} we have 
\begin{align}
	\sum_{N>8} \widetilde{R}^{\Omega}_{N,\delta} \lesssim & \ \delta\sum_{N>8} \widetilde{R}^{\Omega}_{N,\Psi}+t^{\frac12} \norm{AZ}\norm{A^vh}\norm{\Psi_{\neq}}_{\G^{\lambda,\sigma-5}}\norm{\Omega}_{\G^{\lambda,\sigma-5}} \notag\\
	&+t^{\frac12-s}\norm{|\nabla|^\frac{s}{2}AZ}\norm{|\de_v|^{\frac{s}{2}}A^vh}\norm{\Psi_{\neq}}_{\G^{\lambda,\sigma-5}}\norm{\Omega}_{\G^{\lambda,\sigma-5}}\notag \\
	\lesssim&\  \eps t^\frac32 \delta(G_w[Z]+G_\lambda[Z]+G^\eps_{elliptic})+\eps^4 + \eps^2(G_\lambda[Z]+G_\lambda^v[h]).
\end{align}
For $\widetilde{R}^{\Omega}_{N,\dot{v}}$, similarly to \eqref{bd:RTdotv0} and \eqref{bd:RTdotv1} we obtain 
\begin{align}
\sum_{N>8}\skli\widetilde{R}^{\Omega}_{N,\dot{v}}&\lesssim \left(t^{\frac12}\norm{|\de_v|^{\frac{s}{2}}\frac{A_0}{\jap{\de_v}^s}\dot{v}}+\norm{|\de_v|^{\frac{s}{2}}\frac{A_0}{\jap{\de_v}^s}\mathcal{H}}\right)\norm{|\nabla|^{\frac{s}{2}}AZ}\norm{\Omega}_{\G^{\lambda,\sigma-4}}\\
&\lesssim \delta t^{-(s+\frac12)} t^{1+s}\norm{|\nabla|^{\frac{s}{2}}AZ}\left(\norm{|\de_v|^{\frac{s}{2}}\frac{A_0}{\jap{\de_v}^s}\mathcal{H}}+\frac{\eps}{t^2}\right)\\
&\lesssim \delta (G_{\lambda}[Z]+G_{\lambda}[\jap{\de_v}^{-s}\mathcal{H}])+\delta\frac{\eps^2}{t^{\frac32}}.
\end{align}
Moreover, 
\begin{equation}
\sum_{N>8}\widetilde{R}^{\Omega}_{N,com}\lesssim \norm{A\Omega}^2\norm{\bU}_{H^{\sigma-6}}\lesssim \delta\frac{\eps^2}{t}.
\end{equation}
For the last term $\widetilde{R}^{\Omega}_{N,0}$, being $A_0$ always non-resonant, it is not difficult to show that we have a bound  consistent with \eqref{bd:tilR}.
The proof of the bounds for $\widetilde{R}^{\Omega}_{N,\Psi}$ is over. 
\subsubsection{Bound on $\widetilde{R}^\Theta_N$}
For this term we cannot directly reduce ourselves to one which we have already controlled. Therefore we need to present some more details. We rewrite this term as  $$\widetilde{R}_N^{\Theta}=\widetilde{R}^{\Theta}_{N,\Psi}+\widetilde{R}^{\Theta_0}_{N,\Psi}+\widetilde{R}^{\Theta}_{N,\delta}+\widetilde{R}^{\Theta}_{N,\dot{v}}+\widetilde{R}^{\Theta}_{N,com},$$ where we define
\begin{align}
	 \widetilde{R}^{\Theta}_{N,\Psi}&= \left|\jap{Ap^\frac12\mathcal{F}(\nabla^\perp (\Psi_{\neq})_N\cdot \nabla (\Theta_{\neq})_{<N/8}),Ap^{\frac12}\widehat{\Theta}}\right|,\\
	\widetilde{R}^{\Theta_0}_{N,\Psi}&=\left|\jap{Ap^\frac12\mathcal{F}(\de_z (\Psi_{\neq})_N\de_v (\Theta_0)_{<N/8}),Ap^{\frac12}\widehat{\Theta}}\right|,\\
	\widetilde{R}^{\Theta}_{N,\delta}&=\left|\jap{Ap^{\frac12}\mathcal{F}((1-v')\nabla^\perp (\Psi_{\neq})_N\cdot \nabla \Theta_{<N/8}),Ap^{\frac12}\widehat{\Theta}}\right|,\\
	\widetilde{R}^{\Theta}_{N,\dot{v}}&=\left|\jap{Ap^{\frac12}\mathcal{F}(\dot{v}_N \de_v\Theta)_{<N/8}),Ap^{\frac12}\widehat{\Theta}}\right|,\\
	\widetilde{R}^{\Theta}_{N,com}&=\left|\jap{\mathcal{F}\left(\bU_N\cdot \nabla (Ap^{\frac12}\Theta)_{<N/8}\right),Ap^\frac12 \widehat{\Theta}}\right|.
	\end{align}
We will proceed in a similar way to what was done for the term $R^{\Theta}_{N,\Psi}$ but now we have the factor $p^{\frac12}$ instead of $p^{\frac14}$. The bounds will be presented with fewer details with respect to Section \ref{sub:RTheta1}.

\bullpar{Bound on $	\widetilde{R}^\Theta_{N,\Psi}$}
With the notation introduced in \eqref{def:resnonresdecom}, we split the term as 
\begin{equation}
	\widetilde{R}_{N,\Psi}^\Theta=\skli \widetilde{\mathcal{R}}_{N,\Psi}^{\Theta,(R,R)}+\widetilde{\mathcal{R}}_{N,\Psi}^{\Theta,(NR,R)}+\widetilde{\mathcal{R}}_{N,\Psi}^{\Theta,(R,NR)}+\widetilde{\mathcal{R}}_{N,\Psi}^{\Theta,(NR,NR)}+\widetilde{\mathcal{R}}_{N,\Psi}^{\Theta,S}+\widetilde{\mathcal{R}}_{N,\Psi}^{\Theta,L}.
\end{equation}
We now control each term separately.

\diampar{Bound on $\widetilde{\mathcal{R}}_{N,\Psi}^{\Theta,(R,NR)}$} 
From \eqref{bd:Jexgen} and the fact that $|t-\eta/k|\lesssim |\eta/k^2|$, we get
\begin{align}
	(Ap^\frac12)_k(\eta)|\eta\ell-k\xi|\mathbbm{1}_{t\in I_{k,\eta}\cap I_{\ell,\xi}^c}\chi^I
	&\lesssim |\ell,\xi|\frac{|\eta|^\frac12}{|k|(1+|t-\frac{\eta}{k}|)^{\frac12}}|k|(1+|t-\frac{\eta}{k}|)A_\ell(\xi)\e^{c\lambda|k-\ell,\eta-\xi|^s}\mathbbm{1}_{t\in I_{\ell,\xi}^c}\notag\\
	&\lesssim t|k,\eta|^{\frac{s}{2}}|\ell,\xi|^{1-\frac{s}{2}}p^{-1}_{\ell}(\xi)(pA)_\ell(\xi)\e^{\lambda|k-\ell,\eta-\xi|^s}\mathbbm{1}_{t\in I_{\ell,\xi}^c}.
	\end{align}
This way, appealing to \eqref{bd:ketapc}, the elliptic estimate \eqref{bd:precellcontr} and \eqref{bd:Thneqlow}, we obtain
\begin{align}
	\sum_{N>8}\skli |\widetilde{\mathcal{R}}_{N,\Psi}^{\Theta,(R,NR)}|&\lesssim t^{1-s} \norm{|\nabla|^{\frac{s}{2}}A\nabla_L\Theta}\norm{\jap{\frac{\de_v}{t\de_z}}^{-1}\frac{|\nabla|^{\frac{s}{2}}}{\l t \r^s}A (-\Delta_L) \Psi_{\neq}}\norm{\Theta_{\neq}}_{\G^{\lambda,\sigma-4}}\notag\\
	&\lesssim \delta (G_w[\nabla_L\Theta]+G^\delta_{elliptic}),
\end{align}
in agreement with \eqref{bd:tilR}.

\diampar{Bound on $\widetilde{\mathcal{R}}_{N,\Psi}^{\Theta,(NR,R)}$}
This is the most dangerous term. Since $4|\ell|\leq |\xi|$, appealing to \eqref{bd:Jexgood}, \eqref{bd:p/p} (with the role of $(k,\eta)$ and $(\ell,\xi)$ switched) and the fact that $p^{-1/4}_k(\eta)\lesssim \jap{\eta/k}^{-1/2}$, we get 
\begin{align}
	(Ap^\frac12)_k(\eta)|\eta \ell-k\xi|&\mathbbm{1}_{t\in I_{k,\eta}^c\cap I_{\ell,\xi}}\chi^I=p^{-\frac14}_k(\eta)(Ap^{\frac34})_k(\eta)|\eta \ell-k\xi|\mathbbm{1}_{t\in I_{k,\eta}^c\cap I_{\ell,\xi}}\chi^I\notag\\
	&\lesssim |k,\eta| p^{-\frac14}_k(\eta)\frac{|\ell|(1+|t-\frac{\xi}{\ell}|)^\frac12}{|\xi|^\frac12} \frac{|\xi|^\frac32}{|\ell|^3(1+|t-\frac{\xi}{\ell}|)^\frac32}(\widetilde{A}p^\frac34)_\ell(\xi)\e^{c\lambda|k-\ell,\eta-\xi|^s}\notag\\
	&\lesssim  t^{\frac32}  \frac{\de_t w_\ell(\xi)}{w_\ell(\xi)} |\ell|^{\frac12}(\widetilde{A}p^\frac34)_\ell(\xi)\e^{\lambda|k-\ell,\eta-\xi|^s}.
\end{align}
Since $A_k(\eta)\lesssim \tA_k(\eta)$, see \eqref{bd:stupid}, combining the inequality above with \eqref{bd:wexfaway}, \eqref{bd:Thneqlow}  and \eqref{bd:ellipticGZw} we infer 
\begin{align}
	\notag\sum_{N>8}\skli |\widetilde{\mathcal{R}}_{N,\Psi}^{\Theta,(NR,R)}|&\lesssim t^\frac32\norm{\sqrt{\frac{\de_t w}{w}}\tA\nabla_L\Theta}\norm{\sqrt{\frac{\de_t w}{w}}\widetilde{A} (-\Delta_L)^\frac34|\de_z|^{\frac12}\Psi}\norm{\Theta_{\neq}}_{\G^{\lambda,\sigma-4}}\\
	&\lesssim \delta G_w[\nabla_L \Theta]+\delta tG^{\eps}_{elliptic},
\end{align}
that is consistent with \eqref{bd:tilR}.

\diampar{Bound on $\widetilde{\mathcal{R}}^{\Theta,(NR,NR)}_{N,\Psi}$}
We have to treat separately the case $|\xi|\leq |\ell t|$ and $|\xi|>|\ell t|$. In the former, 
from \eqref{bd:Jeximp}, \eqref{bd:p/p} and \eqref{bd:ketapct} we have 
\begin{align*}
	(Ap^{\frac12})_{k}(\eta)|\eta \ell-k\xi|&\mathbbm{1}_{t\in I_{k,\eta}^c\cap I_{\ell,\xi}^c}\mathbbm{1}_{|\xi|\leq |\ell t|}\chi^I\\
	&\lesssim\  \mathbbm{1}_{\{t\in I_{\ell,\xi}^c\}\cap\{|\xi|\leq |\ell t|\}}|k,\eta|^{\frac{s}{2}}|\ell,\xi|^{1-\frac{s}{2}} p^{-\frac12}_{\ell}(\xi) (pA)_\ell(\xi)\e^{c\lambda|k-\ell,\eta-\xi|^s}\\
	&\lesssim \ \mathbbm{1}_{|\xi|\leq |\ell t|} t^{1-2s}|k,\eta|^{\frac{s}{2}}\jap{{\frac{\xi}{\ell t}}}^{1-s}\jap{\frac{\xi}{\ell t}}^{-1}|\ell,\xi|^{\frac{s}{2}}(pA)_\ell(\xi)\e^{c\lambda|k-\ell,\eta-\xi|^s}.
\end{align*}
Since $|\xi|\leq |\ell t|$, using \eqref{bd:precellcontr} and \eqref{bd:Thneqlow} we deduce 
\begin{align*}
		\sum_{N>8}\skli |\widetilde{\mathcal{R}}_{N,\Psi}^{\Theta,(NR,NR)}|\mathbbm{1}_{|\xi|\leq |\ell t|}&\lesssim t^{1-s}\norm{|\nabla|^{\frac{s}{2}}A\nabla_L\Theta}\norm{\jap{\frac{\de_v}{t\de_z}}^{-1}\frac{|\nabla|^{\frac{s}{2}}}{\jap{t}^s}A (-\Delta_L)\Psi_{\neq}}\norm{\Theta_{\neq}}_{\G^{\lambda,\sigma-4}}\\
	&\lesssim \delta (G_\lambda[\nabla_L\Theta] +G^{\delta}_{elliptic}).
\end{align*}
If $|\xi|\geq |\ell t|$ we know $|\ell,\xi|^{1/2}\lesssim t^{1/2-s}|\ell,\xi|^s$. Hence, since  $\jap{\xi/\ell}^{1/2}p^{-1/4}_\ell(\xi)\lesssim 1$ for $t\in I_{\ell,\xi}^c$, from \eqref{bd:Jeximp} and \eqref{bd:p/p} we have 
\begin{align}
	\notag(Ap^{\frac12})_{k}(\eta)|\eta \ell-k\xi|&\mathbbm{1}_{\{t\in I_{k,\eta}^c\cap I_{\ell,\xi}^c\}\cap\{|\xi|\geq |\ell t|\}}\chi^I\\
	&\lesssim t^{\frac12-s}\mathbbm{1}_{t\in I_{\ell,\xi}^c}|\ell,\xi|^{s} \frac{|\ell,\xi|^{\frac{1}{2}}}{|\ell|^{\frac12}}p^{-\frac14}_{\ell}(\xi) |\ell|^{\frac12}(p^{\frac34}A)_\ell(\xi)\e^{c\lambda|k-\ell,\eta-\xi|^s}\\
	\label{bd:RtildeNRNR}&\lesssim t^{\frac12-s}|k,\eta|^{\frac{s}{2}}|\ell,\xi|^{\frac{s}{2}}|\ell|^{\frac12}(p^{\frac34}A)_\ell(\xi)\e^{c\lambda|k-\ell,\eta-\xi|^s}.
\end{align}
From \eqref{bd:RtildeNRNR}, \eqref{bd:ellipticGZw} and \eqref{bd:Thneqlow} we get 
\begin{align}
	\notag \sum_{N>8}\skli |\widetilde{\mathcal{R}}_{N,\Psi}^{\Theta,(NR,NR)}|\mathbbm{1}_{|\xi|\geq |\ell t|}&\lesssim t^{\frac12}\norm{|\nabla|^{\frac{s}{2}}A\nabla_L\Theta}\norm{\frac{|\nabla|^{\frac{s}{2}}}{\jap{t}^s}A (-\Delta_L)^\frac34|\de_z|^\frac12\Psi}\norm{\Theta_{\neq}}_{\G^{\lambda,\sigma-4}}\\
	\notag &\lesssim \eps t^{-\frac32}\norm{|\nabla|^{\frac{s}{2}}A\nabla_L\Theta}^2+\eps  t^{\frac32}\norm{\frac{|\nabla|^{\frac{s}{2}}}{\jap{t}^s}A (-\Delta_L)^\frac34|\de_z|^\frac12\Psi}^2\\
	&\lesssim \eps G_\lambda[\nabla_L\Theta] +\delta t G^{\eps}_{elliptic}.
\end{align}

\diampar{Bound on $\widetilde{\mathcal{R}}^{\Theta,(R,R)}_{N,\Psi}$}
Since $k\neq \ell$, we know that we can only have cases $(b)$ or $(c)$ in Lemma \ref{lemma:trichotomy}. The case $(c)$ is straightforward. In case $(b)$, we know that we can apply \eqref{bd:Jeximp}, so that from \eqref{bd:p/p} we get 
\begin{equation}
	(Ap^\frac12)_{k}(\eta)|\eta\ell-k\xi|\mathbbm{1}_{I_{k,\eta}\cap I_{\ell,\xi}}\lesssim \frac{|\xi|}{|\ell|}\frac{1}{1+|\frac{\xi}{\ell}-t|}(\widetilde{A}p)_{\ell}(\xi)\e^{c\lambda|k-\ell,\eta-\xi|^s}.
\end{equation}
Combining the inequality above with \eqref{bd:trivRR}, \eqref{bd:Thneqlow} and \eqref{bd:ellipticGZw} we conclude that 
\begin{align}
	\sum_{N>8}\skli |\widetilde{\mathcal{R}}_{N,\Psi}^{\Theta,(R,R)}|&\lesssim t\norm{\sqrt{\frac{\de_t w}{w}}\tA \nabla_L\Theta}\norm{\jap{\frac{\de_v}{t\de_z}}^{-1}\sqrt{\frac{\de_t w}{w}}\widetilde{A} (-\Delta_L)\Psi_{\neq}}\norm{\Theta_{\neq}}_{\G^{\lambda,\sigma-4}}\notag\\
	&\lesssim \delta (G_w[\nabla_L\Theta]+G^{\delta}_{elliptic}).
\end{align}

\diampar{Bound on $\widetilde{\mathcal{R}}^{\Theta,S}_{N,\Psi}$}
Applying \eqref{bd:Jeximp} and \eqref{bd:p/p} we have 
\begin{equation*}
	(Ap^{\frac12})_k(\eta)|\eta\ell-k\xi|\chi^S\lesssim \frac{|\ell,\xi|}{|\ell|^{\frac12}}p^{-\frac14}_\ell(\xi) |\ell|^{\frac12}(p^{\frac34}A)_{\ell}(\xi)\e^{c\lambda|k-\ell,\eta-\xi|^s}.
\end{equation*}
Then, since $t\leq \TS$, if $t\in I_{\ell,\xi}^c$ and $|\ell|\leq |\xi|$, we have 
\begin{equation*}
	\frac{|\ell,\xi|}{|\ell|^{\frac12}}p^{-\frac14}_\ell(\xi)\mathbbm{1}_{\{t\in I_{\ell,\xi}^c\}\cap\{|\ell|\leq |\xi|\}}\lesssim \frac{|\xi|}{|\ell|}\frac{1}{1+|\xi|^{\frac12}/|\ell|}\lesssim |\xi|^\frac12\lesssim t^{1-2s}|\ell,\xi|^{\frac{s}{2}}|k,\eta|^{\frac{s}{2}}.
\end{equation*}
If $|\xi|\leq |\ell|$ we simply use that $1\leq t^{-2s}|\ell,\xi|^{\frac{s}{2}}|k,\eta|^{\frac{s}{2}}$. When $t\in I_{\ell,\xi}$ then $|\ell,\xi| |\ell|^{-1/2}p^{-\frac14}_\ell(\xi)\lesssim |\xi/\ell|\lesssim t^{1-2s}|\xi|^s$. Therefore, using \eqref{bd:ellipticGZw} and \eqref{bd:Thneqlow} we obtain 
\begin{align}
	\notag\sum_{N>8}\skli \widetilde{\mathcal{R}}_{N,\Psi}^{\Theta,S}&\lesssim t^{1-s}\norm{|\nabla|^{\frac{s}{2}}A\nabla_L\Theta} \norm{\frac{|\nabla|^{\frac{s}{2}}}{\l t \r^s}A(-\Delta_L)^{\frac34}|\de_z|^{\frac12}\Psi}\norm{\Theta_{\neq}}_{\G^{\lambda,\sigma-4}}\notag \\
	&\lesssim\delta (G_\lambda[\nabla_L\Theta]+G^{\eps}_{elliptic}).
\end{align}

\diampar{Bound on $\widetilde{\mathcal{R}}^{\Theta,L}_{N,\Psi}$} 
As observed in \eqref{bd:longp}, using \eqref{bd:p/p} we can always recover factor of times from negative powers of $p$. In particular, we have
\begin{align}
	(Ap^{\frac12})_k(\eta)|\eta\ell-k\xi|\chi^L &\lesssim |\ell,\xi| \frac{1}{|\ell|(1+|\frac{\xi}{\ell}-t|)}(pA)_\ell(\xi)\e^{c\lambda |k-\ell,\eta-\xi|^s}\notag\\
	&\lesssim t^{-s}|k,\eta|^{\frac{s}{2}}|\ell,\xi|^{\frac{s}{2}}(pA)_\ell(\xi)\e^{c\lambda |k-\ell,\eta-\xi|^s}.
\end{align}
Since in this case we know $|\xi|\leq |\ell t|$, appealing to \eqref{bd:Thneqlow} and \eqref{bd:precellcontr} we get
\begin{align}
	\sum_{N>8}\skli |\widetilde{\mathcal{R}}^{\Theta,L}_{\Psi}|\mathbbm{1}_{t\in I_{\ell,\xi}}&\lesssim \frac{\delta}{t} \norm{|\nabla|^{\frac{s}{2}}A\nabla_L\Theta}\norm{\jap{\frac{\de_v}{t\de_z}}^{-1}\frac{|\nabla|^\frac{s}{2}}{\l t\r^s}(-\Delta_L)A\Psi_{\neq}}\notag\\
	&\lesssim \delta (G_\lambda[\nabla_L\Theta]+G^\delta_{elliptic}).
\end{align}

\bullpar{Bound on $\widetilde{{R}}^{\Theta_0}_{N,\Psi}$} 
Here we can argue as for the term $\mathcal{R}^{\Theta_0}_{N,\Psi}$. In particular, \eqref{bd:RTheta0first} is replaced by 
\begin{equation}
	(Ap^\frac12)_k(\eta)|k|\lesssim \frac{1}{1+|\frac{\xi}{k}-t|} (Ap)_k(\xi)\e^{c\lambda|k-\ell,\eta-\xi|^s}.
\end{equation}
Notice that in \eqref{bd:nice} we recover the factor $\jap{\xi/kt}^{-1}$, which is important to apply \eqref{bd:ellp34APsi}. Therefore,  
\begin{align}
	\sum_{N>8} \widetilde{{R}}^{\Theta_0}_{N,\Psi}\lesssim \delta (G_w[\nabla_L \Theta]+G_\lambda[\nabla_L\Theta]+G^{\delta}_{elliptic})+\delta \eps^2.
\end{align}

\bullpar{Bound on $\widetilde{{R}}^{\Theta}_{N,\delta}$}
Reasoning as in Section \ref{sec:ReactionEn}, rewrite this term as
\begin{equation*}
	\widetilde{{R}}^{\Theta}_{N,\delta}=\widetilde{{R}}^{\Theta}_{\delta,LH}+\widetilde{{R}}^{\Theta,z}_{\delta,HL}+\widetilde{{R}}^{\Theta,v}_{\delta,HL}+\widetilde{{R}}^{\Theta}_{\delta,HH}.
\end{equation*}
With the same arguments done for the term $R^{\Theta}_{N,\delta}$, one can prove 
\begin{equation}
	\sum_{N, M>8}\widetilde{R}^\Theta_{\delta,LH}+\widetilde{R}^{\Theta,z}_{\delta,HL}\lesssim \delta \sum_{N>8}\widetilde{R}^\Theta_{N,\Psi}+\widetilde{R}^{\Theta_0}_{N,\Psi}.
\end{equation}
We are then left with the term 
\begin{align}
	\label{def:RThetadeltav}\widetilde{R}^{\Theta,v}_{\delta,HL}\lesssim& \sum_{k,\ell}\int_{\mathbb{R}^3} |Ap^{\frac12}\widehat{\Theta}|_k(\eta)|Ap^{\frac12}|_k(\eta)\rho_N(\xi)_\ell\mathbbm{1}_{|\ell|\leq 16|\xi'|}|\widehat{h}(\xi')|_M\\
	&\qquad \qquad \times |\widehat{\nabla^\perp\Psi}_{\neq}|_\ell(\xi-\xi')_{<M/8}|\widehat{\nabla\Theta}_{k-\ell}(\eta-\xi)_{<N/8}|\dd\eta \dd\xi \dd\xi'\\
	=&\skli \widetilde{\mathcal{R}}^{\Theta,v}_{\delta,HL},
\end{align}
where $\rho_N$ is the cut-off of the paraproduct decomposition. Here, we have to be careful since $p^{1/2}$ can be a full derivative in $h$. Indeed, in this term we have to crucially exploit the bounds available on $|\de_v|^sh$. In particular, from  \eqref{bd:towerA} and \eqref{bd:p/p} we have
\begin{align*}
	(Ap^{\frac12})_k(\eta)&\lesssim p^{\frac12}_k(\xi')A^{v}(\xi')\e^{c|\ell,\xi'-\xi|^s+c\lambda|k-\ell,\eta-\xi|^s}\\
	&\lesssim (t\mathbbm{1}_{|\xi'|\leq t}+|\xi'|\mathbbm{1}_{|\xi'|\geq t})\jap{\ell}\jap{k-\ell}A^{v}(\xi')\e^{c|\ell,\xi'-\xi|^s+c\lambda|k-\ell,\eta-\xi|^s}.
\end{align*}
When $|\xi'|\leq t$, appealing to \eqref{bd:Thneqlow}, \eqref{bd:Thneqlow0} and \eqref{bd:Psilow} we get 
\begin{align*}
	\sum_{N>8}\skli \widetilde{\mathcal{R}}^{\Theta,v}_{\delta,HL}\mathbbm{1}_{|\xi'|\leq t}\lesssim t\norm{A\nabla_L\Theta}\norm{A^vh}\norm{\Psi_{\neq}}_{\G^{\lambda,\sigma-4}}\norm{\Theta}_{\G^{\lambda,\sigma-4}}\lesssim \eps^4t\lesssim \delta^2 \eps^2.
\end{align*}
If $|\xi'|\geq t$, since $s>1/2$ then $|\xi'|=|\xi'|^{1-s} |\xi'|^s\leq|\xi'|^s |\xi'|^{s/2}|k,\eta|^{s/2}$, so that 
\begin{align*}
	\sum_{N>8}\skli \widetilde{\mathcal{R}}^{\Theta,v}_{\delta,HL}\mathbbm{1}_{|\xi'|\geq t}&\lesssim \norm{A|\nabla|^{\frac{s}{2}}\nabla_L\Theta}\norm{A^v|\de_v|^{\frac{s}{2}}(|\de_v|^sh)}\norm{\Psi_{\neq}}_{\G^{\lambda,\sigma-4}}\norm{\Theta}_{\G^{\lambda,\sigma-4}}\\
	&\lesssim \frac{\delta^2}{t^{2-3s}}t^{-s}\norm{A|\nabla|^{\frac{s}{2}}\nabla_L\Theta}t^{-2s}\norm{A^v|\de_v|^{\frac{s}{2}}(|\de_v|^sh)}\\
	&\lesssim \delta^2 G_{\lambda}[\nabla_L\Theta]+\delta^2t^{-2s}G_\lambda^v[|\de_v|^sh],
\end{align*}
where in the last bound we used $s\leq 2/3$.

 \bullpar{Bound on $\widetilde{{R}}^{\Theta}_{N,\dot{v}}$}
The treatment of this term is similar to $R^{\Theta}_{N,\dot{v}}$. We split this term as 
\begin{align*}
\widetilde{R}^{\Theta}_{N,\dot{v}}&\lesssim \skli |Ap^{\frac12}\widehat{\Theta}|_{k}(\eta)(Ap^{\frac12})_{k}(\eta) |\widehat{\dot{v}}|(\xi)_N|\widehat{\de_v \Theta}|_{k}(\eta-\xi)_{<N/8}\dd \eta \dd \xi \\
	&=\skli\widetilde{\mathcal{R}}^{\Theta}_{N,\dot{v}}(\mathbbm{1}_{t\in I_{k,\eta}\cap I_{k,\xi}}+1-\mathbbm{1}_{t\in I_{k,\eta}\cap I_{k,\xi}})\chi^I.
\end{align*}
Here we do not always have $\Theta_{\neq}$ as in Section \ref{sub:RTheta1}, but this is insignificant for the bound we need.

 For $t\in I_{k,\eta}\cap  I_{k,\xi}$ and $\TS\leq t \leq \TL$, appealing to \eqref{bd:Jeximp} and, since $A_0$ is always non-resonant, using the definition of $w$ \eqref{def:w} and \eqref{bd:wNRexgen} we get  
\begin{align}
	(Ap^{\frac12})_k(\eta)\mathbbm{1}_{t\in I_{k,\eta}\cap I_{k,\xi}} \chi^I&\lesssim  |k|(1+|t-\frac{\eta}{k}|)\frac{A_k(\xi)}{A_0(\xi)}A_0(\xi) \mathbbm{1}_{t\in I_{k,\eta}\cap I_{k,\xi}} \chi^I \notag\\
	&\lesssim |k|(1+|t-\frac{\eta}{k}|)\frac{|\xi|^{\frac12}}{|k|(1+|t-\frac{\xi}{k}|)^\frac12}A_0(\xi)\e^{c\lambda|k,\eta-\xi|^s}\\
\label{bd:tildeRdotvRR}	&\lesssim \frac{|\eta|}{|k|}A_0(\xi)\e^{c\lambda|k,\eta-\xi|^s} \lesssim t|k,\eta|^{\frac{s}{2}}|\xi|^{\frac{s}{2}}\frac{A_0(\xi)}{\jap{\eta}^s}\e^{c\lambda|k,\eta-\xi|^s}.
\end{align}
Notice that for $t\in I_{k,\eta}\cap I_{k,\xi}$ we have $k\neq0$. Hence, given $q$ as in \eqref{def:dotlambda}, using \eqref{bd:Glambdadotv} we get
\begin{align}
	\sum_{N>8} \skli \widetilde{\mathcal{R}}^{\Theta}_{N,\dot{v}}\mathbbm{1}_{t\in I_{k,\eta}\cap I_{k,\xi}}\chi^I &\lesssim t\norm{|\nabla|^{\frac{s}{2}}A\nabla_L\Theta}\norm{|\de_v|^{\frac{s}{2}}\frac{A_0}{\jap{\de_v}^s}\dot{v}}\norm{\Theta_{\neq}}_{\G^{\lambda,\sigma-6}}\notag\\
\notag	&\lesssim \delta \left(t^{-2q}\norm{|\nabla|^{\frac{s}{2}}A\nabla_L\Theta}^2+\frac{1}{t^{2+2(s-q)}}t^{2+2s}\norm{|\de_v|^{\frac{s}{2}}\frac{A_0}{\jap{\de_v}^s}\dot{v}}^2\right)\\
\label{bd:Tildedotv}	&\lesssim\delta G_\lambda[\nabla_L\Theta]+\delta\left(t^{2+2s}G_\lambda[\jap{\de_v}^{-s}\mathcal{H}]+\frac{\eps^2}{t^{\frac32}}\right),
\end{align}
where in the last line we used $2+2(s-q)\geq 2q$. For the remaining term, we need to distinguish whether $k=0$ or not. If $k=0$, observe that
\begin{align}
	(Ap^{\frac12})_0(\eta) \lesssim |\xi| |k,\eta|^{\frac{s}{2}}|\xi|^{\frac{s}{2}}\frac{A_0(\xi)}{\jap{\eta}^s}.
\end{align}
Thus, appealing to \eqref{bd:dvdotv} and \eqref{bd:OmTh} we get
\begin{align}
	\sum_{N>8} \skli \widetilde{\mathcal{R}}^{\Theta}_{N,\dot{v}}(1-\mathbbm{1}_{t\in I_{k,\eta}\cap I_{k,\xi}})\chi^I\mathbbm{1}_{k=0} &\lesssim \norm{|\nabla|^{\frac{s}{2}}A\nabla_L\Theta}\norm{|\de_v|^{\frac{s}{2}}\frac{A_0}{\jap{\de_v}^s}\de_v\dot{v}}\norm{\de_v\Theta_0}_{\G^{\lambda,\sigma-6}}\notag\\
\notag	&\lesssim \frac{\delta}{t^{1+s}}\left( \norm{|\nabla|^{\frac{s}{2}}A\nabla_L\Theta}^2+t^{2+2s} \norm{|\de_v|^{\frac{s}{2}}\frac{A_0}{\jap{\de_v}^s}\de_v\dot{v}}^2\right)\\
\label{bd:tildeRdotv0}&\lesssim\delta G_\lambda[\nabla_L\Theta]+\delta t^{2+2s}G_\lambda[\jap{\de_v}^{-s}\de_v\dot{v}].
\end{align}
When $k\neq 0$, if $t\geq|\eta|$ then $p^{1/2}_k(\eta)\lesssim t\jap{k}$ and $A_k(\eta)\lesssim A_0(\eta)\e^{c\lambda|k,\eta-\xi|^s}$. Hence, we have the same inequality as in \eqref{bd:tildeRdotvRR} and we obtain the same bound in \eqref{bd:Tildedotv}. When $|\eta|\geq t$ we have $p^{1/2}_{k}(\eta)\lesssim |\eta|$ and $A_k(\eta)\lesssim A_0(\eta)\e^{c\lambda|k,\eta-\xi|^s}$, meaning that we can argue similarly to what is done to obtain \eqref{bd:tildeRdotv0}.

\bullpar{{Bound on} $\widetilde{\mathcal{R}}^{\Theta}_{N,com}$}
On the support of the integrand $|k-\ell,\eta-\xi|\lesssim |\ell,\xi|$. From Cauchy-Schwarz and Young's convolution inequality we get 
\begin{equation}
	\sum_{N>8} \widetilde{\mathcal{R}}^{\Theta}_{N,com}\lesssim \norm{A\nabla_L\Theta}^2 \norm{\bU}_{H^{\sigma-6}}\lesssim \frac{\eps^3}{t^\frac12}.
\end{equation}
This concludes the proof of \eqref{bd:tilR}.

\section{Bound on the energy function $E_v$}\label{sec:zero}
In this section, we first prove Lemma \ref{lemma:bdcoeff} and next we provide the coordinate system controls  \eqref{boot:impEv} and \eqref{boot:impvdot} which are stated in (part of) Proposition \ref{prop:bootimpr}. The following proof consists of three main steps.

\subsection{Proof of Lemma \ref{lemma:bdcoeff}}
First, notice that 
\begin{align*}
	&1-(v')^2=(v'-1)^2-2(v'-1)=h^2-2h,\\
	&v''=v'\de_v v'=\de_v(v'-1)+(v'-1)(\de_v(v'-1))=\de_v h+h\de_vh.
\end{align*}
This way, \eqref{bd:Avcoeff} and \eqref{bd:Avscoeff} follow from the algebra property of $A^v$ and the norms in the $G^v_j$ (see \cite{BM15}*{Lemma 3.8}).
To prove \eqref{bd:dvdotv}-\eqref{bd:Gwdotv}, first observe that
\begin{equation}
	\label{eq:dedotv}
	\de_v \dot{v}= \frac{1}{v'}v'\de_v\dot{v} =\mathcal{H}+\left(\frac{1}{v'}-1\right)\mathcal{H}.
\end{equation}
Then, from the bootstrap hypothesis and using that $A\leq A^v$ when $k=0$, we have
\begin{equation}
	\label{bd:v'-1}\norm{ A\left(\frac{1}{v'}-1\right)}=\norm{\sum_{n=1}^{\infty}A(h^n)}\leq \sum_{n=1}^{\infty}\norm{A^vh}^n \lesssim \eps t^{\frac12}. 
\end{equation}
Combining \eqref{eq:dedotv} with the bound above and \eqref{boot:Ev}, we obtain \eqref{bd:dvdotv}.

To prove \eqref{bd:Glambdadotv}-\eqref{bd:Gwdotv}, we cannot directly rely on the low-frequency estimates of $\mathcal{H}$, but we can use the decay properties of $\dot{v}$ in a lower regularity class (the bootstrap hypothesis \eqref{boot:vdot}).
To this end, let us denote $f_{\leq 1}(\eta)=f(\eta)\mathbbm{1}_{|\eta|\leq 1}$ and $f_{>1}=f-f_{\leq 1}$. Notice that $A_0(t,\eta)_{<1}\lesssim 1$ and, using \eqref{boot:vdot}, that
\begin{align*}
	\norm{ \frac{A_0}{\jap{\de_v}^s} |\de_v|^{\frac{s}{2}}\dot{v}}&\lesssim \norm{ \dot{v}_{<1}}_{L^2}+\norm{ |\de_v|^{\frac{s}{2}}\frac{A_0}{\jap{\de_v}^s} \de_v\dot{v}_{>1}} \lesssim \frac{\eps}{t^{2}}+\norm{ |\de_v|^{\frac{s}{2}}\frac{A_0}{\jap{\de_v}^s} \de_v\dot{v}}.
\end{align*}
This way \eqref{bd:Glambdadotv} follows from the above estimate, \eqref{eq:dedotv} and \eqref{bd:v'-1}.
The proof of \eqref{bd:Gwdotv} is analogous: just notice that there is no need of a low frequency analysis as $\de_tw=0$ if $|\eta|\leq 1/2$. The lemma is proved. \qed 
\subsection{Control of $h$}
To obtain \eqref{boot:impEv} we need the following two energy inequalities:
\begin{align}
&\frac12 \ddt \norm{A^v h}^2+\sum_{j\in\{\lambda,w\}}G_j^{v}[h]\leq-\jap{A^v h, A^v\dot{v} \de_v h} + \frac1t\left \l A^v h, A^v \Omega_0 \right \r,
\label{eq:hdteqn}
\end{align}
and
\begin{align}
\frac12 \ddt \left(\jap{t}^{-2s} \norm{A^v |\de_v|^s h}^2\right)+\jap{t}^{-2s}\sum_{j\in\{\lambda,w\}}G_j^{v}[|\de_v|^s h]
&\leq-\jap{t}^{-2s}\jap{A^v |\de_v|^sh, A^v|\de_v|^s(\dot{v} \de_v  h)}\notag \\ 
&\quad+\frac{1}{t}\jap{t}^{-2s} \left \l A^v|\de_v|^s h, A^v|\de_v|^s\Omega_0 \right \r,
\label{eq:hdteqn2}
\end{align}
which are directly derived using \eqref{eq:h}, the definition of $\cH$, and $G^v_j[\cdot]$  in \eqref{def:GwR}.
We only deal with the right-hand side of \eqref{eq:hdteqn2}, as \eqref{eq:hdteqn} is very similar, and in fact slightly simpler.
For the first term, 
notice that
\begin{align*}
-\jap{A^v|\de_v|^s h, A^v|\de_v|^s(\dot{v} \de_v h)}= \frac12\jap{\de_v\dot{v},|A^v|\de_v|^s h|^2}-\jap{A^v |\de_v|^sh, [A^v|\de_v|^s,\dot{v}] \de_v h}.
\end{align*}
For the first piece, we simply use Sobolev embeddings and the bootstrap assumptions \eqref{boot:Ev}-\eqref{boot:vdot} to obtain
\begin{align}\label{eq:divhterm}
\frac12\jap{\de_v\dot{v},|A^v|\de_v|^s h|^2}\lesssim \frac{\eps}{\jap{t}^2}\norm{A^v|\de_v|^s h}^2\lesssim \delta\frac{\eps^2}{\jap{t}^{3/2}} .
\end{align}
The treatment of the second piece is similar to what was done in Section \ref{sec:mainEn}. Namely, we write 
\begin{equation*}
\jap{A^v|\de_v|^s h, [A^v|\de_v|^s,\dot{v}] \de_v h}= \sum_{M>8}T^h_M+\sum_{M>8}R^h_M+\cR^h
\end{equation*}
where
\begin{align*}
T^h_M=\jap{A^v|\de_v|^sh, [A^v|\de_v|^s,\dot{v}_{<M/8}] \de_v h_M},\qquad R^h_M=\jap{A^v|\de_v|^sh, [A^v|\de_v|^s,\dot{v}_{M}] \de_v h_{<M/8}},
\end{align*}
\begin{align}\label{eq:remh}
\cR^h=\sum_{M\in \boldsymbol{D}}\sum_{M/8\leq M'\leq M} \jap{A^v|\de_v|^s h, [A^v|\de_v|^s,\dot{v}_{M}] \de_v h_{M'}}.
\end{align}
The transport nonlinearity $T^h_M$ can be treated as in Section \ref{subsec:TNOmega}, with the simplification that the $z$-frequency is always the same. The reaction
and remainder terms are analogous to those in \cite{BM15}*{Section 8}, since the weight $A^v$ has the same properties as $A^R$ in \cite{BM15}. 
The different  assumptions  \eqref{boot:Ev}-\eqref{boot:vdot}  give
\begin{equation}\label{eq:nonlinearh}
|\jap{A^v|\de_v|^s h, A^v|\de_v|^s(\dot{v} \de_v h)}|\lesssim \delta \left( \sum_{j\in\{\lambda,w\}}G_j^{v}[|\de_v|^s h] +  \jap{t}^{2+2s} G_j\left[\jap{\de_v}^{-s}\cH\right]+\frac{\eps^2}{\jap{t}^{3/2}}\right).
\end{equation}
We now turn to the second term in \eqref{eq:hdteqn2}. For this, we  write
\begin{align*}
 |\left \l A^v|\de_v|^s h, A^v|\de_v|^s\Omega_0 \right \r|\leq L^h_1+L^h_2,
\end{align*}
with
\begin{align*}
L^h_1
&=\sum_{k\neq0}\int_{\RR^2} A^v(\eta)|\eta|^s|\hh(\eta)| A^v(\eta)|\eta|^s |\hOmega_0(\eta)|  \mathbbm{1}_{t\in I_{k,\eta}}\chi^I,\\
L^h_2
&=\sum_{k\neq0}\int_{\RR^2} A^v(\eta)|\eta|^s|\hh(\eta)| A^v(\eta)|\eta|^s |\hOmega_0(\eta)| \left(1-\mathbbm{1}_{t\in I_{k,\eta}}\chi^I\right).
\end{align*}
For $L^h_1$, we use the definition of $A^v$ in \eqref{def:ARvintro} and $w^v$ in \eqref{def:wv} to get 
\begin{align}\label{eq:exchange01}
A^v(\eta) \lesssim \frac{|\eta|}{k^2} \frac{\de_t w_k(\eta)}{w_k(\eta)}  A_0(\eta).
\end{align}
Then
\begin{align*}
L^h_1
&\lesssim  \jap{t}^{1+s}\norm{\sqrt{\frac{\de_t w}{w}} A^v|\de_v |^s h}\norm{\sqrt{\frac{\de_t w}{w}} A \Omega}
\leq \frac{t}{4}  G_w^v[|\de_v |^s h] +    \frac{C_1}{4} \jap{t}^{1+2s} G_w[\Omega],
\end{align*}
where $C_1\geq 1$  is independent of $\eps,\delta$. Turning to $L^h_2$, as $A^v \lesssim A_0$ in the support of the integral, we have
\begin{align*}
L^h_2
&\lesssim \norm{A^v|\de_v|^{s+\frac{s}{2}} h}  \norm{A|\de_v|^{\frac{s}{2}} \Omega} \leq \frac{t}{4}  G_\lambda^v[|\de_v |^s h] +    \frac{C_1}{4} \jap{t}^{1+2s} G_\lambda[\Omega].
\end{align*}
As a consequence, we have from \eqref{eq:hdteqn2} that
\begin{align}\label{eq:linearh}
\frac 1 t  \l t \r^{-2s}  |\left \l A^v|\de_v|^s h, A^v|\de_v|^s\Omega_0 \right \r|\leq \frac{1}{4} \l t \r^{-2s}  \sum_{j\in\{\lambda,w\}}G_j^{v}[|\de_v|^s h] +    \frac{C_1}{4}  \sum_{j\in\{\lambda,w\}}G_j[\Omega].
\end{align}
Collecting \eqref{eq:nonlinearh}, \eqref{eq:linearh} and the analogous bounds for \eqref{eq:hdteqn}, we end the proof of the estimates on $h$ in \eqref{boot:impEv}.

\subsection{Control of $\cH$}
To complete the proof of  \eqref{boot:impEv}, we start from
 the energy inequality
\begin{align}
&\frac12\ddt\left(\langle t \rangle^{2+2s}  \norm{\dfrac{A}{\langle \de_v \rangle^s} \cH}^2\right)  
+\jap{t}^{2+2s}\sum_{j\in\{\lambda,w\}}G_j\left[\jap{\de_v}^{-s}\cH\right]\leq T^{\cH} + F, \label{eq:energy-id-zeromode1}
\end{align}
where $G_j[\cdot]$ is defined in \eqref{def:Gw}
and the transport and forcing terms are given respectively by 
\begin{align}
	T^{\cH}&=- \l t \r^{2+2s} \jap{\dfrac{A}{\langle \de_v \rangle^{s}} \cH, \dfrac{A}{\langle \de_v \rangle^{s}} \dot{v}\de_v \cH},\label{eq:Th}\\
	F&=-  \frac{\l t \r^{2+2s} }{t}\jap{\dfrac{A}{\l \de_v \r^{s}} \cH, \dfrac{A}{\l \de_v \r^{s}} \left( v' \nabla^\perp \Psi_{\neq}  \cdot \nabla \Omega_{\neq}\right)_0}.\label{eq:forcing}
\end{align} 
In this case $T^{\cH}$ in \eqref{eq:Th} is similar to the transport terms of Section \ref{sec:mainEn};  $F$ in  \eqref{eq:forcing}
describes  the nonlinear feedback of the non-zero frequencies onto the zero one.
Bounds on $T^{\cH}$ are obtained as for \eqref{eq:nonlinearh}, giving
\begin{align}\label{eq:TCH}
	T^{\cH}\lesssim
	 \delta \left( \jap{t}^{2+2s}\sum_{j\in\{\lambda,w\}}G_j\left[\jap{\de_v}^{-s}\cH\right]+\frac{\eps^2}{\jap{t}^{3/2}}\right).
\end{align} 
We focus on the forcing term, which contains $v'=1+(v'-1)$, so that we can write
$\displaystyle F=F_0+F^\delta, $
where
\begin{align*}
F^0&=-  \frac{\l t \r^{2+2s} }{t}\jap{\dfrac{A}{\l \de_v \r^{s}} \cH, \dfrac{A}{\l \de_v \r^{s}} \left(  \nabla^\perp \Psi_{\neq}  \cdot \nabla \Omega_{\neq}\right)_0}, \\
F^\delta&=-  \frac{\l t \r^{2+2s} }{t}\jap{\dfrac{A}{\l \de_v \r^{s}} \cH, \dfrac{A}{\l \de_v \r^{s}} \left( (v'-1) \nabla^\perp \Psi_{\neq}  \cdot \nabla \Omega_{\neq}\right)_0}.
\end{align*}
As argued in \cite{BM15}*{Section 8}, it is enough to consider $F^0$, treating in a separate way  low-high, high-low and remainder interactions, namely
$\displaystyle F^0=F^0_{LH}+F^0_{HL}+F^0_{\mathcal{R}},$
with
\begin{align*}
F^0_{LH}&=-  \frac{\l t \r^{2+2s} }{t}\sum_{M> 8}\sum_{k\neq 0}\jap{\dfrac{A}{\l \de_v \r^{s}} \cH, \dfrac{A}{\l \de_v \r^{s}}  (\nabla^\perp \Psi_{k} )_{<M/8} \cdot (\nabla \Omega_{-k})_M},\\
F^0_{HL}&=-  \frac{\l t \r^{2+2s} }{t}\sum_{M> 8}\sum_{k\neq 0}\jap{\dfrac{A}{\l \de_v \r^{s}} \cH, \dfrac{A}{\l \de_v \r^{s}}  (\nabla^\perp \Psi_{k} )_{M} \cdot (\nabla \Omega_{-k})_{<M/8}},\\
F^0_{\mathcal{R}}&=-  \frac{\l t \r^{2+2s} }{t}\sum_{M\in \boldsymbol{D}}\sum_{M/8\leq M'\leq M} \sum_{k\neq 0}\jap{\dfrac{A}{\l \de_v \r^{s}} \cH, \dfrac{A}{\l \de_v \r^{s}} (\nabla^\perp \Psi_{k} )_{M'} \cdot (\nabla \Omega_{-k})_M}.
\end{align*}
There are various similarities with \cite{BM15}*{Section 8} in the treatment of all the non-resonant contributions, as the weight $A$ in our case is comparable to that in \cite{BM15}. In particular, as in \cite{BM15}, appealing to \eqref{bd:ketapc} and the usual arguments for short and long times, taking the case of $F^0_{HL}$ one has
\begin{align}\label{eq:FCH1}
|F^{0,NR}_{HL}|&=\left|  \frac{\l t \r^{2+2s} }{t}\sum_{M> 8}\sum_{k\neq 0}\jap{(1-\mathbbm{1}_{t\in I_{k,\xi}}\chi^I)\dfrac{A}{\l \de_v \r^{s}} \cH, \dfrac{A}{\l \de_v \r^{s}}  (\nabla^\perp \Psi_{k} )_{M} \cdot (\nabla \Omega_{-k})_{<M/8}}\right|\notag\\
&\lesssim \delta \jap{t}^{2+2s} G_\lambda\left[\jap{\de_v}^{-s}\cH\right] +\delta \jap{t}^{8s-2} G_\lambda\left[\jap{\de_v}^{-s}\cH\right] +\delta G^\delta_{elliptic}\notag\\
&\lesssim \delta \jap{t}^{2+2s} G_\lambda\left[\jap{\de_v}^{-s}\cH\right] +\delta G^\delta_{elliptic},
\end{align}
provided $s\leq2/3$. This is not restrictive, as explained after \eqref{def:dotlambda}. For $t\in I_{k,\xi}$, our weight and the one in \cite{BM15} are different, so that a careful treatment is needed in the following.
\begin{align*}
|F^{0,R}_{HL}|&\lesssim \l t \r^{1+2s} \sum_{M> 8}\sum_{k\neq 0}\int_{\RR^2}\mathbbm{1}_{t\in I_{k,\xi}}\chi^I\dfrac{A_0(\eta)}{\jap{\eta}^{s}} |\widehat{\cH}(\eta)|  \dfrac{A_0(\eta)}{\jap{\eta}^{s}}  |k,\xi||\widehat{\Psi}_{k}(\xi)_{M}|  |\widehat{\nabla \Omega}_{-k}(\eta-\xi)_{<M/8}|.
\end{align*}
As $ t\in I_{k,\xi}$, we now exchange $A_0(\eta)$ with $A_k(\xi)$, by means of \eqref{bd:Jexgood}; using that  $|\eta|\approx|\xi|\approx |kt|$, we get
\begin{align*}
\dfrac{A_0(\eta)}{\jap{\eta}^{s}}  |k,\xi| 
\lesssim   \jap{t}^{\frac12-s} \frac{\de_tw_k(\xi)}{w_k(\xi)} \e^{c\lambda |\eta-\xi|^s} p_k^\frac34(\xi)|k|^{\frac12}A_k(\xi).
\end{align*}
 This implies that (see \eqref{bd:wexfgen}) we have
\begin{align}\label{eq:FCH2}
|F^{0,R}_{HL}|
&\lesssim \delta \l t \r^{\frac32+s} \left(\norm{\sqrt{\frac{\de_tw}{w}} \dfrac{A_0}{\jap{\de_v}^{s}} \cH } +\norm{\frac{|\de_v|^\frac{s}{2}}{\jap{t}^s} \dfrac{A_0}{\jap{\de_v}^{s}} \cH }  \right)
\norm{\sqrt{\frac{\de_tw}{w}}\tA (-\Delta_L)^\frac34 |\de_z|^\frac12 \Psi } \notag\\
&\lesssim  \delta \l t \r^{2+2s}\sum_{j\in\{\lambda,w\}}G_j\left[\jap{\de_v}^{-s}\cH\right] +\delta \jap{t} G^\eps_{elliptic}.
\end{align}
We now turn our attention to $F^{0}_{LH}$ where the regularity gap between $\cH$ and $\Omega$ is crucial to control $\nabla_{-k} \Omega_M$ at high frequencies.  
Since $A_0(\eta)\lesssim \e^{c\lambda|\eta-\xi|^s} A_k(\eta)$ and $1-s\leq s$ notice that 
\begin{align*}
\frac{A_0(\eta)}{\jap{\eta}^s}|k,\xi|\lesssim |k,\xi|^{1-s} A_k(\xi) \jap{k}\e^{c\lambda |\eta-\xi|^s}\lesssim |\eta|^{\frac{s}{2}}|k,\xi|^{\frac{s}{2}}A_k(\xi)\e^{c\lambda |k,\eta-\xi|^s}.
\end{align*}
This way, we have
\begin{align}\label{eq:FCH3}
	|F^0_{LH}|&\lesssim \jap{t}^{1+2s}\norm{\frac{A}{\jap{\de_v}^s}|\de_v|^{\frac{s}{2}}\cH}\norm{|\nabla|^{\frac{s}{2}}A\Omega}\norm{\Psi}_{\G^{\lambda,\sigma-4}}\lesssim \delta \jap{t}^{2s-1}\norm{\frac{A}{\jap{\de_v}^s}|\de_v|^{\frac{s}{2}}\cH}\norm{|\nabla|^{\frac{s}{2}}A\Omega}\notag\\
	&\lesssim \delta \left(\jap{t}^{2+2s}G_\lambda[\jap{\de_v}^{-s}\cH]+\jap{t}^{4s-4}\lVert|\nabla|^{\frac{s}{2}}A\Omega\rVert^2\right)\notag\\
	&\lesssim \delta\left(\jap{t}^{2+2s}G_\lambda[\jap{\de_v}^{-s}\cH]+G_\lambda[\Omega]\right),
\end{align}
where the last inequality holds for $s\leq 2/3$.   
It remains to treat $F^{0}_{\mathcal{R}}$. Arguing as in \cite{BM15}, we get
\begin{align}\label{eq:FCH4}
	|F^{0}_{\mathcal{R}}|\lesssim \jap{t}^{1+2s} \norm{\frac{A}{\jap{\de_v}^s}\cH}\norm{\Omega}_{\G^{\lambda,\sigma-1}}\norm{\Psi_{\neq}}_{\G^{\lambda,\sigma-4}}\lesssim \delta\frac{\eps^2}{\jap{t}^{1-s}}.
\end{align}
From \eqref{eq:TCH}, \eqref{eq:FCH1}, \eqref{eq:FCH2}, \eqref{eq:FCH3}, \eqref{eq:FCH4}, and integrating on $(1,t)$ 
we get precisely \eqref{boot:impEv}.

\subsection{Control of $\dot{v}$}
To prove \eqref{boot:impvdot}, we begin from the identity
\begin{align*}
\frac12\ddt\left(\jap{t}^{4}\norm{\dot{v}}_{\G^{\lambda(t),\sigma-6}}^2\right)
&=2\jap{t}^{2}t\norm{\dot{v}}_{\G^{\lambda(t),\sigma-6}}^2+\jap{t}^{4}\left(\dot{\lambda}(t)\norm{|\de_v|^{\frac{s}{2}}\dot{v}}_{\G^{\lambda(t),\sigma-6}}^2+\jap{\dot{v},\de_t\dot{v}}_{\G^{\lambda(t),\sigma-6}}\right).
\end{align*}
Using \eqref{eq:vdot}, we notice that
\begin{align*}
\jap{\dot{v},\de_t\dot{v}}_{\G^{\lambda(t),\sigma-6}}
&=-\frac{2}{t}\norm{\dot{v}}_{\G^{\lambda(t),\sigma-6}}^2-\jap{\dot{v},\dot{v}\de_v\dot{v}}_{\G^{\lambda(t),\sigma-6}}-\frac1t\jap{\dot{v},v'\left(\nabla^\perp \Psi_{\neq} \cdot \nabla U^x_{\neq}\right)_0}_{\G^{\lambda(t),\sigma-6}}.
\end{align*}
A simple computation leads to
\begin{align}
\frac12\ddt\left(\jap{t}^{4}\norm{\dot{v}}_{\G^{\lambda(t),\sigma-6}}^2\right)
\leq \cV_1+\cV_2,\label{eq:vdoteqn}
\end{align}
with
\begin{align*}
\cV_1=-\jap{t}^{4}\jap{\dot{v},\dot{v}\de_v\dot{v}}_{\G^{\lambda(t),\sigma-6}},\qquad \cV_2=-\frac{\jap{t}^{4}}{t}\jap{\dot{v},v'\left(\nabla^\perp \Psi_{\neq} \cdot \nabla U^x_{\neq}\right)_0}_{\G^{\lambda(t),\sigma-6}}.
\end{align*}
To control $\cV_1$, using \eqref{eq:dedotv}, the algebra property and \eqref{bd:v'-1}, it easy to see that
\begin{align}
\cV_1
\lesssim\eps\jap{t}^{\frac72-s}\norm{\dot{v}}_{\G^{\lambda(t),\sigma-6}}^2\label{eq:V1esti}.
\end{align}
Turning to $\cV_2$, we first notice that by \eqref{bd:OmTh}-\eqref{bd:Psilow} we have that
\begin{align}
	\cV_2
	&\lesssim \eps^2\jap{t}\norm{\dot{v}}_{\G^{\lambda(t),\sigma-6}}
	\lesssim  \eps \jap{t}^{\frac72-s}\norm{\dot{v}}_{\G^{\lambda(t),\sigma-6}}^2 + \eps^3 \jap{t}^{-\frac32+s}\label{eq:V2esti}.
\end{align}
Putting together \eqref{eq:V1esti} and \eqref{eq:V2esti} back into \eqref{eq:vdoteqn} and integrating on $(1,t)$ gives precisely \eqref{boot:impvdot}.

\section*{Acknowledgements}
JB was
supported by National Science Foundation CAREER grant DMS-1552826.
MCZ and MD acknowledge funding from the Royal Society through a University Research Fellowship (URF\textbackslash R1\textbackslash 191492). RB and MCZ are partially supported by the GNAMPA (INdAM).

\bibliographystyle{siam}
\bibliography{NonlinearBouss}

\end{document}